\documentclass[12pt, a4]{amsart}

\usepackage{latexsym,amsmath,amssymb,amsthm,mathrsfs,color}
\usepackage{float}
\usepackage{comment}
\usepackage[bbgreekl]{mathbbol}
\usepackage{bbm}
\usepackage{tikz-cd}
\usepackage[all]{xy}
\usepackage{stmaryrd}
\usepackage[colorlinks=true, citecolor=brown,
    filecolor=black,
   linkcolor=blue,
   urlcolor=black]{hyperref}

\usepackage[toc,page]{appendix}

\usepackage[shortalphabetic]{amsrefs}
 \makeatletter
\def\append@label@year@{%
    \safe@set\@tempcnta\bib@year
    \edef\bib@citeyear{\the\@tempcnta}%
    \ifnum\bib@citeyear>9
      \append@to@stem{%
          \ifx\bib@year\@empty
          \else
            \@xp\year@short \bib@citeyear \@nil
          \fi
      }%
    \fi
}
\makeatother

\let\oldtocsection=\tocsection
\renewcommand{\tocsection}[2]{\hspace{0em}\oldtocsection{#1}{#2}}

\usepackage[a4paper]{geometry}

\def\upddots{\mathinner{\mkern 1mu\raise 1pt \hbox{.}\mkern 2mu
\mkern 2mu \raise 4pt\hbox{.}\mkern 1mu \raise 7pt\vbox {\kern 7
pt\hbox{.}}} }

\numberwithin{equation}{section}

\begin{document}
\setlength{\unitlength}{2.5cm}

\newtheorem{thm}{Theorem}[section]
\newtheorem{lm}[thm]{Lemma}
\newtheorem{prop}[thm]{Proposition}
\newtheorem{cor}[thm]{Corollary}
\newtheorem{conj}[thm]{Conjecture}
\newtheorem{specu}[thm]{Speculation}

\theoremstyle{definition}
\newtheorem{dfn}[thm]{Definition}
\newtheorem{eg}[thm]{Example}
\newtheorem{rmk}[thm]{Remark}

\newcommand{\N}{\mathbbm{N}}
\newcommand{\R}{\mathbbm{R}}
\newcommand{\C}{\mathbbm{C}}
\newcommand{\Z}{\mathbbm{Z}}
\newcommand{\Q}{\mathbbm{Q}}

\newcommand{\Mp}{{\rm Mp}}
\newcommand{\Sp}{{\rm Sp}}
\newcommand{\GSp}{{\rm GSp}}
\newcommand{\GL}{{\rm GL}}
\newcommand{\PGL}{{\rm PGL}}
\newcommand{\SL}{{\rm SL}}
\newcommand{\SO}{{\rm SO}}
\newcommand{\Spin}{{\rm Spin}}
\newcommand{\GSpin}{{\rm GSpin}}
\newcommand{\Ind}{{\rm Ind}}
\newcommand{\Res}{{\rm Res}}
\newcommand{\Hom}{{\rm Hom}}
\newcommand{\End}{{\rm End}}
\newcommand{\msc}[1]{\mathscr{#1}}
\newcommand{\mfr}[1]{\mathfrak{#1}}
\newcommand{\mca}[1]{\mathcal{#1}}
\newcommand{\mbf}[1]{{\bf #1}}

\newcommand{\mbm}[1]{\mathbbm{#1}}

\newcommand{\into}{\hookrightarrow}
\newcommand{\onto}{\twoheadrightarrow}

\newcommand{\s}{\mathbf{s}}
\newcommand{\cc}{\mathbf{c}}
\newcommand{\bfa}{\mathbf{a}}
\newcommand{\id}{{\rm id}}
\newcommand{\g}{\mathbf{g}_{\psi^{-1}}}
\newcommand{\w}{\mathbbm{w}}
\newcommand{\Ftn}{{\sf Ftn}}
\newcommand{\p}{\mathbf{p}}
\newcommand{\bq}{\mathbf{q}}
\newcommand{\WD}{\text{WD}}
\newcommand{\W}{\text{W}}
\newcommand{\Wh}{{\rm Wh}}
\newcommand{\ggma}{\omega}
\newcommand{\sct}{\text{\rm sc}}
\newcommand{\Of}{\mca{O}^\digamma}
\newcommand{\gk}{c_{\sf gk}}
\newcommand{\Irr}{ {\rm Irr} }
\newcommand{\Irrg}{ {\rm Irr}_{\rm gen} }
\newcommand{\diag}{{\rm diag}}
\newcommand{\uchi}{ \underline{\chi} }
\newcommand{\Tr}{ {\rm Tr} }
\newcommand{\der}\de
\newcommand{\Stab}{{\rm Stab}}
\newcommand{\Ker}{{\rm Ker}}
\newcommand{\bfp}{\mathbf{p}}
\newcommand{\bfq}{\mathbf{q}}
\newcommand{\KP}{{\rm KP}}
\newcommand{\Sav}{{\rm Sav}}
\newcommand{\de}{{\rm der}}
\newcommand{\tnu}{{\tilde{\nu}}}
\newcommand{\lest}{\leqslant}
\newcommand{\gest}{\geqslant}
\newcommand{\tu}{\widetilde}
\newcommand{\tchi}{\tilde{\chi}}
\newcommand{\tomega}{\tilde{\omega}}
\newcommand{\Rep}{{\rm Rep}}
\newcommand{\A}{{\mbf A}}
\newcommand{\BDI}{{\rm Inv}_{\rm BD}}
\newcommand{\nn}{ n_{\pmb{\varkappa}} }

\newcommand{\cu}[1]{\textsc{\underline{#1}}}
\newcommand{\set}[1]{\left\{#1\right\}}
\newcommand{\ul}[1]{\underline{#1}}
\newcommand{\ol}[1]{\overline{#1}}
\newcommand{\wt}[1]{\overline{#1}}
\newcommand{\wtsf}[1]{\wt{\sf #1}}
\newcommand{\anga}[1]{{\left\langle #1 \right\rangle}}
\newcommand{\angb}[2]{{\left\langle #1, #2 \right\rangle}}
\newcommand{\wm}[1]{\wt{\mbf{#1}}}
\newcommand{\elt}[1]{\pmb{\big[} #1\pmb{\big]} }
\newcommand{\ceil}[1]{\left\lceil #1 \right\rceil}
\newcommand{\floor}[1]{\left\lfloor #1 \right\rfloor}
\newcommand{\val}[1]{\left| #1 \right|}
\newcommand{\bepsilon}{\overline{\epsilon}}
\newcommand{\HH}{\mca{H}}
\newcommand{\WF}{{\rm WF}}

\newcommand{\exc}{ {\rm exc} }

\newcommand{\motimes}{\text{\raisebox{0.25ex}{\scalebox{0.8}{$\bigotimes$}}}}

\makeatletter
\newcommand{\extp}{\@ifnextchar^\@extp{\@extp^{\,}}}
\def\@extp^#1{\mathop{\bigwedge\nolimits^{\!#1}}}
\makeatother

\title[Geometric wavefront sets of genuine Iwahori-spherical representations]{Geometric wavefront sets of genuine Iwahori-spherical representations}

\author{Fan Gao}

\author{Runze Wang}
\address{School of Mathematical Sciences, Zhejiang University, 866 Yuhangtang Road, Hangzhou, China 310058}
\email{{gaofan@zju.edu.cn ({\rm F. G.}),  wang\underline{\ }runze@zju.edu.cn ({\rm R. W.})}}

\date{}
\maketitle

\begin{abstract} 
For Iwahori-spherical genuine representations of central covers with positive real Satake parameters, we prove the upper bound inequality for their geometric wavefront sets, formulated for general genuine representations in an earlier work by Gao--Liu--Lo--Shahidi. Meanwhile, we show the equality is attained for covers of type $A$ groups and for some representations of covers of the exceptional groups. We also verify the equality for certain Iwahori-spherical representations occurring in regular unramified principal series; this uses and generalizes the earlier work of Karasiewicz--Okada--Wang on theta representations. Lastly, we determine the leading coefficients in the Harish-Chandra character expansion of a theta representation when its geometric wavefront set is of a special type.	
\end{abstract}

\tableofcontents

\section{Introduction}\label{Intro}
In this paper, we consider a $p$-adic degree-$n$ central covering group
$$\begin{tikzcd}
\mu_n \ar[r, hook] & \overline{G}^{(n)} \ar[r, two heads] & G
\end{tikzcd}$$
that arises from the Brylinski--Deligne framework. For every irreducible genuine representation $\pi \in \Irr_\iota(\wt{G}^{(n)})$, it is important to understand its wavefront set $\WF(\pi)$, or the coarser geometric wavefront set $\WF^{\rm geo}(\pi)$, which gives the Gelfand--Kirillov dimension of $\pi$. For a linear algebraic group $G:=\wt{G}^{(1)}$, there has been an abundance of literature on this problem. For a general cover $\wt{G}:=\wt{G}^{(n)}$, there exist much fewer results regarding the wavefront sets of its genuine representations, and one notable family of representations concerns the theta representations.

Indeed, for a linear algebraic group $G$ and $\pi \in \Irr(G)$, an upper bound conjecture for $\WF^{\rm geo}(\pi)$ is formulated independently  by  Ciubotaru--Kim \cite{CK24} and Hazeltine--Liu--Lo--Shahidi \cite{HLLS}. It uses the Aubert--Zelevinsky involution ${\rm AZ}$ on $\Irr(G)$,  the local Langlands correspondence and thus the $L$-parameter of ${\rm AZ}(\pi)$, and the Barbasch--Vogan duality map $d_{BV, G}$ on nilpotent orbits. Many results are proved either prior to these formulations, or are verified along this conjecture. Notably, for a large class of unipotent representations representations \cite{CMBO24, CMBO25, CMBO21}, for certain representations of classical groups \cite{JL16, JL25} and for tempered representations of  $G_2$ \cite{CK24}, this upper bound conjecture has been verified. We refer the readers to the work and references in \cite{HLLS, CK24} for extensive reviews on the relevant literature for linear algebraic groups.

This upper bound conjecture is generalized from $G$ to $\wt{G}$ in \cite{GLLS} and it takes a similar form as in the linear case: one postulates that for every $\pi\in \Irr_\iota(\wt{G})$, the inequality
\begin{equation} \label{F:M1}
\WF^{\rm geo}({\rm AZ}(\pi)) \lest d_{BV, G}^{(n)}(\mca{O}(\phi_{\pi}))
\end{equation}
holds; also, the equality holds at the level of L-packets for tempered parameter $\phi_\pi$. However, it is crucial to replace the linear $d_{BV,G}$ by the covering Barbasch--Vogan duality $d_{BV, G}^{(n)}$ in \eqref{E:dBV}, defined and studied in \cite{GLLS}. This conjecture certainly depends on the hypothetical LLC, which is far from being established in the covering setting. Some evidence of the above upper bound conjecture is also given, loco citato. In particular, it encompasses the case of theta representations investigated in \cite{KOW}, see \cite[\S 5.1]{GLLS}. As mentioned, the wavefront set problem of theta representations, both local and global, has been studied extensively as in \cites{BFrG2, BFrG, FG15, Kap17-1, FG18, FG20, KOW}.

\subsection{Main results}
The goal of this paper is of three fold, all related to \eqref{F:M1} above. We assume $p$ is sufficiently large, as explained in the beginning of \S \ref{Prelim}.

First, we consider genuine Iwahori-spherical $\pi \in \Irr_\iota(\wt{G})^I$. By the work of Savin \cite{Sav04}, naturally associated to $\pi$ is a Kazhdan--Lusztig--Reeder parameter $(s, u, \tau)$, or equivalently an enhanced L-parameter. As a consequence of Theorem \ref{T:main}, which is of independent interest, we prove the following result (cf. Corollary \ref{C:key}):

\begin{thm} \label{T:M01}
Assume $p$ is sufficiently large.
Let $\pi:=\pi(s, u, \tau) \in \Irr_\iota(\wt{G})^I$ be an Iwahori-spherical genuine representation with positive real Satake parameter $s$. Then the inequality
\begin{equation} \label{E:main1}
\WF^{\rm geo}({\rm AZ}(\pi)) \lest d_{BV, G}^{(n)}(\mca{O}(\phi_{\pi}))
\end{equation}
holds. Furthermore, if $\wt{G}$ is a cover of $G$ of type $A$, then the equality is attained. For covers of exceptional groups, the equality is also attained when $\tau=\mbm{1}$ and the orbit of $u$, when viewed as in $\mca{N}(\mbf{G}(\overline{F}))$, lies in ${\rm Im}(d_{BV,G}^{(n)})$. 
\end{thm}
For linear algebraic groups, one actually achieves equality in \eqref{E:main1} by the work of Ciubotaru--Mason-Brown--Okada \cite{CMBO21}. However, for covering groups, the above inequality may be strict, see Example \ref{E:c-eg}. We expect that the upper bound in \eqref{E:main1} is ``one-step" optimal, see Conjecture \ref{C:1}.

In view of this, it is thus important to understand when the equality in \eqref{F:M1} actually holds, even for genuine Iwahori-spherical representations with positive real Satake parameters. As mentioned, the equality holds for theta representations, and clearly also for covering Steinberg representations. These two are special examples of the constituents $\pi_S$ of a regular unramified principal series. Here we have $\pi_S \in {\rm JH}(I(\nu))$ with $\nu \in X\otimes \R$ being regular and $S \subseteq \Phi(\nu) \subseteq \Delta$. 
If we set $M_S \subseteq G$ to be the Levi subgroup associated with $S$, with corresponding standard parabolic subgroup denoted by $P_S$, then $\pi_S \in {\rm JH}({\rm Ind}_{\wt{P}_S}^{\wt{G}} \Theta(\wt{M}_S))$. 

It was speculated in \cite{GaTs} that the equality in \eqref{E:main1} is achieved for such $\pi_S$. Regarding this, we show the following result, which could be viewed as a generalization of that for theta representations.

\begin{thm}[Theorem \ref{T:main2}] \label{T:02}
Assume $p$ is sufficiently large.
Assume $S$ is not $\wt{G}^{(n)}$-autotomous (see Definition \ref{D:attm}). If $G$ is of exceptional type, we further assume that $\nn \lest e(S)$. Then the equality 
$$\WF^{\rm geo}(\pi_S) = d_{BV, G}^{(n)}(\mca{O}(\phi_{{\rm AZ}(\pi_S)}))$$
holds.
\end{thm}

In the last part of the paper, we consider the leading coefficient $c_{\Theta}(\mca{O})_\psi$ for a theta representation, where $\mca{O} \in \WF(\Theta)$. Here $c_{\Theta}(\mca{O})_\psi$ is identified with the dimension of certain degenerate Whittaker models of $\Theta$, following \cite{MW87, Pate}. A speculative formula regarding $c_\Theta(\mca{O})_\psi$ is given in \cite[Conjecture 2.5]{GaTs}, and is verified for $\wt{\GL}_r^{(n)}$ in \cite[Theorem 5.1]{GLT25}. Here, we consider general $\wt{G}^{(n)}$ but only when $\WF^{\rm geo}(\Theta)$ is the saturation of the regular orbit of a Levi subgroup. 
Some analysis shows that the formula for $c_\Theta(\mca{O})_\psi$ has deviations from the earlier speculations in \cite{GaTs} general. The result for this part is as follows.

\begin{thm}[Theorem \ref{T:cO}] \label{T:M3}
Assume $p$ is sufficiently large. Let $\wt{G}^{(n)}$ be a covering group of $\diamondsuit$-type (see Definition \ref{D:ds}). Assume that the geometric wavefront orbit $\mca{O}_{\rm Spr}( j_{W_\nu}^W \varepsilon )$ of $\Theta(\wt{G}^{(n)})$ is  the regular orbit of a Levi subgroup; we further assume that $\mca{O}_X^{k, n} = \mca{O}^{k, n}$ if $X \in \set{B, C}$ and that $\mca{O}_X^{2r, n}=(n^{2a+1},1)$ if $X=D$. Then one has
$$c_\Theta(\mca{O})_\psi = \angb{ j_{W_\nu}^W \varepsilon_{W_\nu} }{ \sigma_{\msc{X}_n} \otimes \varepsilon_W }_W$$
for every $\mca{O} \in {\rm WF}(\Theta)$.
\end{thm}

\subsection{Acknowledgement} 
We would like to thank Sandeep Verma for an enlightening discussion on a relevant topic. The work of both authors is partially supported by the National Key R\&D Program of China (No. 2022YFA1005300) and also by NSFC-12171422.

\section{Preliminaries}\label{Prelim}
Let $F$ be a finite extension of $\Q_p$ with discrete valuation $| \cdot |_F$.  Let $\mathfrak{o}_F \subset F$ be the ring of algebraic integers and let $\mathfrak{p}_F \subset \mfr{o}_F$ be the maximal ideal. Let $\kappa_F=\mathfrak{o}_F/\mathfrak{p}_F$ be the residue field with cardinality $q$. Fix a uniformizer $\varpi \in \mathfrak{p}_F$ with $|\varpi|_F=q^{-1}$. We fix an algebraic closure $\wt{F}$ of $F$. 

We assume that $p$ is sufficiently large. This is to ensure the validity of the hypotheses 2.10, 2.14 and 2.17 in \cite[Theorem 2.22]{KOW}, which is used in this paper. A lower bound of $p$ can be chosen to be a constant arising from the root datum of $\mbf{G}$ and the ramification index of $F$ over $\Q$. Let $n\in \N_{\gest 1}$. We assume $F^\times$ contains the full group $\mu_n$ of $n$-th roots of unity, and that $p\nmid n$. Thus, $\mu_n \subseteq \kappa_F^\times$.

\subsection{Covering groups} \label{SS:cov}
Let $ \mbf{G} $ be a connected linear reductive group defined over $\mfr{o}_F$. By abuse of notation, we also use $\mbf{G}$ to denote its fiber over $F$, and assume it is split over $F$. Let $ \mathbf{T} $ be a maximal split torus of $ \mathbf{G} $. Let $$ (X,\Phi,\Delta; \ Y, \Phi^\vee, \Delta^\vee) $$ be the root datum associated with $(\mathbf{G},\mathbf{T})$, where $ X $ is the character lattice, $ Y $ the cocharacter lattice of $\mbf{T}$, and $ \Delta \subset \Phi$ is a choice of simple roots from the set $\Phi$ of roots.
Let $N(\mathbf{T}) \subseteq \mathbf{G}$ be the normalizer of $\mathbf{T}$ in $\mathbf{G}$, and let $ W=N(\mathbf{T})/\mathbf{T} $ be the Weyl group of $ (\mathbf{G},\mathbf{T}) $. 

Let $ G=\mathbf{G}(F)$ and we use this font for the $F$-points of other algebraic groups as well. We consider a degree $n$ Brylinski--Deligne central covering (see \cite{BD01, Wei18a, GG18})
$$\begin{tikzcd}
\mu_n \ar[r, hook] & \overline{G} \ar[r, two heads] & G
\end{tikzcd}$$
of $ G $ by $ \mu_n $ which is associated with the pair $ (D, \pmb{\eta}) $, where
\begin{itemize}
    \item $ D:Y \times Y \rightarrow \Z $ is a bilinear form such that $ Q(y):=D(y,y) $ is a Weyl-invariant quadratic form
    \item $ \pmb{\eta}: Y^{\text{sc}} \rightarrow F^{\times} $  is a homomorphism of the coroot lattice $ Y^{\text{sc}} \subset Y $ into $ F^{\times} $.
\end{itemize}
In this paper, we assume $\pmb{\eta}$ is trivial. For every root $\alpha$, we denote $n_\alpha:=n/\gcd(n,Q(\alpha^\vee))$.
We also write
$$\nn:=\nn(\wt{G}):=n_\alpha$$
for any short coroot $\alpha^\vee$, where for simply-laced group every coroot is considered as short. We call a cover $\wt{G}$ ``primitive" if $n=\nn$. Many results for the cover $\wt{G}$ depend essentially only on $\nn$ instead of $n$, and this is the raison d'\^etre of $\nn$.

We fix $\iota: \mu_n  \into \C^\times$ and consider exclusively $\iota$-genuine representations of $\ol{G}$, i.e., when $\mu_n$ acts via $\iota$. Denote by $\Irr_\iota(\ol{G})$ the set of isomorphism classes of irreducible $\iota$-genuine representations of $\ol{G}$.

\subsection{Dual groups} Let $B_Q: Y \times Y \to \Z$ be the bilinear form associated with $Q$, i.e., $B_Q(y,z):=Q(y+z) - Q(y) - Q(z)$.
Define $$Y_{Q,n}:=\{y \in Y: B_Q(y,z) \in n\Z  \text{ for all } z \in Y\} $$ 
and $ X_{Q,n}:=\text{Hom}_\Z(Y_{Q,n},\Z) $. For every $ \alpha \in \Phi$ we set $$ \alpha_{Q,n}^{\vee}=n_{\alpha}\alpha^{\vee} \text{ and } \alpha_{Q,n}=n_{\alpha}^{-1}\alpha,$$
where $ n_{\alpha}=n/\text{gcd}(n,Q(\alpha^{\vee}))$ as mentioned. Denote $$ \Phi_{Q,n}^{\vee}=\{\alpha_{Q,n}^{\vee}: \alpha \in \Phi^{\vee}\} \text{ and } \Phi_{Q,n}=\{\alpha_{Q,n}: \alpha \in \Phi\}; $$ similarly for $ \Delta_{Q,n}^{\vee} $ and $ \Delta_{Q,n} $.

The root datum 
$$ (Y_{Q,n},\Phi_{Q,n}^{\vee},\Delta_{Q,n}^{\vee}; \ X_{Q,n}, \Phi_{Q,n}, \Delta_{Q,n}) $$
gives rise to a complex linear group $\wt{G}^\vee$, which is called the dual group of $\wt{G}$. Let $ \mathbf{G}_{Q,n} $ be the split connected linear algebraic group over $ F $ whose Langlands dual group $G_{Q,n}^\vee$ is isomorphic to $ \ol{G}^{\vee} $. Let $ \mathbf{T}_{Q,n}\subseteq \mathbf{G}_{Q,n}$ be the maximal split torus contained in a Borel subgroup $\mathbf{B}_{Q,n} \subseteq \mbf{G}_{Q,n}$.

\subsection{Several Iwahori--Hecke type algebras}
Fix a splitting 
$$s: \mathbf{G}(\mathfrak{o}_F) \into \wt{G},$$
the existence of which follows from our assumption that $\pmb{\eta}$ is trivial, see \cite[\S 4]{GG18}. Let 
$$I\subset \mathbf{G}(\mathfrak{o}_F)$$
be the Iwahori subgroup determined by $\Delta $ and  let $I_1 \subset I$
 be the unique maximal pro-$p$ normal subgroup in $I$. By abuse of notation, we also write $I$ and $I_1$ for $s(I)$ and $s(I_1)$ respectively, and thus view them as subgroups of $\ol{G}$. Let 
$$\mathcal{H}_{\bar{\iota}}(\ol{G},I_1):=C^{\infty}_{\bar{\iota},\text{c}}(I_1\backslash \ol{G}/I_1) $$ to be the genuine pro-$ p $ Iwahori--Hecke algebra consisting of $ \bar{\iota} $-genuine locally-constant compactly-supported functions $ f $ on $ \ol{G} $ which are $ I_1 $-biinvariant, i.e., $ f(\zeta \gamma_1 g \gamma_2)=\iota(\zeta)^{-1}f(g) $ for all $ \zeta \in \mu_n, g\in \wt{G}$ and $ \gamma_1,\gamma_2 \in I_1 $.

Let 
$$\chi_0:\ol{\mathbf{T}}(\mathfrak{o}_F)\longrightarrow \C^{\times}$$
be the genuine character trivial on $\mathbf{T}(\mathfrak{o}_F)$, which we recall is used to denote $s(\mbf{T}(\mfr{o}_F))$. Write
\begin{equation} \label{D:X}
\mathscr{X}:=Y/Y_{Q,n}
\end{equation}
 and choose a set 
$$\set{y_0,y_1,...,y_{|\mathscr{X}|-1} } \subseteq Y$$ 
of representatives of $\mathscr{X}$ with $y_0=0$. The group $Y$ acts on $\wt{\mbf{T}}(\mfr{o}_F)$ by 
$y\cdot \wt{t}:= y(\varpi^{-1}) \wt{t} y(\varpi^{-1})^{-1}$, and this induces an action of $Y$ on $\chi_0$. We write
\begin{equation} \label{E:chi-i}
\chi_i:=y_i\cdot \chi_0: \ol{\mbf{T}}(\mfr{o}_F) \longrightarrow \C^\times
\end{equation}
for the character arising from this action.

Naturally, each $\chi_i$ gives a character of $\ol{I}$ and thus also the $\chi_i^{-1}$-spherical twisted Iwahori--Hecke algebra
$$\mathcal{H}_i:=C^{\infty}_{\text{c}}((\ol{I},\chi_i^{-1})\backslash \ol{G}/(\ol{I},\chi_i^{-1})),$$which consists of locally-constant compactly-supported functions $f$ on $\ol{G}$ satisfying $f(\gamma_1g\gamma_2)=\chi^{-1}_i(\gamma_1)f(g)\chi^{-1}_i(\gamma_2)$ for $g \in \ol{G}$ and $\gamma_1,\gamma_2\in \ol{I}$.
Let $\mca{H}(G_{Q,n}, I_{Q,n}):=C_c^\infty(I_{Q,n}\backslash G_{Q,n}/I_{Q,n})$ be the Iwahori--Hecke algebra of $G_{Q,n}$, where $I_{Q,n}$ is the Iwahori subgroup of $G_{Q,n}$ associated with $\Delta_{Q,n}$.
By \cite{Wang24}, there is a natural algebra isomorphism 
\begin{equation} \label{eq:Psi-i}
    \Psi_i:\mathcal{H}_i \rightarrow \mathcal{H}(G_{Q,n}, I_{Q,n}).
\end{equation}
For $i=0$, the above isomorphism was already established earlier by Savin \cite{Sav04}. We note that the isomorphism \eqref{eq:Psi-i} depends on a choice of distinguished genuine character, which is discussed in \cite[\S 6.4]{GG18}.

Let $$\text{Irr}_{\iota}(\ol G)^I \subset \Irr_\iota(\ol{G})$$ be the subset containing those Iwahori-spherical $\pi \in \Irr_\iota(\ol{G})$ such that $\pi^I\neq 0$. It is well-known that $\pi\mapsto \pi^I$ gives a bijection from $\text{Irr}_{\iota}(\ol G)^I$ to $\text{Irr}(\mathcal{H}_0)$.

Similarly, let $\text{Irr}(G_{Q,n})^{I_{Q,n}}$ be the set containing the irreducible representation $\pi_{Q,n}$ of $G_{Q,n}$ such that $\pi_{Q,n}^{I_{Q,n}}\neq 0$. Then $\pi_{Q,n} \mapsto \pi_{Q,n}^{I_{Q,n}}$ induces a bijection from $\text{Irr}(G_{Q,n})^{I_{Q,n}}$ to $\text{Irr}(\mathcal{H}(G_{Q,n}, I_{Q,n}))$. 
Thus via $\Psi_0$, we have a natural bijection 
\begin{equation} \label{E:rep_cor}
  \Psi^{\star}_0: \text{Irr}_{\iota}(\ol G)^I \longrightarrow\text{Irr}(G_{Q,n})^{I_{Q,n}}.
\end{equation}

\subsection{Kazhdan--Lusztig--Reeder parameters} \label{SS:KLR}
Consider a pair $(s, u)$, where $s \in T_{Q,n}^\vee$ and $u \in G_{Q,n}^\vee$ is unipotent such that $s u s^{-1} = u^q$.
We have the component group
$$A_{s, u}:=\pi_0({\rm Stab}_{G_{Q,n}^\vee}(s, u)/Z(G_{Q,n}^\vee))$$
associated with $s, u$. Let $\mfr{B}$ be the variety of all Borel subgroups of $G_{Q,n}^\vee$, and let $\mfr{B}_{s, u} \subseteq \mfr{B}$ be the subvariety consisting of those Borel subgroups which contain both $s$ and $u$. Then ${\rm Stab}_{G_{Q,n}^\vee}(s, u)$ acts naturally on $H_*(\mfr{B}_{s, u}, \C)$ and the action factors through $A_{s, u}$. A representation $\tau \in \Irr(A_{s, u})$ is called geometric if it appears in $H_*(\mfr{B}_{s, u}, \C)$. A Kazhdan--Lusztig--Reeder parameter is a triple 
$$(s, u, \tau)$$
 with $\tau \in \Irr(A_{s, u})$ being geometric. All such triples are endowed with a $ G^{\vee}_{Q,n}$-action. Denote by ${\rm Par}^{\rm KL}_{\rm geo}(G_{Q,n}^\vee)$ the set of $G^\vee_{Q,n}$-orbits of Kazhdan--Lusztig--Reeder triples.
One has the local Langlands correspondence
\begin{equation} \label{eq:LLC}
     \mathcal{L}^{\rm KL}_{G_{Q,n}}: \text{Irr}(\mathcal{H}(G_{Q,n}, I_{Q,n})) \rightarrow {\rm Par}^{\rm KL}_{\rm geo}(G_{Q,n}^\vee),
\end{equation}
as given in \cite{KL2,Ree4,ABPS17}. 

Since $\ol{G}^{\vee}=G_{Q,n}^{\vee}$, we also view ${\rm Par}^{\rm KL}_{\rm geo}(G_{Q,n}^\vee)$ as ${\rm Par}^{\rm KL}_{\rm geo}(\ol{G}^\vee)$. 
Let $\pi \in \Irr_{\iota}(\ol{G})^I$ be Iwahori-spherical. Then $\pi^{I}$ is an irreducible representation of $\mathcal{H}_0$ and one can attach a Kazhdan--Lusztig--Reeder parameter $(s, u, \tau)$ to $\pi$ by using \eqref{E:rep_cor} and \eqref{eq:LLC}. Thus, we write
$$\pi(s, u, \tau) \in \Irr_\iota(\wt{G})^I$$
for $\pi$ to highlight its parameters.

\subsection{Aubert--Zelevinsky duality}
Let $R(\ol{G})$ be the Grothendieck group of smooth $\iota$-genuine representations of $\wt{G}$. The Aubert--Zelevinsky duality $\text{AZ}:R(\ol{G})\rightarrow R(\ol{G})$ (see \cite{Zel80, Aub95}) is given by 
\begin{equation} \label{E:AZdef}
    \text{AZ}(X)=\sum_{\ol{L}\text{ Levi}\atop \ol{B}\subset \ol{L}}(-1)^{\text{rank}({\ol{L}})}\text{ind}^{\ol{G}}_{\ol{L}}(\text{res}^{\ol{G}}_{\ol{L}}(X)).
\end{equation}
Here ${\rm res}_{\ol{L}}^{\ol{G}}$ is the Jacquet-module map and ${\rm ind}_{\ol{L}}^{\ol{G}}$ is the normalized parabolic induction.
For $\pi \in \Irr_\iota(\wt{G})^I$, one has $\text{AZ}(\pi) \in \Irr_\iota(\wt{G})^I$. Also, the bijection $\Psi_0^\star$ in \eqref{E:rep_cor} is ${\rm AZ}$-equivariant.


\subsection{Local character expansion and wavefront sets}
Let $\mathfrak{g}$ be the Lie algebra of $G$. 
The group $G$ acts naturally on the set of nilpotent elements of $\mfr{g}$, and we write $\mathcal N(G)$ for the set of orbits under this action. 

For each nilpotent orbit $\mathcal O\in \mathcal N(G)$, let 
$$I_{\mathcal O}:C_c^\infty(\mathfrak g)\to \C$$
be the associated orbital integral. Every $\pi \in \Irr_\iota(\ol{G})$ defines a character distribution $\chi_\pi$ in a neighborhood of $0$ in $\mathfrak{g}$. There exists a compact open subset $U_\pi \subseteq \mfr{g}$ of 0 such that for all $f \in C^{\infty}_{c}(U_\pi)$ we have 
$$\chi_\pi(f) = \sum_{\mathcal O\in \mathcal N(G)}c_\pi(\mca{O})_\psi \cdot \hat I_{\mca{O}, \psi}(f)$$
with $c_\pi(\mca{O})_\psi \in \C$, see \cite{How1, HC99, Li3}. Here we have fixed a non-trivial additive character $\psi$ of $F$, and $\hat{I}_{\mca{O}, \psi}$ denotes the Fourier transform of $I_\mca{O}$ using this character.

The wavefront set of $\pi$ is defined to be 
$${\rm WF}(\pi) = \max\set{\mca{O} \in \mca{N}(G): c_\pi(\mca{O})_\psi \ne 0}$$
where the maximum is taken with respect to the topological closure ordering. Geometric nilpotent orbits are orbits under the action of $\mathbf{G}(\wt{F})$. The geometric wavefront set 
$$\rm WF^{\rm geo}(\pi)$$ is defined to be the set of maximal geometric orbits (with respect to the Zariski closure) that contain a nilpotent orbit $\mathcal O \in \WF(\pi)$.
It is well-known that the nilpotent orbits of a connected reductive group over either $\wt{F}$ or $\C$ are both classified by their weighted Dynkin diagrams. Thus, for a geometric nilpotent orbit over $\wt{F}$, we identify it with the corresponding nilpotent orbit over $\C$. 

\subsection{Truncated induction and Springer correspondence}
We first recall briefly the notion of $j$-induction, see \cite[\S 11.2]{Car} for details.

The Weyl group $W$ acts on the space $V=Y\otimes \R$. Let $W' \subseteq W$ be a reflection subgroup. We can decompose 
$$V=V^{W'}\oplus V^{\prime}$$
as $W'$-modules.  Let 
$$\text{Sym}^{i}(V^{\prime})$$
 be the space of homogeneous polynomial functions on $V^{\prime}$ of degree $i$. Let $E' \in \Irr(W')$ be an irreducible representation of $W^{\prime}$. Assume that there exists an integer $e$ such that $E^\prime$ occurs in  $\text{Sym}^{e}(V^{\prime})$ with multiplicity $1$ and does not occur in $\text{Sym}^{i}(V^{\prime})$ for $0 \lest i <e$. Now we regard $E'$ as a subspace of $\text{Sym}^{e}(V)$ and one defines the $j$-induction 
 $$j^W_{W^\prime}E^\prime:=W \cdot E' \subseteq {\rm Sym}^e(V),$$
i.e., it is the $W$-module in $\text{Sym}^{e}(V)$ generated by $E'$.

 To recall the Springer correspondence \cite{Spr78, Sho88}, we temporarily (in this subsection) allow $F$ to be either as before or a finite field $\mbm{F}_q, q=p^f$ with $p$ being sufficiently large. Also, $\mbf{G}$ is a split connected reductive group over such $F$. As in \S \ref{SS:KLR}, let $\mfr{B}$ be the flag variety of all Borel subgroups of $\mbf{G}(\wt{F})$. For a nilpotent element $x\in \mfr{g}\otimes \wt{F}$ one has the subvariety $\mfr{B}_{u_x}$ of Borel subgroups containing the unipotent element $u_x$  associated with $x$. The group $\text{Stab}_{\mbf{G}_{ad}(\ol{F})}(u_x)$, which is the stabilizer of $u_x$ in $\mbf{G}_{ad}(\ol{F})$, acts on $\mfr{B}_{u_x}$. One has a well-defined action of $\text{Stab}_{\mbf{G}_{ad}(\ol{F})}(u_x)$ on the $l$-adic cohomology space $H^*(\mfr{B}_{u_x}, \wt{\Q}_l)$ which factors through the component group 
$$A_x:=\pi_0(\text{Stab}_{\mbf{G}_{ad}(\ol{F})}(u_x)).$$
There is a natural action of $W$ on $H^*(\mfr{B}_{u_x}, \wt{\Q}_l)$ which commutes with that of $A_x$. This gives a decomposition of the top degree cohomology space
$$H^{\rm top}(\mfr{B}_{u_x}, \wt{\Q}_l) = \bigoplus_{\eta \in \Irr(A_x)} \eta \boxtimes \sigma_\eta,$$
where $\sigma_\eta \in \set{0} \cup \Irr(W)$. There are many properties of the correspondence thus established, one of which concerns us is that every $\sigma \in \Irr(W)$ is isomorphic to $\sigma_\eta$ for a unique nilpotent orbit $\mca{O}_x$ of $x \in \mfr{g}\otimes \wt{F}$ and a unique $\eta \in \Irr(A_x)$. In fact, $A_x$ depends only on the conjugacy class $\mca{O}_x$ of $x$. Thus, for a nilpotent orbit $\mca{O} \subset \mfr{g}\otimes \wt{F}$, we use  $A_{\mca{O}}$ to denote $A_x$ for any $x\in \mca{O}$. Let $\mca{N}(\mbf{G}(\wt{F}))$ be the set of geometric nilpotent orbits of $\mbf{G}$. Defining
$$\mca{N}(\mbf{G}(\wt{F}))^{\rm en}=\set{(\mca{O}, \eta): \ \mca{O} \in \mca{N}(\mbf{G}(\wt{F})) \text{ and } \eta \in \Irr(A_\mca{O})},$$
we thus obtain an injective map
$$\begin{tikzcd}
{\rm Spr}:={\rm Spr}_\mbf{G}: \Irr(W) \ar[r, hook] & \mca{N}(\mbf{G}(\wt{F}))^{\rm en}
\end{tikzcd}$$
denoted by
$${\rm Spr}(\sigma)=(\mca{O}_{\rm Spr}^\mbf{G}(\sigma), \eta_\sigma);$$
we call $\mca{O}_{\rm Spr}(\sigma):=\mca{O}_{\rm Spr}^\mbf{G}(\sigma) \subset \mfr{g}\otimes \wt{F}$ the nilpotent orbit associated with $\sigma$. In particular, we have  $\mca{O}_{\rm Spr}(\mbm{1}) = \mca{O}_{\rm reg}$, the regular orbit; on the other hand, $\mca{O}_{\rm Spr}(\varepsilon_W) = \set{0}$, the trivial orbit. Note that for every $\mca{O} \in \mca{N}(\mbf{G}(\wt{F}))$, the pair $(\mca{O}, \mbm{1})$ lies in the image of ${\rm Spr}$, i.e., $(\mca{O}, \mbm{1}) = {\rm Spr}(\sigma_\mca{O})$ for a unique $\sigma_\mca{O} \in \Irr(W)$. For $(\mca{O}_u, \eta) \in \mca{N}(\mbf{G}(\ol{F}))^{\rm en}$, we may write
$$E_{u, \eta}:={\rm Spr}^{-1}(\mca{O}_u, \eta).$$


\subsection{Covering Barbasch--Vogan duality} \label{SS:dBVn}
For the coordinates used in the Dynkin diagram of the root system of $\mbf{G}$, we follow Bourbaki's notation \cite{BouL2}. 
For any partition $\mfr{p}$ and $X \in \set{B, C, D}$, we use $\mfr{p}_X$ to denote the $X$-collapse of $\mfr{p}$. Also, if we write $\mfr{p}=(p_1, p_2, ..., p_k)$ in the non-increasing order, then we have the two operations
$$\mfr{p}^+:= (p_1 + 1, p_2, ..., p_k) \text{ and } \mfr{p}^-:=(p_1, p_2, ..., p_k -1).$$
These operations give rise to two natural maps
$$(\cdot)^-{}_C: \mca{N}(\SO_{2r+1}) \longrightarrow \mca{N}(\Sp_{2r})$$
and 
$$(\cdot)^+{}_B: \mca{N}(\Sp_{2r}) \longrightarrow \mca{N}(\SO_{2r+1}).$$

Now we recall the covering Barbasch--Vogan duality map
\begin{equation} \label{E:dBV}
d_{BV, G}^{(n)}: \mca{N}(\ol{G}^\vee) \longrightarrow \mca{N}(\mbf{G}(\ol{F}))
\end{equation}
defined and studied in \cite{GLLS}. The map $d_{BV,G}^{(n)}$ is order-reversing, i.e., $d_{BV,G}^{(n)}(\mca{O}^\vee) \gest d_{BV, G}^{(n)}(\mca{O}^\vee_1)$ if $\mca{O}^\vee \lest \mca{O}^\vee_1$, see \cite[Theorem 1.1]{GLLS}. Also, for classical type it is a ``deformation" of the type $A$ duality map. 

\subsubsection{\texorpdfstring{Type $A_r$}{}}
Let $G$ be of type $A_r$ with Dynkin diagram of simple roots given as follows: 
\vskip 5pt
$$
\qquad 
\begin{picture}(4.7,0.2)(0,0)
\put(1,0){\circle{0.08}}
\put(1.5,0){\circle{0.08}}
\put(2,0){\circle{0.08}}
\put(2.5,0){\circle{0.08}}
\put(3,0){\circle{0.08}}
\put(1.04,0){\line(1,0){0.42}}
\multiput(1.55,0)(0.05,0){9}{\circle{0.02}}
\put(2.04,0){\line(1,0){0.42}}
\put(2.54,0){\line(1,0){0.42}}
\put(1,0.1){\footnotesize $\alpha_{1}$}
\put(1.5,0.1){\footnotesize $\alpha_{2}$}
\put(2,0.1){\footnotesize $\alpha_{r-2}$}
\put(2.5,0.1){\footnotesize $\alpha_{r-1}$}
\put(3,0.1){\footnotesize $\alpha_{r}$}
\end{picture}
$$
\vskip 15pt
Let $\wt{G}$ be an $n$-fold cover of $G$ associated with $Q$.  
The dual group $\wt{G}^\vee$ is also of type $A_r$.  Let
$$\mfr{p}=(p_1, p_2, ..., p_k)$$
be a partition of $r+1$, viewed as an orbit of $\wt{G}^\vee$. For every number $m, z \in \N_{\gest 1}$ we consider the partition 
\begin{equation} \label{smn}
\mfr{s}(m;z):=(z^a b)=(z, z, ..., z, b)
\end{equation}
of $m$, where $m=a\cdot z + b$ with $0\lest b < z$. 
Define
$$
d_{com}^{(z)}(\mfr{p}) := \sum_{i=1}^k \mfr{s}(p_i; z) 
$$
and 
\begin{equation} \label{E:dBV-A}
d_{BV,G}^{(n)}(\mfr{p}) := d_{com}^{(\nn)}(\mfr{p}) \in \mca{N}(\mbf{G}(\ol{F})).
\end{equation}
Here, the sum of partitions is the usual one, for example, $(3,2,1)+(5,4)+(3,2)=(11,8,1)$; it corresponds to the induction of orbits. The map $d_{com}^{(z)}$ is order-reversing. Note that if $G$ is a primitive cover (see \S \ref{SS:cov}), then $d_{BV, G}^{(n)} =d_{com}^{(n)}$. On the other extreme,
if $\nn=1$, then we get
$$d_{BV,G}^{(n)}(\mfr{p}) = \mfr{p}^\top = d_{BV, G}(\mfr{p}),$$
where $\mfr{p}^\top$ means the transpose of the partition and $d_{BV,G}$ is the type $A$ linear Barbasch--Vogan duality map.

\subsubsection{\texorpdfstring{Type $B_r$}{}} 
Consider $G$ of type $B_r$ with simple roots and the extended Dyndin diagram given as follows
\vskip 5pt
$$ \qquad 
\begin{picture}(4.7,0.2)(0,0)
\put(1,0.25){\circle{0.08}}
\put(1,-0.25){\circle*{0.08}}
\put(1.5,0){\circle{0.08}}
\put(2,0){\circle{0.08}}
\put(2.5,0){\circle{0.08}}
\put(3,0){\circle{0.08}}
\put(1.46,0.03){\line(-2,1){0.42}}
\multiput(1.55,0)(0.05,0){9}{\circle*{0.02}}
\put(2.04,0){\line(1,0){0.42}}
\put(2.54,0.015){\line(1,0){0.42}}
\put(2.54,-0.015){\line(1,0){0.42}}
\put(2.68,-0.05){{\large $>$}}
\put(1.46,-0.03){\line(-2,-1){0.43}}
\put(1,0.33){\footnotesize $\alpha_1$}
\put(1.5,0.1){\footnotesize $\alpha_2$}
\put(2,0.1){\footnotesize $\alpha_{r-2}$}
\put(2.5,0.1){\footnotesize $\alpha_{r-1}$}
\put(3,0.1){\footnotesize $\alpha_r$}
\put(1,-0.4){\footnotesize $\alpha_0$}
\end{picture}
$$
\vskip 30pt
Let $\wt{G}$ be a cover of $G$ associated with $Q: Y \to \Z$. We have
$$\nn=n_{\alpha_1} =
\begin{cases}
n_{\alpha_r} & \text{ if  $\nn$ is odd},\\
2n_{\alpha_r} & \text{ if  $\nn$ is even}.
\end{cases}
$$
In the first case (i.e., when $\nn$ is odd), the dual group $\wt{G}^\vee$ is of type $C_r$; in the second case, the dual group is of type $B_r$. For $\mfr{p} \in \mca{N}(\wt{G}^\vee)$, we define (see \cite{GLLS})
\begin{equation} \label{E:dBV-B}
d_{BV, G}^{(n)}(\mfr{p}):=
\begin{cases}
d_{com}^{(\nn)}(\mfr{p})^{+} {}_B & \text{ if  $\nn$ is odd},\\
d_{com}^{(\nn)}(\mfr{p})_B & \text{ if  $\nn$ is even}.
\end{cases}
\end{equation}

\begin{eg}
Let $\wt{\SO}_{2r+1}$ be the $n$-fold cover of $\SO_{2r+1}$ associated with the unique $Q$ such that $Q(\alpha_1^\vee)=2$ (and thus $Q(\alpha_r^\vee)=4$). That is, $\wt{\SO}_{2r+1}$ is the $n$-fold cover obtained from restricting the $n$-fold cover of $\SL_{2r+1}$ associated with $Q(\alpha^\vee)=1$ for any root $\alpha$ of $\SL_{2r+1}$. One has
$$\wt{\SO}_{2r+1}^\vee \simeq
\begin{cases}
\Sp_{2r} & \text{ if $n$ is odd or $n=2k$ with $k$ odd},\\
\SO_{2r+1} & \text{ otherwise.}
\end{cases}
$$
Given any partition $\mfr{p}$ representing a nilpotent orbit of $\wt{\SO}_{2r+1}^\vee$, the orbit $d_{BV, \SO_{2r+1}}^{(n)}(\mfr{p})$ is given in \eqref{E:dBV-B} with $\nn=n/\gcd(2,n)$.
As another example, consider the cover $\wt{\Spin}_{2r+1}$ associated with $Q$ such that $Q(\alpha_1^\vee) = 1$. Thus $Q(\alpha_r^\vee)=2$. This is a primitive cover, as mentioned in \S \ref{SS:cov}.  In this case, we have (see \cite[\S 2.7]{Wei18a})
$$\wt{\Spin}_{2r+1}^\vee \simeq
\begin{cases}
{\rm PGSp}_{2r} & \text{ if $n$ is odd},\\
\text{ a group isogenous to } \Spin_{2r+1} & \text{ if $n$ is even.}
\end{cases}
$$
The map $d_{BV, \Spin_{2r+1}}^{(n)}$ is given by \eqref{E:dBV-B} with $\nn=n$.
\end{eg}

\subsubsection{\texorpdfstring{Type $C_r$}{}}
Consider $G$  of type $C_r$, whose simple roots and the extended Dynkin diagram are given as follows
\vskip 5pt
$$ \qquad 
\begin{picture}(4.2,0.2)(0,0)
\put(1,0){\circle{0.08}}
\put(0.5,0){\circle*{0.08}}
\put(1.5,0){\circle{0.08}}
\put(2,0){\circle{0.08}}
\put(2.5,0){\circle{0.08}}
\put(3,0){\circle{0.08}}
\put(1.04,0){\line(1,0){0.42}}
\multiput(1.55,0)(0.05,0){9}{\circle*{0.02}}
\put(2.04,0){\line(1,0){0.42}}
\put(2.54,0.015){\line(1,0){0.42}}
\put(2.54,-0.015){\line(1,0){0.42}}
\put(0.54,0.015){\line(1,0){0.42}}
\put(0.54,-0.015){\line(1,0){0.42}}
\put(2.68,-0.05){{\large $<$}}
\put(0.68,-0.05){{\large $>$}}
\put(1,0.1){\footnotesize $\alpha_1$}
\put(1.5,0.1){\footnotesize $\alpha_2$}
\put(2,0.1){\footnotesize $\alpha_{r-2}$}
\put(2.5,0.1){\footnotesize $\alpha_{r-1}$}
\put(3,0.1){\footnotesize $\alpha_r$}
\put(0.45,0.1){\footnotesize $\alpha_0$}
\end{picture}
$$
\vskip 15pt
Let $\wt{G}$ be a cover of $G$ associated with $Q: Y \to \Z$. We have
$$\nn= n_{\alpha_r} =
\begin{cases}
n_{\alpha_1} & \text{ if  $\nn$ is odd},\\
2n_{\alpha_1} & \text{ if  $\nn$ is even}.
\end{cases}
$$
In the first case, the dual group $\wt{G}^\vee$ is of type $B_r$; in the second case, it is of type $C_r$. For $\mfr{p} \in \mca{N}(\wt{G}^\vee)$, we define (see \cite{GLLS})
\begin{equation} \label{E:dBV-C}
d_{BV, G}^{(n)}(\mfr{p}):=
\begin{cases}
d_{com}^{(\nn)}(\mfr{p})^{-} {}_C & \text{ if  $\nn$ is odd},\\
((d_{com}^{(\nn/2)}(\mfr{p})^+)^-)_C & \text{ if $\nn$ is even with $\nn/2$ odd},\\
d_{com}^{(\nn/2)}(\mfr{p})_C & \text{ if  $\nn$ is even with $\nn/2$ even}.
\end{cases}
\end{equation}
For example, let $\wt{\Sp}_{2r}$ be the $n$-fold cover of $\Sp_{2r}$ associated with $Q$ such that $Q(\alpha_r^\vee)=1$. 
It is a primitive cover and we have
$$
\wt{\Sp}_{2r}^\vee=
\begin{cases}
\SO_{2r+1} & \text{ if $n$ is odd},\\
\Sp_{2r} & \text{ if $n$ is even}.
\end{cases}
$$
The map $d_{BV,\Sp_{2r}}^{(n)}$ is given by \eqref{E:dBV-C} with $\nn=n$.

\subsubsection{\texorpdfstring{Type $D_r$}{}} \label{SSS:D}
Consider $G$ of type $D_r$ with extended Dynkin diagram given as follows
\vskip 5pt
$$
\begin{picture}(4.7,0.4)(0,0)
\put(1,0.25){\circle{0.08}}
\put(1,-0.25){\circle*{0.08}}
\put(1.5,0){\circle{0.08}}
\put(2,0){\circle{0.08}}
\put(2.5,0){\circle{0.08}}
\put(3,0){\circle{0.08}}
\put(3.5, 0.25){\circle{0.08}}
\put(3.5, -0.25){\circle{0.08}}
\put(1.46,0.03){\line(-2,1){0.42}}
\put(1.54,0){\line(1,0){0.42}}
\multiput(2.05,0)(0.05,0){9}{\circle{0.02}}
\put(2.54,0){\line(1,0){0.42}}
\put(3.03,0.03){\line(2,1){0.43}}
\put(3.03,-0.03){\line(2,-1){0.43}}
\put(1.46,-0.03){\line(-2,-1){0.43}}
\put(1,0.36){\footnotesize $\alpha_1$}
\put(1.5,0.1){\footnotesize $\alpha_2$}
\put(2,0.1){\footnotesize $\alpha_3$}
\put(2.5,0.1){\footnotesize $\alpha_{r-3}$}
\put(2.9,0.15){\footnotesize $\alpha_{r-2}$}
\put(3.5,0.35){\footnotesize $\alpha_{r-1}$}
\put(3.5,-0.4){\footnotesize $\alpha_r$}
\put(1,-0.4){\footnotesize $\alpha_0$}
\end{picture}
$$
\vskip 35pt
Let $\wt{G}$ be an $n$-fold cover of $G$ associated with $Q: Y\to \Z$. The dual group is also of type $D_r$.
For each partition $\mfr{p}$ associated with a nilpotent orbit of $\wt{G}^\vee$, we define
\begin{equation} \label{E:dBV-D}
d_{BV,G}^{(n)}(\mfr{p}) := d_{com}^{(\nn)}(\mfr{p})_D
\end{equation}
If $r$ is even and $\mfr{p}$ is a very even partition, then there are two orbits $\mfr{p}^{\rm I}:=\mca{O}_\mfr{p}^{\rm I}, \mfr{p}^{\rm II}:=\mca{O}_\mfr{p}^{\rm II}$ associated with $\mfr{p}$, where we follow the convention of Lusztig on $\set{\rm I, II}$,  as discussed in detail in \cite{BMW25}. Suppose $d_{BV, G}^{(n)}(\mfr{p})$ is also a very even partition, then for $\set{\heartsuit, \spadesuit} = \set{\rm I, II}$ we require that
\begin{equation} \label{E:D-dec}
d_{BV,G}^{(n)}(\mfr{p}^\heartsuit) := 
\begin{cases}
d_{BV,G}^{(n)}(\mfr{p})^\heartsuit & \text{ if  $r/2$ is even},\\
d_{BV,G}^{(n)}(\mfr{p})^{\spadesuit} & \text{ if  $r/2$ is odd}.
\end{cases}
\end{equation}

As an example, consider the cover $\wt{\SO}_{2r}$ that is obtained from restricting a primitive $n$-fold cover $\wt{\SL}_{2r}$ to the embedded $\SO_{2r} \subset \SL_{2r}$. 
We have
$\wt{\SO}_{2r}^\vee \simeq \SO_{2r}$, and in this case $d_{BV,\SO_{2r}}^{(n)}$ is given in \eqref{E:dBV-D} and \eqref{E:D-dec} with $\nn=n/\gcd(n,2)$.
As another example, consider $\wt{\Spin}_{2r}$ associated with $Q$ such that $Q(\alpha_i^\vee)=1$ for any $i$; it is a primitive cover. The dual group $\wt{\Spin}_{2r}^\vee$ is isogenous to $\Spin_{2r}$, see \cite[\S 2.7]{Wei18a}, and $d_{BV, \Spin_{2r}}^{(n)}$ is given by \eqref{E:dBV-D} and \eqref{E:D-dec} with $\nn=n$.

\subsubsection{Exceptional types} \label{SS:dBV-exc}
Let $\wt{G}$ be the cover of an almost-simple exceptional group $G$ associated with $Q: Y\to \Z$. We define $d_{BV,G}^{(n)}$ as in \cite[Definition 2.1]{GLLS}. For primitive $\wt{G}$, an explicit form of $d_{BV,G}^{(n)}$ is given in the Appendix there. For general $\wt{G}$, the map $d_{BV,G}^{(n)}$ is just the one from \cite[Appendix]{GLLS} by substituting $\nn$ for $n$ there, the reason of which follows from \cite[\S 2.2]{GLLS}.



\section{Geometric wavefront sets for $\Irr_{\iota}(\ol{G})^I$}
Let $\pi=\pi(s,u,\tau) \in \Irr_\iota(\wt{G})^I$ be with Kazhdan--Lusztig--Reeder (KLR) parameter $(s,u,\tau)$ as in \S \ref{SS:KLR}. Let $\text{AZ}(\pi)$ be the Aubert--Zelevinsky dual of $\pi$. In this section, we study the geometric wavefront set ${\rm WF}^{\rm geo}({\rm AZ}(\pi))$ and relate it to the KLR parameter of $\pi$, especially to the nilpotent orbit $\mca{O}_u \in \mca{N}(\wt{G}^\vee)$ via $d_{BV, G}^{(n)}$.

\subsection{A first reduction}
Let $\mathcal{B}(G)$ be the Bruhat--Tits building of $G$. For $x \in \mathcal{B}(G)$, let $P_x$ be the parahoric subgroup attached to $x$ and let $P^{+}_x \subseteq P_x$ be its maximal pro-$p$ subgroup. The quotient 
$$L_{x}:=P_{x}/P_{x}^{+}$$ is isomorphic to the $\kappa_F$-points of a split connected reductive group $\mathbf{L}_{x}$ over $\kappa_F$. Let $W_x$ be the Weyl group of $L_x$ associated with root system $\Phi_x:=\{\alpha \in \Phi: \langle \alpha,x\rangle \in \Z\}$. Since we have assumed $p\nmid n$, the group $\wt{G}$ splits uniquely over $P_x^+$, and we get a $\mu_n$-central extension $\wt{L}_x$ of $L_x$. For every $\pi \in \Irr_\iota(\ol{G})$ we consider
$$\pi_x:=\pi^{P^+_x},$$
which is naturally a genuine $\ol{L}_x$-representation.

Consider the universal tame extension
$$\begin{tikzcd}
\mu_{q-1} \ar[r, hook] & \overline{G}^{\rm uni} \ar[r, two heads] & G
\end{tikzcd}$$
of $G$, which is of degree $q-1$. Similar as $L_x$, one has a $\mu_{q-1}$-extension $\wt{L}_x^{\rm uni}$. It is known (see \cite{BD01, We2}) that $\ol {L}^{\rm uni}_x$ is reductive, i.e., it is actually the $\kappa_F$-points of a reductive group over $\kappa_F$. Thus, one can define the Kawanaka wavefront sets  of representations of $\wt{L}_x^{\rm uni}$, see \cite{Oka21, KOW}. 

We denote by $\mbf{L}_x^F \subseteq \mbf{G}$ the split algebraic  subgroup over $F$ with the same root datum as $\mbf{L}_x$. By \cite[Lemma 2.16]{KOW}, there is a bijection between the geometric conjugacy classes of nilpotent orbits of $\ol{L}^{\text{uni}}_x$ and nilpotent orbits of the group $\mathbf{L}_x^F(\wt{F})$, which we denote by 
$$\mathcal{O} \mapsto \mathcal{O}(\wt{F}).$$
Let $\pi^{\rm uni} \in \Irr_\iota(\wt{G}^{\rm uni})$ be the representation of $\ol{G}^{\rm uni}$ obtained as the pull-back of $\pi\in \Irr_\iota(\ol{G})$ via the map $\wt{G}^{\rm uni} \onto \wt{G}$.

Henceforth, we fix $\mathcal{C}$ to be the set of vertices of the fundamental alcove in $\mca{B}(G)$.

\begin{thm}[\cite{KOW}, Theorem 2.22, Lemma 2.23] \label{T:red}
Assume that $p$ is sufficiently large. For every $\pi \in \Irr_\iota(\wt{G})^I$ the following holds:
\quad
    \begin{itemize}
        \item[(i)] $\WF(\pi)=\WF(\pi^{\rm uni})$;
        \item[(ii)] $\WF^{\rm geo}(\pi^{\rm uni})=\text{\rm max}\{\text{\rm Sat}^{\mathbf{G}(\wt{F})}_{\mathbf{L}_x^F(\wt{F})}(^{\ol{\mbm{F}}_q}\text{\rm WF}(\pi^{\rm uni}_x)(\wt{F})): \ x\in\mca{C} \},$ where $^{\ol{\mbm{F}}_q}\text{\rm WF}(-)$ denotes the geometric Kawanaka wavefront set over $\wt{\mbm{F}}_q$.
    \end{itemize}
\end{thm}

\subsection{Computation of the Kawanaka wavefront set} \label{SS:KWF}
In this subsection, $\mbf{G}$ denotes a split group over $\mbm{F}_q$. 
Let $ (\Omega, V) $ be an irreducible representation of $ \mathbf{G}(\mathbbm{F}_q) $. In \cite{Lus84-B,Lus88}, Lusztig presented a parametrization of $ (\Omega,V) $ in terms of a commuting pair $(s, u)$ up to conjugacy, where $ s $ is a semisimple element in the dual group $ G^*$ of $\mbf{G}$ and $u \in G^*$ a special unipotent element. Let $C_{G^*}(s)$ be the centralizer of $s$ in $G^*$ and $C^0_{G^*}(s)$ be its connected component. 
Let $W^{*,0}(s)$ be the Weyl group of $C^0_{G^*}(s)$.
Set
$$E_u^*:={\rm Spr}^{-1}(\mca{O}_u, \mbm{1}) \in \Irr(W^{*,0}(s)),$$
i.e.,  the representation of $W^{*,0}(s)$ associated with $\mca{O}_u$ and the trivial local system. 

One has a canonical isomorphism $W \to W^*$ between the Weyl groups of $\mbf{G}$ and $G^*$. It induces a bijection between $\Irr(W)$ and $\Irr(W^*)$, denoted by $E \leftrightarrow E^*$ for their elements. This also applies to subgroups of $W$ and $W^*$ which correspond to each other under the above isomorphism $W \to W^*$.
In particular, $W^0(s) \subseteq W$ corresponds to $W^{*,0}(s) \subseteq W^*$, and $E_u \in \Irr(W^0(s))$ corresponds to $E_u^* \in \Irr(W^{*,0}(s))$.
Following this convention on the notation, for the Kawanaka wavefront set of $(\Omega, V)$ one has (see \cite[Theorem 14.16]{Tay16})
\begin{equation} \label{E:WF1}
{}^{\overline{\mbm{F}}_q}\text{WF}(\Omega) = \mca{O}_{\text{Spr}}(j^{W}_{W^0(s)}E_u).
\end{equation}

We consider some special cases. First, we call $\Omega \in \Irr(\mbf{G}(\mathbbm{F}_q))$ a unipotent principal series representation if $\Omega^{\mbf{B}(\mbm{F}_q)} \ne 0$, or equivalently, $\Omega$ is an irreducible constituent of the principal series representation of $\mbf{G}(\mbm{F}_q)$ induced from the trivial character of $\mbf{B}(\mbm{F}_q)$. In this case, the space $\Omega^{\mbf{B}(\mbm{F}_q)}$ affords an irreducible representation of the Hecke algebra $\mathcal{H}(\mathbf{B}(\mbm{F}_q)\backslash \mathbf{G}(\mbm{F}_q)/\mathbf{B}(\mbm{F}_q))$. Now, from the Lusztig isomorphism 
\begin{equation}
	\C[W] \rightarrow \mathcal{H}(\mathbf{B}(\mbm{F}_q)\backslash \mathbf{G}(\mbm{F}_q)/\mathbf{B}(\mbm{F}_q))
\end{equation}
of algebras, the representation $\Omega$ gives rise to an irreducible representation of $\C[W]$, which we denote by $\Omega_{q\to 1}$. By \cite[Proposition 12.6]{Lus84-B}, for such unipotent $\Omega$ we have
\begin{equation}
^{\overline{\mathbbm{F}}_q}\text{WF}(\Omega)=\mathcal{O}_{\text{Spr}}((\Omega_{q\rightarrow 1}\otimes \varepsilon_W )^{\text{spe}}), 
\end{equation}
where $E^{\text{spe}} \in \Irr(W)$ denotes the unique special representation in the same Lusztig family as $E\in \Irr(W)$.

In general, let $\chi: \mbf{T}(\mbm{F}_q) \to \C^\times$ be a character which corresponds to an element $s:=s_\chi \in G^*$. It is also viewed as a character of $\mbf{B} (\mbm{F}_q)$ trivial on the unipotent subgroup. We call $\Omega \in \Irr(\mbf{G}(\mbm{F}_q))$ a general principal series representation if $\Omega^{(\mbf{B}(\mbm{F}_q),\chi)} \neq 0$ for some $\chi$. For general $\chi$, one has an embedding 
\begin{equation} 
	\C[W^0(s)] \hookrightarrow \mathcal{H}((\mathbf{B}(\mbm{F}_q),\chi^{-1})\backslash \mathbf{G}(\mbm{F}_q)/(\mathbf{B}(\mbm{F}_q),\chi^{-1})).
\end{equation}
Similar to the unipotent case above, a general principal series representation $\Omega$ gives rise to a representation (not necessarily irreducible) $ \Omega_{q \rightarrow 1}$ of $W^0(s)$.


\begin{lm}
For a general principal series representation $\Omega \in \Irr(\mbf{G}(\mathbbm{F}_q))$ , one has 
\begin{equation} \label{E:Kaw1}
^{\overline{\mathbbm{F}}_q}\WF(\Omega)=\mathcal{O}_{\rm Spr}(j^W_{W^0(s)}(\sigma \otimes \varepsilon)^{\rm spe}),
\end{equation}
where $\sigma \subseteq \Omega_{q\to 1}$ is an irreducible constituent occurring in $\Omega_{q\to 1}$.
\end{lm}
\begin{proof}
If $\textbf{G}$ has connected center, then $C_{G^*}(s)=C^0_{G^*}(s)$. In this case, $\Omega_{q\rightarrow 1}$ is irreducible and we know from \cite[\S 5]{AM22} and \cite[Proposition 12.6]{Lus84-B} that $E_u=(\Omega_{q\rightarrow1} \otimes \varepsilon )^{\text{spe}}$. The result then follows from \eqref{E:WF1}.

In general, we need to consider a regular embedding $\textbf{G}(\mathbbm{F}_q) \hookrightarrow \textbf{G}^{\prime}(\mathbbm{F}_q)$ such that $\textbf{G}'$ has connected center, as in \cite{Lus88}. Let $\Omega^{\prime}\in \text{Irr}(\textbf{G}^{\prime}(\mathbbm{F}_q))$ be containing $\Omega$. Let $\chi^{\prime}: \textbf{B}^{\prime}(\mathbbm{F}_q)\rightarrow \C^{\times}$ be an extension of $\chi$ such that $(\Omega')^{(\textbf{B}^{\prime}(\mathbbm{F}_q),\chi^\prime)}\neq 0$. Then $\C[W^0(s)] \simeq \mathcal{H}((\mathbf{B}^{\prime}(\mbm{F}_q),\chi^{\prime-1})\backslash \mathbf{G}^{\prime}(\mbm{F}_q)/(\mathbf{B}^{\prime}(\mbm{F}_q),\chi^{\prime-1}))$ and $\Omega^{\prime}_{q\rightarrow 1}$ is an irreducible constituent of $\Omega_{q\rightarrow1}$. Since the unipotent parameter for $\Omega$ can be taken to be the same as for $\Omega^{\prime}$, we see the lemma holds for $\Omega$ in view of \eqref{E:WF1}, by taking $\sigma:=\Omega'_{q\to 1}$.
\end{proof}

In fact, by \cite[\S 12.7]{Tay16}, the equality in \eqref{E:Kaw1} holds for any irreducible constituent $\sigma$ of $\Omega_{q \to 1}$.  Thus we will simply write 
\begin{equation} \label{E:Kaw2}
 ^{\overline{\mathbbm{F}}_q}\WF(\Omega)=\mathcal{O}_{\text{Spr}}(j^W_{W^0(s)}(\Omega_{q\to 1} \otimes \varepsilon )^{\text{spe}})
 \end{equation}
in the place of \eqref{E:Kaw1}.

\subsection{A formula for the geometric wavefront set} \label{SS:RepIH}
Let $\pi:=\pi(s, u, \tau) \in \Irr_\iota(\ol{G})^I$. The Kazhdan--Lusztig--Reeder triple $ (s,u,\tau) $ also gives an Iwahori-spherical representation $\pi_{Q,n}:= \pi_{G_{Q,n}}(s,u,\tau) $ of $G_{Q,n}$. Recall that $\mathcal{C}$ denotes the set of vertices of the fundamental alcove. We need to determine $\pi^{\text{uni}}_x, x\in \mca{C}$ and ${}^{\wt{\mbm{F}}_q} \WF(\pi_x^{\rm uni})$ that appear in Theorem \ref{T:red}.

First note that 
$$ \ol{B}_x:=\ol{I}/P^+_x \subset \ol{P}_x/P^+_x$$
is a Borel subgroup of $\ol{L}_x$. We also have the Borel subgroup $\ol{B}^{\text{uni}}_x$ of $\ol{L}^{\text{uni}}_x$. The character $\chi_i$ in \eqref{E:chi-i} descends to a character of $\wm{T}(\kappa_F)$, and we further view it as a character of $\ol{B}_x$ and $\ol{B}^{\text{uni}}_x$ via the natural pull-backs.

Since 
$$(\pi^{\text{uni}}_x)^{(\ol{B}^{\text{uni}}_x,\chi_0)}=(\pi_x)^{(\ol{B}_x,\chi_0)}=\pi^{(\ol{I},\chi_0)} \neq 0,$$
it follows from \cite[Propositions 3.5 and 3.6]{HW09} that every irreducible constituent of $\Omega \subseteq \pi^{\text{uni}}_x$ is a principal series representation of $\ol{L}^{\text{uni}}_x$, i.e., $\Omega^{(\ol{B}^{\text{uni}}_x,\chi)} \neq 0$ for some character $\chi: \wt{B}_x^{\rm uni} \to \C^\times$. Moreover, we note that \cite[Proposition 3.2 (5)]{KOW} was stated for theta representations, but its proof clearly applies to genuine Iwahori-spherical representations. Thus, it follows from \cite[Proposition 3.2 (5)]{KOW} that, if $\Omega^{(\ol{B}^{\text{uni}}_x,\chi)} \neq 0$, then $\chi$ is equal to some $\chi_i$ defined in \eqref{E:chi-i}.
This gives the inclusion
$$\Omega^{(\ol{B}^{\text{uni}}_x,\chi_i)} \into (\pi_x^{\rm uni})^{(\ol{B}^{\text{uni}}_x,\chi_i)} $$
as modules over the algebra $\mathcal{H}((\overline{B}_x^{\rm uni},\chi_i^{-1})\backslash \ol{L}^{\text{\rm uni}}_x/(\overline{B}_x^{\rm uni},\chi_i^{-1}))$, and $\Omega^{(\ol{B}^{\text{uni}}_x,\chi_i)} \ne \set{0}$ is irreducible.
	
In view of this, we only need to determine the action of $\mathcal{H}((\overline{B}_x^{\rm uni},\chi_i^{-1})\backslash \ol{L}^{\text{\rm uni}}_x/(\overline{B}_x^{\rm uni},\chi_i^{-1}))$ on the space $(\pi_x^{\rm uni})^{(\wt{B}_x^{\rm uni}, \chi_i)}$. Consider the isomorphism 
$$ \mathcal{H}((\overline{B}_x^{\rm uni},\chi_i^{-1})\backslash \ol{L}^{\text{\rm uni}}_x/(\overline{B}_x^{\rm uni},\chi_i^{-1})) \simeq \mathcal{H}((\overline{I}^{\rm uni},\chi_i^{-1})\backslash \ol{P}^{\text{\rm uni}}_x/(\overline{I}^{\rm uni},\chi_i^{-1})) $$
induced by the projection $\ol{P}^{\text{\rm uni}}_x \onto  \ol{L}^{\text{\rm uni}}_x  $. 
We have the identification $(\pi_x^{\rm uni})^{(\wt{B}_x^{\rm uni}, \chi_i)} = (\pi^{\rm uni})^{(\wt{I}^{\rm uni}, \chi_i)}$, which is equivariant with respect to the preceding isomorphism of algebras. Hence, we are further reduced to determine the action of $\mathcal{H}((\overline{I}^{\rm uni},\chi_i^{-1})\backslash \ol{P}^{\text{\rm uni}}_x/(\overline{I}^{\rm uni},\chi_i^{-1}))$ on $(\pi^{\rm uni})^{(\wt{I}^{\rm uni}, \chi_i)}$.

By \cite[Lemma 3.13]{KOW}, the projection map $\ol{G}^{\text{uni}}\twoheadrightarrow \ol{G}$ naturally induces an isomorphism 
$$\mathcal{H}((\overline{I}^{\rm uni},\chi_i^{-1})\backslash \ol{P}^{\rm uni}_x/(\overline{I}^{\rm uni},\chi_i^{-1})) \simeq \mathcal{H}((\overline{I},\chi_i^{-1})\backslash \ol{P}_x/(\overline{I},\chi_i^{-1}))$$
of algebras. Therefore, it suffices to determine how $\mathcal{H}((\overline{I},\chi_i^{-1})\backslash \ol{P}_x/(\overline{I},\chi_i^{-1}))$ acts on $\pi^{(\ol{I},\chi_i)} \simeq (\pi^{\rm uni})^{(\wt{I}^{\rm uni}, \chi_i)}$. Following this and using the isomorphism
\begin{equation} \label{E:P}
\mathcal{P}: \mathcal{H}((\overline{B}_x,\chi_i^{-1})\backslash \ol{L}_x/(\overline{B}_x,\chi_i^{-1})) \simeq \mathcal{H}((\overline{I},\chi_i^{-1})\backslash \ol{P}_x/(\overline{I},\chi_i^{-1}))
 \end{equation}
we can use \eqref{E:Kaw2} to compute the geometric Kawanaka wavefront set ${}^{\wt{\mbm{F}}_q} \WF(\pi_x^{\rm uni})$.

Let $\chi_i: \wt{I} \to \C^\times$ be as in \eqref{E:chi-i} with $0\lest i < \val{\msc{X}}$. 
Since we have an inclusion 
$$\mathcal{H}((\overline{I},\chi_i^{-1})\backslash \ol{P}_x/(\overline{I},\chi_i^{-1})) \into \mathcal{H}((\overline{I},\chi_i^{-1})\backslash \overline{G}/(\overline{I},\chi_i^{-1})) $$
of algebras, the latter of which acts naturally on $\pi^{(\wt{I}, \chi_i)}$, we will work out the module structure of $\pi^{(\wt{I}, \chi_i)}$ over  $\mathcal{H}_i:=\mathcal{H}((\overline{I},\chi_i^{-1})\backslash \overline{G}/(\overline{I},\chi_i^{-1}))$. In fact, we will do so by transporting the action via $\Psi_i$ in \eqref{eq:Psi-i}.

If an element in $\mca{H}_i$ has support of the form $\wt{I}\cdot \varpi^w \wt{I}$ with $w\in W\ltimes Y$, by abuse of language we say that it has support on the double set of $w$. It follows from \cite[\S 5]{Wang24} that the isomorphism
$$\Psi_i:\mathcal{H}_i\rightarrow \mathcal{H}(G_{Q,n}, I_{Q,n})$$
in \eqref{eq:Psi-i} maps an element supported on the double set of $w \in W\ltimes Y$ to an element supported on the double set of 
$$(w_i y_i^{-1}) w (y_i w_i^{-1}) \in W \ltimes Y_{Q,n}$$
for some specific $ w_i=\widehat{w}_i\cdot y^{Q,n}_i \in W\ltimes Y_{Q,n}$ depending on $\chi_i$, and $y_i \in Y$ is the element such that $\chi_i = y_i \cdot \chi_0$ as in \eqref{E:chi-i}. For  the parahoric subgroup $ \overline{P}_x \subseteq \overline{G} $ associated with $x\in \mca{C}$, we have 
\begin{equation} \label{iso}
\Psi_i: \mathcal{H}((\overline{I},\chi_i^{-1})\backslash \overline{P}_x/(\overline{I},\chi_i^{-1})) \simeq \mathcal{H}(I_{Q,n}\backslash P^{Q,n}_{w_i\cdot (x-y_i)}/I_{Q,n}),
\end{equation}
where $ P^{Q,n}_{w_i\cdot (x-y_i)} $ is the parahoric subgroup of $ G_{Q,n} $ associated with $w_i\cdot (x-y_i)$. Here $w_i \cdot (x-y_i) \in Y\otimes \R$ arises from the natural action of $w_i \in W\ltimes Y_{Q,n}$ on $x-y_i \in Y\otimes \R$, and it lies in the closure of the fundamental alcove associated with $G_{Q,n}$.

We observe that 
\begin{equation} \label{E:Phi-Qn-z}
\Phi^{Q,n}_{w_i\cdot (x-y_i)}=\{\alpha/n_{\alpha} \in \Phi_{Q,n}:\ \langle \alpha/n_{\alpha}, w_i\cdot (x-y_i) \rangle \in \Z \}
\end{equation}
is the root system of the reductive group $L^{Q,n}_{w_i\cdot (x-y_i)}$ associated with $P^{Q,n}_{w_i\cdot (x-y_i)}$ and thus underlies the support of the Hecke algebra $\mathcal{H}(I_{Q,n}\backslash P^{Q,n}_{w_i\cdot (x-y_i)}/I_{Q,n})$ via (the $G_{Q,n}$-version of) $\mathcal{P}$ in \eqref{E:P}. Here $\Phi^{Q,n}_{w_i\cdot (x-y_i)}$ is a psuedo-Levi subsystem of the root system $\Phi_{Q,n}$, i.e., it corresponds to a subdiagram of the extended Dynkin diagram of $\Phi_{Q,n}$. Similarly,
$$\Phi^{Q,n}_{x,\chi_i}=\{\alpha \in \Phi: \ \langle \alpha/n_{\alpha},x-y_i \rangle \in \Z\}$$ is the root system that underlies the support of $ \mathcal{H}((\ol{I},\chi_i^{-1} )\backslash \ol{P}_x/(\ol{I},\chi_i^{-1} ))$ via $\mca{P}$ in \eqref{E:P}. Also, $\Phi^{Q,n}_{x, \chi_i}$ corresponds canonically bijective to 
$$\set{\alpha/n_\alpha: \ \alpha \in \Phi^{Q,n}_{x, \chi_i}},$$
which is a pseudo-Levi subsystem of $\Phi_{Q,n}$. On the other hand, when $x \in \mca{C}$ is hyperspecial, it follows from \cite[Lemma 4.5]{KOW} that $\Phi^{Q,n}_{x,\chi_i}$ itself is actually a dual-pseudo-Levi subsystem of $\Phi$, i.e., its dual
$$\Phi^{Q,n,\vee}_{x,\chi_i}:=\set{\alpha^\vee: \ \alpha \in \Phi^{Q,n}_{x,\chi_i}} \subseteq \Phi^\vee$$
is a pseudo-Levi subsystem of $\Phi^\vee$. In general, $\Phi^{Q,n,\vee}_{x,\chi_i}$ is a pseudo-Levi subsystem of $\Phi_x^\vee$, which could be shown from the observation at the beginning of \S\ref{SS:Pf-cla}.


There exists an element $ \Theta_{t_i} \in  \mca{H}_{\bar\iota}(\wt{G}, I_1)$ with strongly positive $t_i \in Y$ (as defined in \cite{GGK1}) of the pro-$p$ Iwahori--Hecke algebra such that 
$$\mathcal{H}_i=\Theta_{t_i} \mathcal{H}_0\Theta^{-1}_{t_i}$$
and 
\begin{equation} \label{E:Theta-ti}
\Theta_{t_i} \cdot \pi^{(\overline{I},\chi_0)}=\pi^{(\overline{I},\chi_i)}.
\end{equation}
We consider the automorphism
$$\pmb{\beta}: \mathcal{H}(G_{Q,n}, I_{Q,n}) \to \mathcal{H}(G_{Q,n}, I_{Q,n})$$ 
given by
$$\pmb{\beta}(T):= \Psi_i(\Theta_{t_i}(\Psi_0^{-1}(T))\Theta^{-1}_{t_i})$$
and also the automorphism 
$$\pmb{\beta}^\dag: \mathcal{H}(G_{Q,n}/I_{Q,n}) \rightarrow \mathcal{H}(G_{Q,n}/I_{Q,n})$$
given by 
$$
\pmb{\beta}^\dag(T):=T^{-1}_{\widehat{w}_i^{-1}} \cdot \pmb{\beta}(T) \cdot T_{\widehat{w}_i^{-1}},
$$
where  $T_{\widehat{w}_i^{-1}}$ is the element supported on $ \widehat{w}_i^{-1}$.

We summarize various relations of the above Hecke algebras in the following commutative diagram:
\begin{equation} \label{E:keyCD}
\begin{tikzcd}
\mca{H}_0 \ar[r, "\Psi_0"] \ar[d, "{\Theta_{t_i}(\cdot) \Theta_{t_i}^{-1}}"'] & \mca{H}(G_{Q,n}, I_{Q,n}) \ar[d, "{\pmb{\beta}}"'] \ar[rrd, "{\pmb{\beta}^\dag}"]   \\
\mca{H}_i \ar[r, "\Psi_i"] & \mca{H}(G_{Q,n}, I_{Q,n})  \ar[rr, "{T^{-1}_{\widehat{w}_i^{-1}} (\cdot) T_{\widehat{w}_i^{-1}}}"'] & & \mca{H}(G_{Q,n}, I_{Q,n})  \\
\mathcal{H}((\ol{I},\chi_i^{-1} )\backslash \ol{P}_x/(\ol{I},\chi_i^{-1} )) \ar[u, hook] \ar[r, "{\Psi_i}"] &  \mathcal{H}(I_{Q,n}\backslash P^{Q,n}_{w_i\cdot (x-y_i)}/I_{Q,n}) \ar[u, hook]
\end{tikzcd}
\end{equation}

To write down the KLR parameter of $\pi^{(\wt{I}, \chi_i)}$ viewed as a module over $\mca{H}(G_{Q,n}, I_{Q,n})$ via $\Psi_i$, we note that we can further push the action via the inner automorphism 
$$T^{-1}_{\widehat{w}_i^{-1}}(\cdot)T_{\widehat{w}_i^{-1}}: \mca{H}(G_{Q,n}, I_{Q,n}) \to \mca{H}(G_{Q,n}, I_{Q,n}),$$
which does not change the KLR parameter. Further more, in view of \eqref{E:Theta-ti}, it is equivalent to obtain the KLR parameter of $\pi^{(\wt{I}, \chi_0)}$ viewed as a module over  $\mca{H}(G_{Q,n}, I_{Q,n})$ via $\pmb{\beta}^\dag \circ \Psi_0$. Note that $\pi^{(\wt{I}, \chi_0)}$ viewed as a module over $\mca{H}(G_{Q,n}, I_{Q,n})$ has KLR parameter exactly $(s, u, \tau)$. Moreover, $\pmb{\beta}^\dag$
preserves the elements supported on the strongly positive $ t \in Y_{Q,n}$, it then follows from
\cite[Lemma 14.3]{ABPS17} that $\pmb{\beta}^\dag$ fixes the parameters $s, u$ in $(s, u, \tau)$ and may change $\tau$. That is, $\pi^{(\overline{I},\chi_0)} $ viewed as a representation of $ \mathcal{H}(G_{Q,n},I_{Q,n})$ through $\pmb{\beta}^\dag \circ \Psi_0$ has Kazhdan--Lusztig--Reeder parameter 
$$(s,u,\tau_i)$$
for some $ \tau_i$.

We call $s$ positive real if $s\in \wt{T}^\vee(\R_{>0})$; it is termed real in \cite{CMBO21}.
Now we use results from \cite{ABPS17} to show that if $s$ is positive-real, then $\tau_i = \tau$. 
First, given the KLR parameter $(s,u,\tau)$, we consider
$$s^\sharp:= s\cdot \varphi_u(\alpha^\vee(q^{-1/2}))^{-1} \in \wt{T}^\vee,$$
where $\varphi_u: \SL_2(\C) \to \wt{G}^\vee$ is the Jacobson--Morozov homomorphism associated with $u$ and $\alpha^\vee$ is the coroot of $\SL_2(\C)$. One has $u \in Z_{\ol{G}^{\vee}}(s^\sharp)$. Set $M_{s^\sharp}:=Z_{\ol{G}^{\vee}}(s^\sharp)$ and let $M_{s^\sharp}^0 \subseteq M_{s^\sharp}$ be the connected component of $M_{s^\sharp}$. By \cite[Proposition 6.1 and (49)]{ABPS17}, we have a natural embedding 
$$\pi_0(Z_{M_{s^\sharp}^0}(u)) \into A_{s,u}.$$
For the KLR parameter $(s, u, \tau_i)$ arising as above, the proof of \cite[Lemma 14.3]{ABPS17} in fact shows that 
$$\tau|_{\pi_0(Z_{M_{s^\sharp}^0}(u))}=\tau_i|_{\pi_0(Z_{M_{s^\sharp}^0}(u))}.$$
However, if $s$ is positive real, then $M_{s^\sharp} =M_{s^\sharp}^0$, and by \cite[Proposition 6.1 and (49)]{ABPS17} we have $A_{s,u}=\pi_0(Z_{M_{s^\sharp}^0}(u))$; we get $\tau_i=\tau$ in this case.

 Thus, if $s$ is positive real, then the KLR parameter of $\pi^{(\wt{I}, \chi_i)}$ as $\mca{H}(G_{Q,n},I_{Q,n})$-module via $\Psi_i$ is $(s, u, \tau)$.
 To further restrict to $\mathcal{H}(I_{Q,n}\backslash P^{Q,n}_{w_i\cdot (x-y_i)}/I_{Q,n})$, let
 $W_{Q,n}$ be the Weyl group of $G_{Q,n}$ and $W^{\vee}_{Q,n}$ be the Weyl group of $\ol G^{\vee}$. One has the canonical bijections
$$\vartheta: \Phi \rightarrow \Phi_{Q,n},\quad   \alpha \mapsto \alpha_{Q,n}:=\alpha/n_\alpha$$
and
$$\vartheta^\vee: \Phi \rightarrow \Phi_{Q,n}^{\vee}, \quad   \alpha \mapsto \alpha_{Q,n}^\vee:=n_\alpha \alpha^\vee.$$
They induce isomorphisms 
\begin{equation} \label{E:Widen}
\vartheta: W\rightarrow W_{Q,n} \text{ and } \vartheta^{\vee}: W\rightarrow W^{\vee}_{Q,n}.
\end{equation}
By further abuse of notation, we also denote by $\vartheta: \Irr(W) \rightarrow \Irr(W_{Q,n})$ and $\vartheta^\vee: \Irr(W) \rightarrow \Irr(W^{\vee}_{Q,n})$ the induced bijections. Also note that the map
\begin{equation} \label{E:vt-what}
\vartheta \circ \hat{w}_i = \hat{w}_i \circ \vartheta: \Phi_{x, \chi_i}^{Q,n} \longrightarrow \Phi_{w_i\cdot (x-y_i)}^{Q,n}
\end{equation}
given by $\alpha \mapsto \hat{w}_i \circ \vartheta(\alpha) = \hat{w}_i(\alpha/n_\alpha)$ is a bijection.

Henceforth, we identify the algebras and their modules of the two sides of $\mca{P}$ in \eqref{E:P}.
Let $\pi_{G_{Q,n}}(s, u, \tau) \in \Irr(G_{Q,n})^{I_{Q,n}}$ and let $z$ be in the closure of fundamental alcove associated with $G_{Q,n}$. Let $W_{Q,n,z}:=W(\Phi^{Q,n}_z)$ with $\Phi^{Q,n}_z$ given as in \eqref{E:Phi-Qn-z}.
If $s$ is positive-real, then it follows from \cite[(2.6.5), Lemma 3.0.7 and (3.0.9)]{CMBO21} that
\begin{equation} \label{eq:res}
(\text{AZ}(\pi_{G_{Q,n}}(s,u,\tau))^{I_{Q,n}}|_{\mathcal{H}(I_{Q,n}\backslash P^{Q,n}_{z}/I_{Q,n})})_{q \rightarrow 1}=\text{Res}^{W_{Q,n}}_{W_{Q,n,z}}(\bigoplus_{\rho \in \Irr(W)_{s,u,\tau}} \vartheta(\rho))
\end{equation}
where the set $\Irr(W)_{s,u,\tau} \subseteq \Irr(W)$ satisfies the following:
\begin{enumerate}
\item[$\bullet$] every $\vartheta^\vee(\rho)$ with $\rho \in \Irr(W)_{s,u,\tau}$ is of the form $ E_{u',\eta}:={\rm Spr}_{\wt{G}^\vee}^{-1}(\mca{O}_{u'}, \eta)$ for some $\mathcal{O}_{u^{\prime}} \gest \mathcal{O}_{u}$ and $\eta \in \Irr(A_{\mca{O}_{u'}})$ (note $\wt{G}^\vee = G_{Q,n}^\vee$);
\item[$\bullet$] every $E_{u, \eta}$ with geometric $\eta \in \Irr(A_{\mca{O}_u})$ occurs as some $\vartheta^\vee(\rho), \rho \in \Irr(W)_{s,u,\tau}$ with multiplicity $\dim \Hom_{A_{s, u}}(\tau, \eta)$, where we have $A_{s, u} \into A_{\mca{O}_u}$.
\end{enumerate}
Note that the above two properties regarding $\rho$ (such that it lies in $\Irr(W)_{s, u, \tau}$) are only necessary, and in general it is difficult to determine precisely the set $\Irr(W)_{s,u,\tau}$ that occurs in \eqref{eq:res}.

\begin{prop} \label{main_prop}
Retain the above notation. Assume $p$ is sufficiently large and that $s$ is positive real.  Then for every $\pi(s, u, \tau) \in \Irr_\iota(\wt{G})^I$ one has
\begin{equation*}
{\rm WF}^{\text{\rm geo}}(\text{\rm AZ}(\pi(s,u,\tau)))=
	\mathop{\text{\rm max}}\limits_{x\in \mca{C}\atop 0 \lest i < \val{\msc{X}}}
	\left\{
	\begin{array}{cc}
	\text{\rm Sat}^{\mathbf{G}(\wt{F})}_{\mathbf{L}_x^F(\wt{F})}(\mathcal{O}_\text{\rm Spr}^{\mbf{L}_x^F}(j^{W_x}_{W(\Phi^{Q,n}_{x,\chi_i})}(\sigma^{\rm spe}))): \\
\bullet \ \sigma \in \Irr(W(\Phi_{x, \chi_i}^{Q,n})),\\
\bullet \  \sigma \subseteq {\rm Res}^{W}_{W(\Phi^{Q,n}_{x,\chi_i})}(\bigoplus_{\rho \in \Irr(W)_{s, u, \tau}} \rho \otimes \varepsilon ) 
	\end{array}
	\right\}.
\end{equation*}
\end{prop}
\begin{proof}
We write $\pi:=\pi(s, u, \tau)$. By Theorem \ref{T:red}, we only need to compute $^{\ol{\mbm{F}}_q}\text{\rm WF}(\pi^{\rm uni}_x)(\wt{F})$ for each $x \in \mathcal{C}$, for which a specific formula is given in \eqref{E:Kaw2}. That is, we have 
$$
{\rm WF}^{\text{\rm geo}}(\text{\rm AZ}(\pi))=
\max_{\substack{x\in \mca{C}}} 
\left \{
\begin{array}{ccc}
\text{\rm Sat}^{\mathbf{G}(\wt{F})}_{\mathbf{L}_x^F(\wt{F})}(\mathcal{O}_\text{\rm Spr}^{\mbf{L}_x^F}(j^{W_x}_{W_x^0(s_\Omega)}((\Omega_{q\to 1} \otimes \varepsilon)^{\rm spe}))): \\
\bullet \ \Omega \in \Irr(\wt{L}_x^{\rm uni})  \text{ is any irreducible }\\
\quad \text{ constituent of }  {\rm AZ}(\pi)_x^{\rm uni} 
\end{array}
\right \}.
$$
Here $s_\Omega:=s_{\chi_i}$ if $\Omega \subseteq {\rm Ind}_{\wt{B}_x^{\rm uni}}^{\wt{L}_x^{\rm uni}} \chi_i$; and it also gives $W_x^0(s_\Omega) = W(\Phi_{x, \chi_i}^{Q,n})$.
Also, $\Omega_{q\to 1}$ occurs as a constituent of ${\rm AZ}(\pi)^{(\wt{I}, \chi_i)}|_{\mathcal{H}((\overline{I},\chi_i^{-1})\backslash \overline{P}_x/(\overline{I},\chi_i^{-1})) }$.
The discussion above in this subsection shows that, since $s$ is positive-real, then ${\rm AZ}(\pi)^{(\wt{I}, \chi_i)}$ viewed as $\mca{H}(G_{Q,n},I_{Q,n})$-module via $\Psi_i$ has the same KLR parameter as ${\rm AZ}(\pi)$. Moreover, \eqref{eq:res} gives the restriction of this module to $\mathcal{H}(I_{Q,n}\backslash P^{Q,n}_{w_i\cdot (x-y_i)}/I_{Q,n})$, this in turn shows that  
$$
({\rm AZ}(\pi)^{(\wt{I}, \chi_i)}|_{\mathcal{H}((\overline{I},\chi_i^{-1})\backslash \overline{P}_x/(\overline{I},\chi_i^{-1})) })_{q\rightarrow1}  = (\Psi_i^*)^{-1} \circ {\rm Res}^{W_{Q,n}}_{W_{Q,n,w_i\cdot (x-y_i)}}(\bigoplus_{\rho \in \Irr(W)_{s,u,\tau}} \vartheta(\rho)),
$$
where $\Irr(W)_{s, u, \tau}$ is described as above. Here $W_{Q,n} = W(\Phi_{Q,n})$, and $W_{Q,n,w_i\cdot (x-y_i)} = W(\Phi^{Q,n}_{w_i\cdot (x-y_i)})$, and $\Psi_i^*$ is the map induced from $\Psi_i$ in \eqref{iso} and Lusztig's isomorphism.
However, the isomorphism $\Psi_i$, when restricted to the group algebras $\C[W(\Phi_{x, \chi_i}^{Q,n})]$ and $\C[W(\Phi_{w_i\cdot (x-y_i)}^{Q,n})]$ for the two sides of \eqref{iso} via the Lusztig isomorphism, is exactly the one induced from $\vartheta \circ \hat{w}_i$ in \eqref{E:vt-what}. This shows that
\begin{equation} \label{E:kred}
({\rm AZ}(\pi)^{(\wt{I}, \chi_i)}|_{\mathcal{H}((\overline{I},\chi_i^{-1})\backslash \overline{P}_x/(\overline{I},\chi_i^{-1})) } )_{q\rightarrow1}=  {\rm Res}_{W(\Phi_{x, \chi_i}^{Q,n})}^W \bigoplus_{\rho \in \Irr(W)_{s, u, \tau}} \rho
 \end{equation}
Thus, $\Omega_{q\to 1}$ ranges over all irreducible constituents of the right hand side of \eqref{E:kred}. We replace $\Omega_{q\to 1} \otimes \varepsilon$ by $\sigma$, then the desired equality follows immediately, and this completes the proof.
\end{proof}

\subsection{Reduction for the upper bound}

\begin{lm} \label{lm:AA}
Let $E_{u, \eta}:= {\rm Spr}_{\wt{G}^\vee}^{-1}(\mca{O}_u^\vee, \eta)$ and $\sigma \in \Irr(W(\Phi^{Q,n}_{x,\chi_i}))$, where $\mca{O}^\vee_u \in \mca{N}(\wt{G}^\vee)$ is the orbit associated with $u$. If $\sigma$ satisfies that $\vartheta^\vee(\sigma) \subseteq \text{\rm Res}^{\vartheta^{\vee}(W)}_{\vartheta^{\vee}(W(\Phi^{Q,n}_{x,\chi_i}))}(E_{u,\psi}\otimes \varepsilon)$, then 
$$\mathcal{O}_{\text{\rm Spr}}(j^{\vartheta^{\vee}(W)}_{\vartheta^{\vee}(W(\Phi^{Q,n}_{x,\chi_i}))}(\vartheta^\vee(\sigma) \otimes \varepsilon)^{\text{\rm spe}}) \gest \mca{O}^\vee_u.$$
\end{lm}
\begin{proof}
As mentioned in \S \ref{SS:RepIH}, the set $\set{\alpha_{Q,n}: \alpha \in \Phi^{Q,n}_{x,\chi_i}}$ is a pseudo-Levi subsystem of $\Phi_{Q,n}$, we see that $\vartheta^\vee(W(\Phi^{Q,n}_{x,\chi_i}))$ is the Weyl group of the dual-pseudo-Levi subsystem 
$$\set{\alpha_{Q,n}^\vee: \alpha \in \Phi^{Q,n}_{x,\chi_i}}$$
of $\Phi_{Q,n}^\vee$ (and thus of the group $G_{Q,n}^\vee \simeq \wt{G}^\vee$). The result then follows from \cite[Proposition 4.3]{AcAu07}.
\end{proof}

Let $\mca{O}^\vee \in \mca{N}(\wt{G}^\vee)$. Motivated from Lemma \ref{lm:AA} (or rather, from Corollary \ref{C:key} later), for every fixed $x \in  \mca{C}$ and $0\lest i \lest k$, we define 
\begin{equation} \label{E:Dxchi}
\mca{D}(\mca{O}^\vee)_{x, \chi_i}:=\max
\left\{
	\begin{array}{cc}
	\text{\rm Sat}^{\mathbf{G}(\wt{F})}_{\mathbf{L}_x^F(\wt{F})}(\mathcal{O}_\text{\rm Spr}^{\mbf{L}_x^F}(j^{W_x}_{W(\Phi^{Q,n}_{x,\chi_i})}(\sigma^{\rm spe}))): \\
\bullet \ \sigma \in \Irr(W(\Phi_{x, \chi_i}^{Q,n})),\\
\bullet \  \mathcal{O}_{\text{\rm Spr}}(j^{\vartheta^{\vee}(W)}_{\vartheta^{\vee}(W(\Phi^{Q,n}_{x,\chi_i}))}(\vartheta^{\vee}(\sigma) \otimes \varepsilon)^{\text{\rm spe}}) \gest \mca{O}^\vee 
	\end{array}
	\right\}.
\end{equation}
For every $x \in \mca{C}$, we also set
\begin{equation} \label{E:Dx}
\mca{D}(\mca{O}^\vee)_x: = \max \bigcup_{0\lest i < \val{\msc{X}}} \mca{D}(\mca{O}^\vee)_{x,\chi_i}
\end{equation}
and define
\begin{equation} \label{E:D}
\mca{D}(\mca{O}^\vee):= \max \bigcup_{x\in \mca{C}} \mca{D}(\mca{O}^\vee)_x.
\end{equation}
We also define
\begin{equation} \label{D:Dflat}
\mca{D}^\flat(\mca{O}^\vee):=\max_{\substack{ x\in \mca{C} \\ 0\lest i < \val{\msc{X}}} }
\left\{
	\begin{array}{cc}
	\text{\rm Sat}^{\mathbf{G}(\wt{F})}_{\mathbf{L}_x^F(\wt{F})}(\mathcal{O}_\text{\rm Spr}^{\mbf{L}_x^F}(j^{W_x}_{W(\Phi^{Q,n}_{x,\chi_i})}(\sigma^{\rm spe}))): \\
\bullet \ \sigma \in \Irr(W(\Phi_{x, \chi_i}^{Q,n})),\\
\bullet \  \mathcal{O}_{\text{\rm Spr}}(j^{\vartheta^{\vee}(W)}_{\vartheta^{\vee}(W(\Phi^{Q,n}_{x,\chi_i}))}(\vartheta^{\vee}(\sigma) \otimes \varepsilon)^{\text{\rm spe}}) = \mca{O}^\vee 
	\end{array}
	\right\}
\end{equation}

\begin{thm} \label{T:main}
For every $\mca{O}^\vee \in \mca{N}(\wt{G}^\vee)$, one has
$\mca{D}(\mca{O}^\vee) \lest d^{(n)}_{BV, G}(\mca{O}^\vee)$. Moreover, if $G$ is of type $A$ or exceptional type, then the equality holds.
\end{thm}

We will prove Theorem \ref{T:main} for classical types from  a case-by-case analysis in \S \ref{SS:Pf-cla}, based on some combinatorial preparation given in \S \ref{SS:comb}. We treat exceptional groups in \S \ref{SS:exc}. In fact, a small reduction is given below.

\begin{rmk} \label{R:minPhi}
Suppose we have $\Phi^{Q,n}_{x,\chi_i} \subseteq \Phi^{Q,n}_{x,\chi_j}$ for some $\chi_i, \chi_j, 0\lest i, j \lest k$. Let $\sigma \in \Irr(W(\Phi^{Q,n}_{x,\chi_j}))$ and consider an irreducible $\sigma_0 \subseteq \sigma^{\text{spe}}|_{W(\Phi^{Q,n}_{x,\chi_i})}$. Assume that $ \Phi^{Q,n}_{x,\chi_j}$ does not contain $E_7$ or $E_8$. Then $\vartheta^\vee(\sigma)^{\text{spe}}\otimes \varepsilon=(\vartheta^\vee(\sigma)\otimes \varepsilon)^{\text{spe}}$. For a representation $\Pi$ of a Weyl group, we write $\mca{O}_{\rm Spr}(\Pi):=\max\set{\mca{O}_{\rm Spr}(\pi): \pi \text{ irred. and } \pi \subseteq \Pi}$.
Thus, if $\mathcal{O}_{\text{\rm Spr}}(j^{\vartheta^{\vee}(W)}_{\vartheta^{\vee}(W(\Phi^{Q,n}_{x,\chi_j}))}(\vartheta^\vee(\sigma) \otimes \varepsilon)^{\text{\rm spe}}) \gest \mca{O}^\vee$ with $\mca{O}^\vee \in \mca{N}(\wt{G}^\vee)$, then by \cite[Proposition 4.3]{AcAu07} we have 
\begin{equation}
\mathcal{O}_{\text{Spr}}(j^{\vartheta^{\vee}(W)}_{\vartheta^{\vee}(W(\Phi^{Q,n}_{x,\chi_i}))}(\vartheta^\vee(\sigma_0) \otimes \varepsilon)^{\text{spe}})  \gest \mathcal{O}_{\text{Spr}}(\text{Ind}^{\vartheta^{\vee}(W)}_{\vartheta^{\vee}(W(\Phi^{Q,n}_{x,\chi_i}))}(\vartheta^\vee(\sigma_0) \otimes \varepsilon))  \gest \mca{O}^\vee.
\end{equation}
On the other hand, we also have 
\begin{equation}
\mathcal{O}_{\text{Spr}}(j^{W_x}_{W(\Phi^{Q,n}_{x,\chi_i})}\sigma_0^{\text{spe}}) \gest \mathcal{O}_{\text{Spr}}(\text{Ind}^{W_x}_{W(\Phi^{Q,n}_{x,\chi_i})}\sigma_0)\gest \mathcal{O}_{\text{Spr}}(j^{W_x}_{W(\Phi^{Q,n}_{x,\chi_j})}\sigma^{\text{spe}}).
\end{equation}
Hence, in view of \eqref{E:D}, to compute $\mca{D}(\mca{O}^\vee)$, we only need to consider those minimal $\Phi^{Q,n}_{x,\chi_i}$ and also the ones containing $E_7$ or $E_8$. 

Moreover, if the image of $d^{(n)}_{BV,G}$ is always the regular orbit (this is true if $n$ is larger than a specific constant determined by $\mbf{G}$), it is obvious that Theorem \ref{T:main} is true. Thus, for fixed $G$, there is only a finite amount of computation needed, and in particular it is so for all the exceptional groups.
\end{rmk}

\subsection{Some combinatorics} \label{SS:comb}
In the remaining of this section, we assume that the root system $\Phi$ of $\mbf{G}$ is irreducible.

First, we describe $\Phi_{x, \chi_i}^{Q,n}$ for hyperspecial  $x\in \mca{C}$ for each classical type. To state our result, we fix some convention here, by referring to \S \ref{SS:dBVn} for the labelling of the  extended Dynkin diagram of type $B_r,C_r$ and $D_r$. Thus, by $D_2$, we mean the subsystem given by nodes $\set{\alpha_0, \alpha_1}$ in type $B_r$ or by $\set{\alpha_0,\alpha_1}$, $\set{\alpha_{r-1}, \alpha_r}$ in type $D_r$. By $D_3$, we mean the subsystem given by the nodes $\set{\alpha_0,\alpha_1,\alpha_2}$ in type $B_r$ or $\set{\alpha_0,\alpha_1,\alpha_2}$, $\set{\alpha_{r-2}, \alpha_{r-1}, \alpha_r}$ in type $D_r$. By $B_1$ or $C_1$, we mean the subsystem given by the node $\set{\alpha_r}$ in type $B_r$ or by $\set{\alpha_0}$, $\set{\alpha_r}$ in type $C_r$. Also, for every root system $\Phi'$ we set
$$\msc{A}(\Phi'):=\set{\text{irreducible components in $\Phi'$ of type $A$}}$$
and 
$$\msc{E}(\Phi'):=\set{\text{irreducible components in $\Phi'$ not of type $A$}}.$$

\begin{lm} \label{L:Phixchi} 
For $G$ of classical type and hyperspecial $x\in \mca{C}$, the root subsystem $\Phi_{x,\chi_i}^{Q,n}$ is given as follows.
\begin{enumerate}
    \item[(i)] For $G$ of type $A_r$ we have
    $$\Phi^{Q,n}_{x,\chi_i} \simeq A_{r_1-1}+ A_{r_2-1}+...+ A_{r_k-1}$$
with $ \sum_{i=1}^{k}r_i=r+1, r_i \gest 1$ and $k=\val{\msc{A}(\Phi_{x, \chi_i}^{Q,n})} \lest \nn$. Moreover, every such root subsystem $\sum_{j=1}^k A_{r_j-1}$ is equal to $\Phi_{x, \chi_i}^{Q,n}$ for some $x$ and $\chi_i$.
    \item[(ii)] For $G$ of type $B_r$ and $\nn$ odd, we have
    $$\Phi^{Q,n}_{x,\chi_i} \simeq A_{r_1-1}+ A_{r_2-1}+...+ A_{r_{k-1}-1} + B_{r_k}$$ with $\sum_{i=1}^{k}r_i=r$ and $2\cdot \val{\msc{A}(\Phi_{x, \chi_i}^{Q,n})} + \val{\msc{E}(\Phi_{x, \chi_i}^{Q,n})} \lest \nn$.
For $\nn$ even we have 
$$\Phi^{Q,n}_{x,\chi_i} \simeq B_{r_1}+ A_{r_2-1}+...+ A_{r_{k-1}-1} + B_{r_k}$$ with $ \sum_{i=1}^{k} r_i=r$ and $2\cdot \val{\msc{A}(\Phi_{x, \chi_i}^{Q,n})} + \val{\msc{E}(\Phi_{x, \chi_i}^{Q,n})}\lest \nn$. 
    \item[(iii)] For $G$ of type $C_r$ and $\nn$ odd, we have
    $$\Phi^{Q,n}_{x,\chi_i} \simeq A_{r_1-1}+ A_{r_2-1}+...+ A_{r_{k-1}-1} + C_{r_k}$$ with $ \sum_{i=1}^{k} r_i=r$ and $2\cdot \val{\msc{A}(\Phi_{x, \chi_i}^{Q,n})} + \val{\msc{E}(\Phi_{x, \chi_i}^{Q,n})}\lest \nn$. For $\nn$ even we have
    $$\Phi^{Q,n}_{x,\chi_i} \simeq D_{r_1}+ A_{r_2-1}+...+ A_{r_{k-1}-1} + C_{r_k}$$ with $\sum_{i=1}^{k}r_i=r$ and $2\cdot \val{\msc{A}(\Phi_{x, \chi_i}^{Q,n})} + \val{\msc{E}(\Phi_{x, \chi_i}^{Q,n})}\lest \nn/2$.
    \item[(iv)] For $G$ of type $D_r$ we have 
    $$\Phi^{Q,n}_{x,\chi_i} \simeq D_{r_1}+ A_{r_2-1}+...+ A_{r_{k-1}-1}+ D_{r_k}$$
    with $ \sum_{i=1}^{k} r_i=r$ and $2\cdot \val{\msc{A}(\Phi_{x, \chi_i}^{Q,n})} + \val{\msc{E}(\Phi_{x, \chi_i}^{Q,n})}\lest \nn$.
\end{enumerate}
\end{lm}
\begin{proof}
First, as already mentioned in \S \ref{SS:RepIH}, it follows from \cite[Lemma 4.5]{KOW} that $\Phi_{x, \chi_i}^{Q,n}$ is a dual pseudo-Levi subsystem of $\Phi$. 
At the same time, $\vartheta(\Phi_{x,\chi_i}^{Q,n})$ is a pseudo-Levi subsystem of $\vartheta(\Phi)=\Phi_{Q,n}$. This shows that $\Phi_{x, \chi_i}^{Q,n}$ has to be of such a form in all cases (i)--(iv). 

Since $x$ is a hyperspecial point, the discussion in \cite[\S 4]{KOW} gives restrictions on the number of nodes that one can omit from the extended Dynkin diagram of $\vartheta(\Phi)$, in order to obtain the pseudo-Levi subsystem $\vartheta(\Phi_{x, \chi_i}^{Q,n})$. For example, for type $A_r$ one can remove at most $\nn$-many points from the extended Dynkin diagram of $\vartheta(\Phi)$. This gives the upper bound $\nn$ for the number of components $k$ in $\Phi_{x,\chi_i}^{Q,n}$. For other classical types, it follows from analogous argument, as elaborated in \cite[\S 4]{KOW}.
\end{proof}

In the following proposition, we record some results regarding certain representations of $W$ arising from the truncated induction. For a partition $\lambda$ of type $X \in \set{A, B, C, D}$, if $\lambda$ is not a very even partition (only of type $D$), then as before we set $E_\lambda:={\rm Spr}^{-1}(\lambda, \mbm{1}) \in \Irr(W_X)$. If $\lambda$ is very even in the case of type $D$, we also use $E_{\lambda}$ to denote either choice of the representation associated with $\lambda$; this suffices for our later purpose.

\begin{prop} \label{P:j-induction}
Let $W^{\prime} \subseteq W$ be the Weyl group of a pseudo Levi subsystem of $\Phi$.
\begin{enumerate}
    \item[(i)] Let $W$ be of type $A_r$ and $W^{\prime} \subseteq W $ be of type $A_{r_1-1}+ A_{r_2-1}+...+ A_{r_k-1}$ with $ \sum_{i=1}^{k}r_i=r+1 $. Let $\lambda_i$ be a partition of $r_i$. Then $$\mathcal{O}_{\rm Spr}(j^{W}_{W^\prime}(\mathop{\boxtimes}\limits_{i} E_{\lambda_i}))=(\lambda_1^{\top}\cup\lambda_2^{\top}\cup...\cup\lambda_k^{\top})^{\top}.$$
    \item[(ii)] Let $W$ be of type $B_r$ and let $W^{\vee}$ be its dual group of type $C_r$. Let $W^{\prime} \subseteq W^{\vee} $ be a subgroup of type $C_{r_1} + A_{r_2-1}+...+ A_{r_{k-1}-1} + C_{r_k}$. We also view $W^{\prime}$ as a subgroup of $W$ by the canonical isomorphism between $W$ and $W^{\vee}$ (i.e., the isomorphism $\vartheta^\vee$ in \eqref{E:Widen} for $n=1$). Let $\lambda_1$ and $\lambda_k$ be special $C$-partitions of $2r_1$ and $2r_k$, respectively. Let $\lambda_i$ be a partition of $A_{r_i-1}$ for $2 \lest i \lest k-1$. Then
      \begin{align*}
    	\mathcal{O}_{\rm Spr}(j^{W}_{W^\prime}(\mathop{\boxtimes}\limits_{i} E_{\lambda_i})) & =(\lambda_1^{\top}\cup(\lambda_2^{\top}\cup\lambda_2^{\top})\cup...\cup(\lambda_{k-1}^{\top}\cup\lambda_{k-1}^{\top})\cup(\lambda_k^+{}_B)^{\top})^{\top}{}_B\\
    	& =((\lambda_1^+{}_B)^{\top}\cup(\lambda_2^{\top}\cup\lambda_2^{\top})\cup...\cup(\lambda_{k-1}^{\top}\cup\lambda_{k-1}^{\top})\cup\lambda_k^{\top})^{\top}{}_B
    \end{align*}
    \item[(iii)] Let $W$ be of type $C_r$ and let $W^{\vee}$ be its dual group of type $B_r$. Let $W^{\prime} \subseteq W^{\vee} $ be a subgroup of type $D_{r_1} + A_{r_2-1}+...+ A_{r_{k-1}-1} + B_{r_k}$. We view $W^{\prime}$ as a subgroup of $W$ by the canonical isomorphism between $W$ and $W^{\vee}$. Let $\lambda_1$ be a special $D$-partition of $2r_1$ 
    and $\lambda_k$ be a special $B$-partition of $2r_k+1$. Let $\lambda_i$ be a partition of $A_{r_i-1}$ for $2 \lest i \lest k-1$. Then $$\mathcal{O}_{\rm Spr}(j^{W}_{W^\prime}(\mathop{\boxtimes}\limits_{i} E_{\lambda_i}))=(\lambda_1^{\top}{}_D\cup(\lambda_2^{\top}\cup\lambda_2^{\top})\cup...\cup(\lambda_{k-1}^{\top}\cup\lambda_{k-1}^{\top})\cup(\lambda_k^-{}_C)^{\top})^{\top}{}_C.$$
    \item[(iv)] Let $W$ be of type $D_r$ and $W^{\prime} \subseteq W $ be a subgroup of type $D_{r_1} + A_{r_2-1}+...+ A_{r_{k-1}-1} + D_{r_k}$. Let $\lambda_1$ and $\lambda_k$ be special $D$-partitions of $2r_1$ and $2r_k$. Let $\lambda_i$ be a partition of $A_{r_i-1}$ for $2 \lest i \lest k-1$. Then 
    \begin{align*}
        \mathcal{O}_{\rm Spr}(j^{W}_{W^\prime}(\mathop{\boxtimes}\limits_{i} E_{\lambda_i})) & =(\lambda_1^{\top}\cup(\lambda_2^{\top}\cup\lambda_2^{\top})\cup...\cup(\lambda_{k-1}^{\top}\cup\lambda_{k-1}^{\top})\cup(\lambda_{k}^{\top})_D)^{\top}{}_D\\
        & =((\lambda_1^{\top})_D\cup(\lambda_2^{\top}\cup\lambda_2^{\top})\cup...\cup(\lambda_{k-1}^{\top}\cup\lambda_{k-1}^{\top})\cup\lambda_{k}^{\top})^{\top}{}_D
    \end{align*}
\end{enumerate}    
\end{prop}
\begin{proof}
For (i), it follows easily from \cite[Proposition 11.4.1]{Car}. For (ii), if $r_1=0$, then the group $\mathbf{M}(\ol{F})$ associated with $W^{\prime}$ is a Levi subgroup of $\mathbf{G}(\ol{F})$. Let $\mathcal{O}_M$ be the nilpotent orbit given by $\lambda_2^{\top},...,\lambda_{k-1}^{\top},\lambda_k^+{}_B{}^{\top}$. Then by the property of linear Barbasch--Vogan duality, $\mathcal{O}_{\rm Spr}(j^{W}_{W^\prime}(\mathop{\boxtimes}\limits_{i} E_{\lambda_i})) $ is equal to $ d_{\text{BV}}(\text{Sat}^{\mathbf{G}(\ol{F})}_{\mathbf{M}(\ol{F})}(\mathcal{O}_M)) $, which is exactly the result in (ii). If $r_1 \neq 0$, then it follows from \cite[(7) and (8)]{Som01} and the transitivity of $j$-induction. The statement in (iii) and (iv) is shown with a similar argument as for (ii).
\end{proof}

\begin{lm} \label{lm:Ar}
Let $z\in \N_{\gest 1}$ and let $\mfr{p}$ be a partition of $r+1$.
\begin{enumerate}
    \item[(i)]  We have   
  \begin{equation}\label{lm:Ar(1)}
   d_{com}^{(z)}(\mfr{p}) =
\max   \left\{
	\begin{array}{cc}
	(\lambda_1^{\top}\cup\lambda_2^{\top}\cup...\cup\lambda_z^{\top})^{\top}: \\
\bullet \ \lambda_i, 1\lest i \lest z \text{ are partitions, possibly zero},\\
\bullet \  \lambda_1\cup\lambda_2\cup...\cup\lambda_z=\mfr{p}^{\top}
	\end{array}
	\right\}.
\end{equation}

    \item[(ii)] If $\mfr{p}_1^{\top}\cup \mfr{p}_2^{\top}=\mfr{p}^{\top}$ for two partitions $\mfr{p}_1, \mfr{p}_2$, then 
    $$d^{(z)}_{com}(\mfr{p}_1)\cup d^{(z)}_{com}(\mfr{p}_2)\lest d^{(z)}_{com}(\mfr{p}).$$
\end{enumerate}
\end{lm}
\begin{proof} For (i), we write
$$(l_1,...,l_k):=\text{\rm max}\{(\lambda_1^{\top}\cup\lambda_2^{\top}\cup...\cup\lambda_z^{\top})^{\top}: \lambda_1\cup\lambda_2\cup...\cup\lambda_z=\mfr{p}^{\top}\}$$ with $ l_1\gest ... \gest l_k $ and 
$$d^{(z)}_{com}(\mfr{p})=(l_1^{\prime},...,l_m^{\prime})$$ with $ l_1^{\prime}\gest ... \gest l_m^{\prime} $. However, viewing $\mfr{p}=(p_1, ..., p_i, ..., p_s)$ as a Young diagram with $p_i$ being the length of the $i$-th row,  one can check easily that \begin{equation} \label{eq:BV_dual_Ar}
l_i=l^{\prime}_i=\sum_{j=z(i-1)+1}^{zi}(\text{length of the }j \text{-th column of } \mfr{p})
\end{equation}
for every $i$. This gives (i). The equality  \eqref{eq:BV_dual_Ar} also gives (ii).
\end{proof}

\subsection{Proof of Theorem \ref{T:main} for classical types} \label{SS:Pf-cla}
Before giving a case by case argument, we make a crucial observation, to be used in the proof. 

For $x\in \mca{C}$, suppose $\Phi_x = \prod_j \Phi_j$ where $\Phi_j \subseteq \Phi_x$ is an irreducible component. It gives a decomposition
$$\Phi_{x, \chi_i}^{Q,n} = \prod \Phi_{j, \chi_i}^{Q,n},$$
where $\Phi_{j, \chi_i}^{Q,n} \subseteq \Phi_j$. We notice that the projection $x_j$ say of $x-y_i$ in the $\R$-space spanned by $\Phi_j^\vee$ is naturally a coweight for $\Phi_j$. Thus 
$$\Phi_{j, \chi_i}^{Q,n} = (\Phi_j)_{x_j, 0}^{Q,n},$$
that is, $\Phi_{j, \chi_i}^{Q,n}$ can be reduced to the hyperspecial case for the root system $\Phi_j$ with respect to the hyperspecial $x_j$ and the trivial character (in the place of $\chi_i$). This enables us to deal with $\Phi_{x, \chi_i}^{Q,n}$ for general $x\in \mca{C}$ by applying Lemma \ref{L:Phixchi} to each of the component $\Phi_{j, \chi_i}^{Q,n}$. 

Now we proceed to the proof of Theorem \ref{T:main} for classical types.

\subsubsection{Type $A_r$}
For $G$ of type $A_r$, by Lemma \ref{L:Phixchi}, the root subsystem $\Phi^{Q,n}_{x,\chi_i}$ is of the form 
$$A_{r_1-1}+ A_{r_2-1}+...+ A_{r_k-1}$$
satisfying that $ \sum_{i=1}^{k} r_i=r+1 $ and $k\lest \nn$; also, every such root subsystem is isomorphic to $\Phi_{x, \chi_i}^{Q,n}$ for some $x \in \mca{C}$ and $\chi_i$.  It follows from Proposition \ref{P:j-induction}(i), Lemma \ref{lm:Ar}(1), and the definition of $d_{BV,G}^{(n)}$ that
$$\mca{D}(\mca{O}_\mfr{p}^\vee) = \max \set{d_{BV, G}^{(n)}(\mca{O}_{\mfr{p}'}^\vee): \ \mfr{p}' \gest \mfr{p}}.$$
However, the map $d^{(n)}_{BV, G}$ is order-reversing \cite{GLLS}, i.e., $d_{BV, G}^{(n)}(\mca{O}_{\mfr{p}'}^\vee) \lest d_{BV, G}^{(n)}(\mca{O}_{\mfr{p}}^\vee)$ for all $\mfr{p}'\gest \mfr{p}$. This shows that
$$\mca{D}(\mca{O}_\mfr{p}^\vee) = d_{BV, G}^{(n)}(\mca{O}_{\mfr{p}}^\vee)$$
 for type $A_r$.

\subsubsection{Type $B_r$}
For $G$ of type $B_r$, we first consider the case $\nn$ is odd. In this case, $\wt{G}^\vee$ is of type $C_r$.  Since $\vartheta(\Phi^{Q,n}_{x,\chi_i})$ is  a pseudo-Levi subsystem of $\Phi_{Q,n}$, the root subsystem $\Phi^{Q,n}_{x,\chi_i}$ is of the form 
$$D_{r_1}+ A_{r_2-1}+...+ A_{r_{k-1}-1} + B_{r_k}.$$
Since $\Phi_x$ is of the type $D_{m} + B_{r-m}$, by Lemma \ref{L:Phixchi} and the discussion above in this subsection, we can assume that
$$\Phi^{Q,n}_{x,\chi_i} \simeq ( \underbrace{D_{r^{\prime}_1}+ A_{r^{\prime}_2-1}+...+ A_{r^{\prime}_{k_{1}}-1}}_{\subseteq D_m})+ (\underbrace{A_{r^{\prime\prime}_1-1}+ A_{r^{\prime\prime}_2-1}+...+ A_{r^{\prime\prime}_{k_{2}-1}-1} + B_{r^{\prime\prime}_{k_2}} }_{\subseteq B_{r-m}})$$
with $r_1'=r_1, r_{k_2}''=r_k$.
By the observation given at the beginning of this subsection and Lemma \ref{L:Phixchi}, we get 
$$2k_1-1  \lest \nn \text{ and } 2k_2-1\lest \nn.$$

Now, let $\lambda_1,...,\lambda_{k_1}$ be special partitions of  $D_{r_{k_1}'}$ and $A_{r_i'-1}$'s. Let $\mu_1,...,\mu_{k_2}$ be special partitions for $A_{r_i''-1}$'s and $B_{r_{k_2}''}$.
Let $\mca{O}_\mfr{p}^\vee \in \mca{N}(\wt{G}^\vee)$ be an orbit with $\mfr{p}$ being a symplectic partition.
 By Proposition \ref{P:j-induction}, one has
$$\begin{aligned}
& \mca{D}(\mca{O}^\vee_\mfr{p})_{x,\chi_i} \\
= & \text{max}\left\{
\begin{array}{cc}
(\lambda^{\top}_1\cup\lambda^{\top}_2\cup\lambda^{\top}_2\cup...\lambda^{\top}_{k_1}\cup\lambda^{\top}_{k_1})^{\top}_{\ D}\cup (\mu^{\top}_1\cup\mu^{\top}_1\cup...\mu^{\top}_{k_2-1}\cup\mu^{\top}_{k_2-1}\cup\mu^{\top}_{k_2})^{\top}_{\ B}: \\
\bullet \ C_\mfr{p}((\lambda_i)_i, (\mu_j)_j)
\end{array}
\right\}
\end{aligned}
$$
where the condition $C_\mfr{p}((\lambda_i)_i, (\mu_j)_j)$ is given by$$ (\lambda_1\cup\lambda_2\cup\lambda_2\cup...\cup\lambda_{k_1}\cup\lambda_{k_1}\cup\mu_1\cup\mu_1\cup...\cup\mu_{k_2-1}\cup\mu_{k_2-1}\cup (\mu_{k_2}^-)_C)^{\top}_{\ C} \gest \mfr{p}.$$

We set 
$$\mfr{p}_1:=\lambda_1\cup\lambda_2\cup\lambda_2\cup...\cup\lambda_{k_1}\cup\lambda_{k_1}, \quad \mfr{p}_2:=\mu_1\cup\mu_1\cup...\cup\mu_{k_2-1}\cup\mu_{k_2-1}\cup (\mu_{k_2}^{-})_C$$
and
$$\mfr{q}_1:=\lambda^{\top}_1\cup\lambda^{\top}_2\cup\lambda^{\top}_2\cup...\lambda^{\top}_{k_1}\cup\lambda^{\top}_{k_1}, \quad \mfr{q}_2:=\mu^{\top}_1\cup\mu^{\top}_1\cup...\mu^{\top}_{k_2-1}\cup\mu^{\top}_{k_2-1}\cup((\mu_{k_2}^{-})_C)^{\top}.$$
Since $2k_1-1 \lest \nn$, we have $\mfr{q}_1^\top \lest d_{com}^{(\nn)}(\mfr{p}_1^\top)$ by Lemma \ref{lm:Ar}(i). Similarly, $\mfr{q}_2^\top \lest d_{com}^{(\nn)}(\mfr{p}_2^\top)$.
This gives
$$ \mfr{q}_1^\top \cup \mfr{q}_2^\top \lest d_{com}^{(\nn)}(\mfr{p}_1^\top) \cup d_{com}^{(\nn)}(\mfr{p}_2^\top) \lest d_{com}^{(\nn)}((\mfr{p}_1 \cup \mfr{p}_2)^\top),$$
where the last inequality follows from Lemma \ref{lm:Ar}(ii). But $d_{com}^{(\nn)}$ is order-reversing, and thus $\mfr{q}_1^\top \cup \mfr{q}_2^\top \lest d_{com}^{(\nn)}(\mfr{p})$. Since $\mu_{k_2}$ is a special partition, we have $\mu_{k_2}=\mu_{k_2}^-{}_C{}^+{}_B$; this implies that $\mca{D}(\mca{O}_\mfr{p}^\vee)_{x,\chi_i}$ is bounded above by $d^{(\nn)}_{com}(\mfr{p})^+$ and thus by $d^{(\nn)}_{com}(\mfr{p})^+{}_B=d_{BV,G}^{(n)}(\mfr{p})$.

For type $B_r$ and $\nn$ even, we see that $G_{Q,n}$ and $G^{\vee}$ are both of type $C_r$. In this case the root subsystem $\Phi^{Q,n}_{x,\chi_i}$ is of the form 
$$B_{r_1}+ A_{r_2-1}+ A_{r_3-1}+...+ A_{r_{k-1}-1}+ B_{r_k}.$$
Let $\Phi_x$ be of the type $D_{m} + B_{r-m}$. In either $\Phi':=D_{m}$ or $B_{r-m}$, we see that 
$$2 \cdot \val{\msc{A}(\Phi^{Q,n}_{x,\chi_i} \cap \Phi')}+  \val{\msc{E}(\Phi^{Q,n}_{x,\chi_i} \cap \Phi')} \lest \nn.$$ 
By similar argument as above, the theorem holds for type $B_r$ and $\nn$ even.

\subsubsection{Type $C_r$}
For type $C_r$ and odd $\nn$, the analysis is almost the same as the type $B_r$ case, and thus we omit the details. For type $C_r$ and $\nn$ even, since $G_{Q,n}$ is of type $B_r$, the root subsystem $\Phi^{Q,n}_{x,\chi_i}$ is of the form 
$$D_{r_1}+ A_{r_2-1}+...+ A_{r_{k-1}-1} + C_{r_k}.$$
Let $\Phi_x$ be of type $C_{m} \times C_{r-m}$ for some $m$. We assume that
$$\Phi^{Q,n}_{x,\chi_i} \simeq ( \underbrace{D_{r^{\prime}_1}+ A_{r^{\prime}_2-1}+...+ A_{r^{\prime}_{k_{1}}-1}}_{\subseteq C_m}) + (\underbrace{A_{r^{\prime\prime}_1-1}+ A_{r^{\prime\prime}_2-1}+...+ A_{r^{\prime\prime}_{k_{2}-1}-1} + C_{r^{\prime\prime}_{k_2}} }_{\subseteq C_{r-m}})$$
with $r_1'=r_1$ and $r_{k_2}''=r_k$.
We elaborate on this case, since other cases (such as when $D_{r_1}$ and $C_{r_k}$ are both in $C_m$) 
are also verified by similar argument as below.

Let $\lambda_1,...,\lambda_{k_1}$ be special partitions of $D_{r_{k_1}'}$ and $A_{r_i'-1}$'s. Let $\mu_1,...,\mu_{k_2}$ be special partitions for $A_{r_i''-1}$'s and $C_{r_{k_2}''}$.
Let $\mca{O}_\mfr{p}^\vee \in \mca{N}(\wt{G}^\vee)$ be an orbit with $\mfr{p}$ being a symplectic partition.
By Proposition \ref{P:j-induction}, we have
$$\begin{aligned}
	& \mca{D}(\mca{O}^\vee_\mfr{p})_{x,\chi_i} \\
	= & \text{max}
	\left\{
	\begin{array}{cc}
	(\lambda^{\top}_1{}_D\cup\lambda^{\top}_2\cup\lambda^{\top}_2\cup...\lambda^{\top}_{k_1}\cup\lambda^{\top}_{k_1})^{\top}_{\ C}\cup (\mu^{\top}_1\cup\mu^{\top}_1\cup...\mu^{\top}_{k_2-1}\cup\mu^{\top}_{k_2-1}\cup\mu^{\top}_{k_2})^{\top}_{\ C}: \\
	\bullet \  C_\mfr{p}((\lambda_i)_i, (\mu_j)_j)
	\end{array}
	\right\}
\end{aligned}
$$
where the condition $C_\mfr{p}((\lambda_i)_i, (\mu_j)_j)$ is given by$$ (\lambda_1\cup\lambda_2\cup\lambda_2\cup...\cup\lambda_{k_1}\cup\lambda_{k_1}\cup\mu_1\cup\mu_1\cup...\cup\mu_{k_2-1}\cup\mu_{k_2-1}\cup \mu_{k_2})^{\top}_{\ C} \gest \mfr{p}.$$
By Lemma \ref{L:Phixchi}, one has
$$2k_1-1  \lest \nn/2 \text{ and } 2k_2-1\lest \nn/2.$$
Following the same line of reasoning as in type $B_r$, we get $$ (\lambda^{\top}_1\cup\lambda^{\top}_2\cup\lambda^{\top}_2\cup...\lambda^{\top}_{k_1}\cup\lambda^{\top}_{k_1})^{\top}\cup (\mu^{\top}_1\cup\mu^{\top}_1\cup...\mu^{\top}_{k_2-1}\cup\mu^{\top}_{k_2-1}\cup\mu^{\top}_{k_2})^{\top}\lest d_{com}^{(\nn/2)}(\mfr{p}). $$

Since $\lambda_1$ is a $D$-special partition, the partition $\lambda_1^{\top}$ is of type $C$. Then by \cite[Corollary 3.16]{GLLS}, $\lambda_1{}^\top{}_D{}^{\top}\lest \lambda_1^{+-}$. We see $\mca{D}(\mca{O}_\mfr{p}^\vee)_{x,\chi_i}$ is bounded above by $d^{(\nn/2)}_{com}(\mfr{p})^{+-}$ and thus by $d^{(\nn/2)}_{com}(\mfr{p})^{+-}{}_C$. 
This finishes the proof if $\nn/2$ is odd.
We observe that when $\nn/2$ is even, the partition $d^{(\nn/2)}_{com}(\mfr{p})$ contains only even numbers; therefore,  $d^{(\nn/2)}_{com}(\mfr{p})^{+-}{}_C=d^{(\nn/2)}_{com}(\mfr{p})_C$ in this case and an upper bound of $\mca{D}(\mca{O}_\mfr{p}^\vee)_{x,  \chi_i}$ is also given by $d_{BV,G}^{(n)}(\mfr{p})$.

\subsubsection{Type $D_r$}
For $G$ of type $D_r$,  following the same method as above, we can obtain the inequality $\mca{D}(\mca{O}^\vee_\mfr{p}) \lest d^{(\nn)}_{com}(\mfr{p})_D = d_{BV,G}^{(n)}(\mca{O}_\mfr{p}^\vee)$ at the partition level.

Now we consider $D_r$ with $r$  even, and $\mca{O}^\vee_\mfr{p}$ is a very even orbit. We need to show that if the equality $\mca{D}(\mca{O}^\vee_\mfr{p})  = d_{BV,G}^{(n)}(\mca{O}_\mfr{p}^\vee)$ holds at the partition level, then $\mca{D}(\mca{O}^\vee_\mfr{p})$ has the same labelling as $d_{BV,G}^{(n)}(\mca{O}_\mfr{p}^\vee)$ as well. We observe that for a very even partition $\mfr{p}=(2l_1,2l_1,...,2l_m,2l_m)$ with $l_1\gest ...\gest l_m$, if $\mfr{p}'$ is a partition of $D_r$ such that $\mfr{p}' > \mfr{p}$,  then we get
$$\mfr{p}^{\prime} \gest (2l_1,2l_1,...,2l_m+1,2l_m-1)=:\mfr{p}^\triangle.$$
But $\mfr{p}^\triangle$ is not a very even partition, and thus neither is $d^{(n)}_{BV, G}(\mfr{p}^\triangle)$. Hence, if the partition underlying $\mca{D}(\mca{O}^\vee_\mfr{p})$ is equal to that of $d^{(n)}_{BV, G}(\mfr{p})$, then 
$$\mca{D}(\mca{O}^\vee_\mfr{p})=d_{BV, G}^{(n)}(\mfr{p})> d_{BV, D}^{(n)}(\mfr{p}^\triangle) \gest d_{BV, G}^{(n)}(\mfr{p}^{\prime}) \gest \mca{D}(\mca{O}^\vee_{\mfr{p}^{\prime}}). $$
This shows that for $\mfr{p}$ very even, we have
$$\mca{D}(\mfr{p}^\vee) = \mca{D}^\flat(\mfr{p}^\vee).$$
We know $\Phi^{Q,n}_{x,\chi_i}$ is of type $D_{r_1} + A_{r_2-1}+...+ A_{r_{k-1}-1} + D_{r_k}$. Note that
$$\mathcal{O}_\text{\rm Spr}(j^{W_x}_{W(\Phi^{Q,n}_{x,\chi_i})}\sigma^{\rm spe}) \text{ and } \mathcal{O}_{\text{\rm Spr}}(j^{W}_{W(\Phi^{Q,n}_{x,\chi_i})}(\sigma \otimes \varepsilon)^{\text{\rm spe}})$$ are both very even only if $r_1,r_k$ are even and $4|r_i$ for $2\lest r \lest k-1$. Then when $4|r$, we have $4|r_1,r_k$ or $4\nmid r_1,r_2$. On the other hand, when $4\nmid r$, we have $4|r_1,4\nmid r_2$ or $4\nmid r_1,4|r_2$. For type $D_r$, if $\pi$ is a special representation corresponding to a very even partition, then the nilpotent orbit associated with $\pi \otimes\varepsilon$ has the same labelling I or II (following the convention in \cite{BMW25}) as that with $\pi$  when $4|r$, and has different labelling when $4\nmid r$. We write
$$\sigma^{\text{spe}}=\sigma_1\boxtimes...\boxtimes \sigma_k \in \Irr(W(\Phi_{x,\chi_i}^{Q,n})),$$
where $\sigma_i$ is a special irreducible representation of the Weyl group of the $i$-th component in $D_{r_1} +A_{r_2-1}+...+ A_{r_{k-1}-1} + D_{r_k}$. This gives 
$$(\sigma\otimes\varepsilon)^{\text{spe}}=\sigma^{\text{spe}}\otimes \varepsilon=(\sigma_1\otimes\varepsilon)\boxtimes...\boxtimes (\sigma_k\otimes\varepsilon).$$
We see that when $4|r$, the nilpotent orbits associated with $\sigma_1$ and $\sigma_1\otimes\varepsilon$ have the same labellings if and only if the orbits associated with $\sigma_k$ and $\sigma_k\otimes\varepsilon$ have the same labellings.  On the other hand, if $4\nmid r$, then the orbits associated with $\sigma_1$ and $\sigma_1\otimes\varepsilon$ have the same labellings if and only if the orbits associated with $\sigma_k$ and $\sigma_k\otimes\varepsilon$ possess different labellings.

In general, $W_x$ is of type $D_{m_1}+ D_{m_2} \subseteq D_r$. Let $\mfr{p}_1 \times \mfr{p}_2$ be the partition underlying $\mathcal{O}_{\text{\rm Spr}}(j^{W_x}_{W(\Phi^{Q,n}_{x,\chi_i})}(\sigma \otimes \varepsilon)^{\text{\rm spe}})$. Assume that $\mfr{p}_1$ and $\mfr{p}_2$ are both very even. We divide the four possible labellings of $\mfr{p}_1 \times \mfr{p}_2$ into two groups $\set{\text{(I,I),(II,II)}}$ and $\set{\text{(I,II),(II,I)}}$ and call each group an equivalence classes of labellings for $\mfr{p}_1\times\mfr{p}_2$. Assume the equality $$\mathcal{O}_{\text{\rm Spr}}(j^{W}_{W(\Phi^{Q,n}_{x,\chi_i})}(\sigma \otimes \varepsilon)^{\text{\rm spe}})=\mca{O}^\vee_\mfr{p}$$holds, then the labelling of $\mfr{p}_1\times\mfr{p}_2$ is determined by $\mca{O}^\vee_\mfr{p}$ up to an equivalence class. It follows from the argument in the preceding paragraph that the labellings for $$\mathcal{O}_\text{\rm Spr}(j^{W_x}_{W(\Phi^{Q,n}_{x,\chi_i})}\sigma^{\rm spe}) \text{ and } \mathcal{O}_{\text{\rm Spr}}(j^{W_x}_{W(\Phi^{Q,n}_{x,\chi_i})}(\sigma \otimes \varepsilon)^{\text{\rm spe}})$$are the same up to an equivalence class if and only if $4|r$. Thus $$\text{\rm Sat}^{\mathbf{G}(\wt{F})}_{\mathbf{L}_x^F(\wt{F})}(\mathcal{O}_\text{\rm Spr}(j^{W_x}_{W(\Phi^{Q,n}_{x,\chi_i})}\sigma^{\rm spe})) \text{ and } \mathcal{O}_{\text{\rm Spr}}(j^{W}_{W(\Phi^{Q,n}_{x,\chi_i})}(\sigma \otimes \varepsilon)^{\text{\rm spe}})$$have the same labellings if and only if $4|r$. This is compatible with the definition of $d_{BV, G}^{(n)}$ for type $D$ in \eqref{E:D-dec}.

The above completes the proof for type $D_r$ and also the verification of Theorem \ref{T:main} for $G$ of classical type.

\subsection{Proof of Theorem \ref{T:main} for exceptional types} \label{SS:exc}
For a cover $\wt{G}$ of exceptional type, we have done the computation of $\mca{D}(\mca{O}^\vee)$ by using ``PyCox" developed by Meinolf Geck etc. We sketch the main algorithmic steps as follows:
\begin{itemize}
    \item[$\bullet$] First,  by Remark \ref{R:minPhi}, we only need to consider the minimal $\Phi^{Q,n}_{x,\chi_i}$ and the ones containing $E_7$ or $E_8$. We sort out all such $\Phi^{Q,n}_{x,\chi_i}$ up to conjugacy.
    \item[$\bullet$] Second, we use ``PyCox" to compute the $j$-induction $j^{\vartheta^{\vee}(W)}_{\vartheta^{\vee}(W(\Phi^{Q,n}_{x,\chi_i}))}$. We traverse all the special representations and pick up the ones satisfying the condition as stipulated in the definition of $\mca{D}(\mca{O}^\vee)$ in  \eqref{E:D}.
    \item[$\bullet$] Third, for a representation $\sigma \in \Irr(W(\Phi_{x, \chi_i}^{Q,n}))$ obtained in the preceding step, we take the special representation in the same family as $\sigma \otimes \varepsilon$ and again  use ``PyCox" to compute the $j$-induction $j^{W_x}_{W(\Phi^{Q,n}_{x,\chi_i})}$.
    \item[$\bullet$] Finally, we compute the saturation by using the tables extracted from \cite{Som98}. For each $W_x, x\in \mca{C}$, we put the orbits in a list and thus easily obtain the maximal ones: this gives $\mca{D}(\mca{O}^\vee)_x$ as in \eqref{E:Dx}. Taking maximum over $x\in \mca{C}$ gives us $\mca{D}(\mca{O}^\vee)$.
\end{itemize}

We illustrate the output of the last step above by giving two examples concerning $G_2$ and $E_6$. Note that $\Phi_{x,\chi_i}^{Q,n}$ actually depends only on $\nn$, as in the classical type case. In Figure \ref{F:G2} and \ref{F:E6}, the second column starting with ``orbit" denotes $\mca{O}^\vee$. A later column starts with the term $X\_ Y\_ \nn$, where $X$ means the group $G$ and $Y=W_x$. The entries in this column denotes $\mca{D}(\mca{O}^\vee)_{x}$, which is always a singleton set. For example, the $(2, G2\_ G2\_ 3)$-entry in Figure \ref{F:G2} means that for $\nn=3$, we have 
$$\mca{D}(\tilde{A}_1)_{0} = G_2(a_1).$$ 
Similarly, for $G=E_6$ and $x\in \mca{C}$ such that $W_x = 3A_2$, the $(7, E6\_ 3A2\_ 3)$ -entry in Figure \ref{F:E6} means that for $\nn=3$ we have $\mca{D}(A_4)_x = D_4(a_1)$.

In any case, from the data of $\mca{D}(\mca{O}^\vee)_x$ for all $x\in \mca{C}$ (available at \url{https://mathrunze.github.io/}) and thus $\mca{D}(\mca{O}^\vee)$ we have checked that $\mca{D}(\mca{O}^\vee) = d_{BV,G}^{(n)}(\mca{O}^\vee)$
for all $\mca{O}^\vee \in \mca{N}(\wt{G}^\vee)$, i.e., Theorem \ref{T:main} holds for exceptional groups.

\begin{figure}[H]
\includegraphics[scale=0.39]{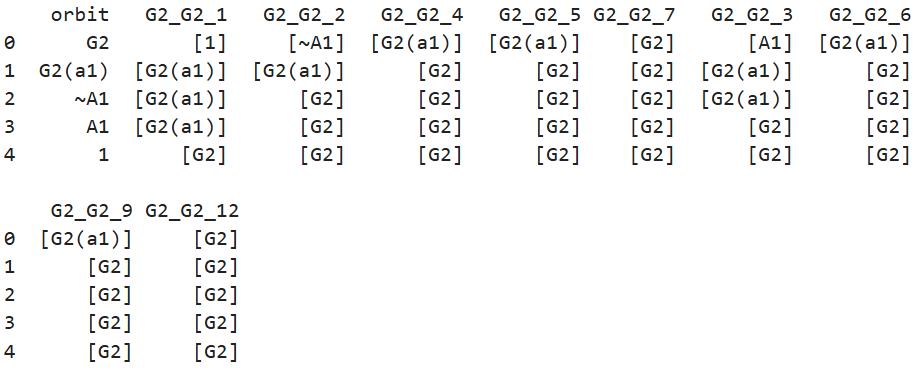}
\caption{$W$ is of type $G_2$ and $x$ is hyperspecial}
\label{F:G2}
\end{figure}

\begin{figure}[H]
\includegraphics[scale=0.4]{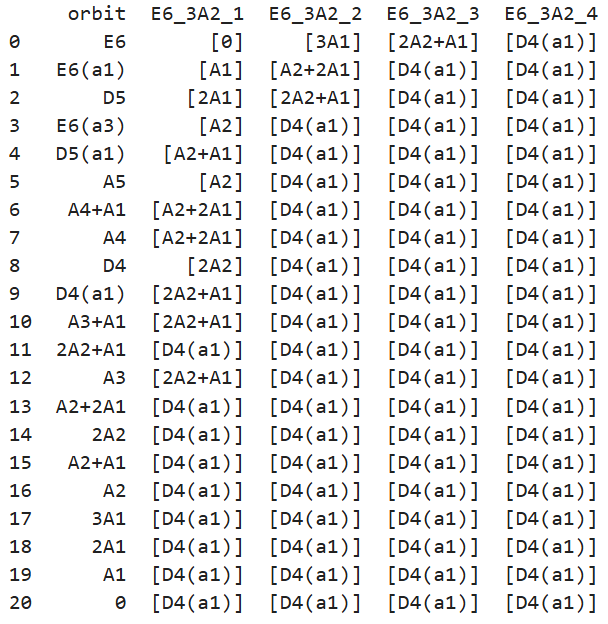}
\caption{$W$ is of type $E_6$ and $W_x$ is of the type $3A_2$}
\label{F:E6}
\end{figure}

\subsection{Upper bound and some equalities}
Theorem \ref{T:main} is pivotal in giving us the upper bound of ${\rm WF}^{\rm geo}({\rm AZ}(\pi))$ for $\pi \in \Irr_\iota(\wt{G})^I$.

\begin{cor} \label{C:key}
Assume $p$ is sufficiently large.
Let $\pi(s, u, \tau) \in \Irr_\iota(\ol{G})^I$ be with positive real $s$.
\begin{enumerate}
\item[(i)] One has 
$$\text{\rm WF}^{\text{\rm geo}}(\text{\rm AZ}(\pi(s,u,\tau))) \lest d^{(n)}_{BV, G}(\mca{O}_u^\vee),$$
where $\mca{O}^\vee_u \in \mca{N}(\wt{G}^\vee)$ is the orbit associated with $u$. Moreover, if $G$ is of type $A$, then the equality holds.
\item[(ii)] Assume $\mca{D}(\mca{O}_u^\vee) < \mca{D}(\mca{O}_{u'}^\vee)$ for every $\mca{O}_{u'}^\vee > \mca{O}_u^\vee$. Then 
$$\mca{D}^\flat(\mca{O}_u^\vee) = \text{\rm WF}^{\text{\rm geo}}(\text{\rm AZ}(\pi(s,u,\mbm{1}))) = \mca{D}(\mca{O}_u^\vee).$$
If particular, if $G$ is of exceptional type and the orbit in $\mca{N}(\mbf{G}(\ol{F}))$, which canonically corresponds to $\mca{O}_u^\vee$ via weighted Dynkin diagrams, lies in  ${\rm Im}(d_{BV,G}^{(n)})$, then we have the equality
$$\text{\rm WF}^{\text{\rm geo}}(\text{\rm AZ}(\pi(s,u,\mbm{1}))) = d_{BV,G}^{(n)}(\mca{O}_u^\vee).$$
\end{enumerate}
\end{cor}
\begin{proof} 
For (i), we have the chain of inequalities 
\begin{equation} \label{E:WF-D}
\text{\rm WF}^{\text{\rm geo}}(\text{\rm AZ}(\pi(s,u,\tau))) \lest \mca{D}(\mca{O}^\vee_u) \lest d^{(n)}_{BV, G}(\mca{O}^\vee_u),
\end{equation} 
where the first inequality follows from Lemma \ref{lm:AA} and the second one from Theorem \ref{T:main}.

For the second statement in (i), we assume $G$ is of type $A$. To show the equality holds, in view of Theorem \ref{T:main}, it suffices to show $\mca{D}(\mca{O}_u^\vee) \lest {\rm WF}^{\text{\rm geo}}(\text{\rm AZ}(\pi(s,u,\tau)))$. 
From the equality $\mca{D}=d_{BV,A}^{(n)}$ in Theorem \ref{T:main} and also Lemma \ref{lm:Ar}, we see
\begin{equation} \label{E:Db}
\mca{D}(\mca{O}_u^\vee)= \mca{D}^\flat(\mca{O}_u^\vee).
\end{equation}
For fixed $\mca{O}_u^\vee \in \mca{N}(\wt{G}^\vee)$, if $\sigma \in \Irr(W(\Phi_{x, \chi_i}^{Q,n}))$ satisfies
$$
\mathcal{O}_{\text{\rm Spr}}(j^{\vartheta^{\vee}(W)}_{\vartheta^{\vee}(W(\Phi^{Q,n}_{x,\chi_i}))}(\vartheta^\vee(\sigma) \otimes \varepsilon)) = \mca{O}_u^\vee
$$
as in the definition of $\mca{D}^\flat(\mca{O}^\vee_u)$, then clearly $\sigma \subseteq {\rm Res}_{W(\Phi_{x, \chi_i}^{Q,n})}^W (\vartheta^\vee(E_{u, \mbm{1}}) \otimes \varepsilon)$.  By the description of  $\Irr(W)_{s,u,\tau}$ in \eqref{eq:res}, we have $\vartheta^\vee(E_{u,\mbm{1}}) \in \Irr(W)_{s,u,\tau}$ (since we must have $\tau = \mbm{1}$ in this case); thus, every such $\sigma$ 
belongs to the set in the right hand side of the equality of computing ${\rm WF}^{\text{\rm geo}}(\text{\rm AZ}(\pi(s,u,\tau)))$, as given in Proposition \ref{main_prop}. This shows $\mca{D}^\flat(\mca{O}_u^\vee) \lest {\rm WF}^{\text{\rm geo}}(\text{\rm AZ}(\pi(s,u,\tau)))$ and completes the proof of the second statement of (i) in view of \eqref{E:Db}.

For (ii), we have the chain of inequalities
$$\mca{D}^\flat (\mca{O}^\vee_u) \lest \text{\rm WF}^{\text{\rm geo}}(\text{\rm AZ}(\pi(s,u,\mbm{1}))) \lest \mca{D}(\mca{O}^\vee_u),$$
where the first inequality follows from Proposition \ref{main_prop}. The assumption in (ii) implies that $\mca{D}^\flat (\mca{O}^\vee_u) = \mca{D}(\mca{O}^\vee_u)$, which gives the chain of equalities in (ii). For the second assertion in (ii), let $G$ be of exceptional type with the assumption installed. In view of the first statement in (ii) and Theorem \ref{T:main}, it suffices to show that for such $\mca{O}_u^\vee$ one has
$$d_{BV,G}^{(n)}(\mca{O}_u^\vee) < d_{BV,G}^{(n)}(\mca{O}_{u'}^\vee)$$
for every $\mca{O}_{u'}^\vee > \mca{O}_u^\vee$. However, this follows from a direct inspection of the definition of $d_{BV,G}^{(n)}$ from \cite[Appendix]{GLLS}. This completes the proof of (ii).
\end{proof}
Note that the argument for the second statement in Corollary \ref{C:key}(ii) also applies for type $A$. That is, if $G$ is of type $A$ and $\mfr{p} \in {\rm Im}(d_{BV,G}^{(n)})$, then $d_{BV,G}^{(n)}(\mfr{p}) < d_{BV,G}^{(n)}(\mfr{p}')$
for every $\mfr{p}' > \mfr{p}$. Indeed, from the combinatorial definition of $d_{com}^{(\nn)}$ and thus $d_{BV,G}^{(n)}$, it is easy to see that 
$$(d_{com}^{(\nn)})^2(\mfr{p}') \gest \mfr{p}' \text{ and } (d_{com}^{(\nn)})^3(\mfr{p}') = d_{com}^{(\nn)}(\mfr{p}')$$ 
for all partition $\mfr{p}'$. This gives $(d_{com}^{(\nn)})^2(\mfr{p}') \gest \mfr{p}' > \mfr{p}$. Then, we necessarily have
$$d_{com}^{(\nn)}(\mfr{p}') < d_{com}^{(\nn)}(\mfr{p}).$$
Otherwise, it gives the equality $d_{com}^{(\nn)}(\mfr{p}') = d_{com}^{(\nn)}(\mfr{p})$ and implies that $(d_{com}^{(\nn)})^2(\mfr{p}') = (d_{com}^{(\nn)})^2(\mfr{p}) = \mfr{p}$, since $\mfr{p} \in {\rm Im}(d_{com}^{(\nn)})$; this is a contradiction.
\vskip 10pt

We have shown that $\text{\rm WF}^{\rm geo}(\text{\rm AZ}(\pi(s,u,\tau))) \lest d^{(n)}_{BV,G}(\mca{O}^\vee_u)$ for positive real $s$. The equality holds for linear groups by \cite{CMBO21}; for covers, it holds for the theta representations by \cite{KOW}, for type $A$, and for certain orbits $\mca{O}_u^\vee$ in the case of exceptional groups by Corollary \ref{C:key}. However, the equality is not always attained for general Iwahori-spherical representations of $\wt{G}$, even just for positive real $s$.

\begin{eg} \label{E:c-eg}
First, consider a primitive cover $\wt{G}=\wt{\Sp}_6^{(3)}$; we get $\wt{G}^\vee = \SO_7$. Let $\mca{O}_u^\vee$ be associated with the partition $\mfr{p}:=(3,3,1)$. Let $s \in \wt{T}^\vee$ be positive real such that $A_{s,u}=A_{\mathcal{O}_u}$.   Assume $\pi(s, u, \tau)$ has nontrivial $\tau$. Then $d^{(3)}_{BV, \Sp_6}(\mca{O}_u^\vee)=(6)=\mca{D}(\mca{O}^\vee_u)$  but $\text{\rm WF}^{\rm geo}(\text{\rm AZ}(\pi(s,u,\tau)))=(4,2)$.

Second, for a primitive cover $\wt{G}_2^{(2)}$ of $G_2$. Consider $\pi(s, u, \mbm{1})$ with positive real $s$ and $\mca{O}_u^\vee=A_1$, and $\pi(s, u, \mbm{1})^I$ is an one-dimensional module over $\mca{H}_0$. Then $d_{BV, G_2}^{(2)}(\mca{O}_u^\vee) = G_2 = \mca{D}(\mca{O}_u^\vee)$ but $\text{\rm WF}^{\rm geo}(\text{\rm AZ}(\pi(s,u,\mbm{1})))=G_2(a_1)$.
\end{eg}


\subsection{Lower bound and a speculation}
First, we show that for positive real $s$ there is always a coarse lower bound of the geometric wavefront set $\WF^{\rm geo}(\text{AZ}(\pi(s,u,\tau)))$
given by that of theta representations. Indeed, for every $\sigma \in \Irr(W(\Phi_{x,\chi_i}^{Q,n}))$, it follows from \cite[Proposition 4.3]{AcAu07} that 
$$\mathcal{O}_\text{\rm Spr}(j^{W_x}_{W(\Phi^{Q,n}_{x,\chi_i})}(\sigma^{\rm spe})) \gest \mathcal{O}_\text{\rm Spr}(j^{W_x}_{W(\Phi^{Q,n}_{x,\chi_i})} \varepsilon).$$
Notice that if $\text{AZ}(\pi(s,u,\tau))$ is a theta representation (i.e., $\pi(s,u,\tau)$ is a covering Steinberg representation), then $\Irr(W)_{s,u,\tau}=\set{\varepsilon}$ is a singleton set,  where $\varepsilon$ is just the sign representation of $W(\Phi^{Q,n}_{x,\chi_i})$, as shown in \cite{KOW}. Thus the theta representation has the minimal geometric wavefront set among the representations considered in Proposition \ref{main_prop}. This is a generalization of the fact that theta representation has the least Whittaker dimension among all representations in $\Irr_\iota(\wt{G})^I$, see \cite[Corollary 9.9]{GGK1}.

In general, one has the following speculation, which is supported by the numerical evidence so far.
\begin{conj} \label{C:1}
Let $\pi(s, u, \tau) \in \Irr_\iota(\ol{G})^I$ be such that $s$ is positive real. Then the following holds:
\begin{enumerate}
\item[(i)] either $\WF^{\text{\rm geo}}(\text{\rm AZ}(\pi(s,u,\tau)))= d^{(n)}_{BV, G}(\mca{O}_u^\vee)$, or
\item[(ii)] $\WF^{\text{\rm geo}}(\text{\rm AZ}(\pi(s,u,\tau)))$ is equal to a maximal orbit in the set 
$$\set{\mca{O} \in {\rm Im}(d_{BV,G}^{(n)}): \ \mca{O} < d^{(n)}_{BV, G}(\mca{O}_u^\vee)};$$
\end{enumerate}
that is, $\WF^{\text{\rm geo}}(\text{\rm AZ}(\pi(s,u,\tau)))$ is at most ``one-step" smaller from $d^{(n)}_{BV, G}(\mca{O}_u^\vee)$ within the image of $d_{BV,G}^{(n)}$.
\end{conj}

\section{The equality for certain $\pi_S$}
In this section, we generalize the results of \cite{KOW} concerning ${\rm WF}^{\rm geo}(\Theta)$ in another direction. Following the exposition and notation from \cite{GaTs}, we consider $\nu \in X\otimes \R$ that satisfies the following:
\begin{enumerate}
\item[--] $\nu$ is regular, that is, its stabilizer subgroup of $W$ is trivial;
\item[--] the set $\Phi(\nu):=\set{\alpha\in \Phi: \nu(\alpha_{Q,n}^\vee) =1}$ is a subset of $\Delta$.
\end{enumerate}
Taking $\pi^\dag$ to be an unramified distinguished genuine representation of $\wt{T}$, we have a regular unramified genuine principal series $I(\nu):=I(\pi^\dag, \nu)$ of $\wt{G}$.
One has
\begin{equation} \label{E:Rod}
I(\nu)^{\rm ss} = \bigoplus_{S \subseteq \Phi(\nu)} \pi_S,
\end{equation}
where the left hand side denotes the semisimplification of $I(\nu)$. The decomposition is multiplicity-free and the irreducible constituent $\pi_S$ is characterized by its Jacquet module, see \cite{Rod4} and \cite[\S 3]{Ga6}. For example, if $\Phi(\nu) =\Delta$, then $\pi_\Delta = \Theta(\nu)$ is an unramified theta representation and $\pi_\emptyset$ is a covering Steinberg representation.

For every $S \subseteq \Phi(\nu)$, let $\Phi(S) \subseteq \Phi$ be the root subsystem with simple roots. Denote by
$$W(S) \subseteq W$$
the subgroup generated by simple reflections from $S$. Let $\mbf{M}_S \subseteq \mbf{G}$ be the Levi subgroup associated with $S$. Denote 
$$W_\nu^S:=\text{the integral Weyl subgroup of $W(S)$ generated by } \set{\alpha \in \Phi(S): \ \angb{\alpha^\vee}{\nu} \in \Z}.$$
Let $\varepsilon_\nu^S$ be the sign character of $W_\nu^S$. 

With all the above notation, we expect that 
\begin{equation} \label{E:piS1}
\WF^{\rm geo}(\pi_S)= \mca{O}_{\rm Spr}(j_{W_\nu^S}^W(\varepsilon_\nu^S)).
\end{equation}
By using properties of $j$-induction and the induction of orbits, as noted in \cite{GaTs}, we get
$$\mca{O}_{\rm Spr}(j_{W_\nu^S}^W(\varepsilon_\nu^S))
 = {\rm Ind}_{\mbf{M}_S}^{ \mbf{G}}( \mca{O}_{\rm Spr}(j_{W_\nu^S}^{W(S)}(\varepsilon_\nu^S))).
$$

However, from \cite{GLLS} we know 
$$\mca{O}_{\rm Spr}(j_{W_\nu^S}^{W(S)}(\varepsilon_\nu^S)) = d_{BV, M_S}^{(n)}(\mca{O}^{\vee,S}_{\rm reg})),$$
where we denote by $\mca{O}_{\rm reg}^{\vee, S}$ the regular orbit of $\wt{M}_S^\vee$. 
Also, for every orbit $\mca{O}^\vee$ of $\wt{M}_S^\vee$, the covering Barbasch--Vogan duality intertwins the induction and saturation of orbits as in 
$$d_{BV, G}^{(n)} \circ {\rm Sat}_{\wt{M}_S^\vee}^{\wt{G}^\vee}(\mca{O}^\vee) = {\rm Ind}_{\mbf{M}_S}^\mbf{G} \circ d_{BV, M_S}^{(n)} (\mca{O}^\vee),$$
see \cite[Theorem 1.1]{GLLS}.
We have $\mca{O}(\phi_{\pi_{\Phi(\nu)-S}}) = {\rm Sat}_{\wt{M}_S^\vee}^{\wt{G}^\vee}( \mca{O}^{\vee, S}_{\rm reg})$ and also 
${\rm AZ}(\pi_S) = \pi_{\Phi(\nu) - S}$, which gives
$$\mca{O}(\phi_{{\rm AZ}(\pi_S)}) = {\rm Sat}_{\wt{M}_S^\vee}^{\wt{G}^\vee}(\mca{O}_{\rm reg}^{\vee,S}).$$
Combining all the above, we see that 
\begin{equation} \label{E:j=d}
\mca{O}_{\rm Spr}(j_{W_\nu^S}^W(\varepsilon_\nu^S)) =  d_{BV, G}^{(n)}(\mca{O}_{\rm reg}^{\vee,S}) =d_{BV, G}^{(n)}(\mca{O}(\phi_{{\rm AZ}(\pi_S)})),
\end{equation}
where henceforth, by abuse of notation, we also write $\mca{O}_{\rm reg}^{\vee,S}$ for the saturated orbit ${\rm Sat}_{\wt{M}_S^\vee}^{\wt{G}^\vee}(\mca{O}_{\rm reg}^{\vee,S})$ for $\wt{G}^\vee$, and there should be no risk of confusion from the context.
In particular, we see that \eqref{E:piS1} is equivalent to
\begin{equation} \label{E:piS2}
\WF^{\rm geo}(\pi_S) = d_{BV, G}^{(n)}(\mca{O}_{\rm reg}^{\vee,S}).
\end{equation}

The equality \eqref{E:piS1} (or \eqref{E:piS2}) was known in the following cases:
\begin{enumerate}
\item[$\bullet$] For linear group $G$ (i.e., $n=1$), the equality becomes 
$$\WF^{\rm geo}(\pi_S) = {\rm Ind}_{\mbf{M}_S}^\mbf{G}(0)= d_{BV, G}^{(1)}(\mca{O}_{\rm reg}^{\vee,S})$$
and was proved by M\oe glin and Waldspurger in \cite[Proposition II.1.3]{MW87}.
\item[$\bullet$] For general covering $\wt{G}$, since $(\pi_\emptyset)^I$ always contains the sign representation of the finite Weyl-group algebra inside $\mca{H}_0$, it follows from \cite{GGK1} that $\pi_\emptyset$ is generic. Thus, \eqref{E:piS2} holds for $\pi_\emptyset$.  If $S=\Delta$ (and thus necessarily $\Phi(\nu)=\Delta$), then $\pi_\Delta$ is the theta representation, and in this case \eqref{E:piS2} follows from the work \cite{KOW}.
\end{enumerate}

For $S \subseteq \Delta$, we write $e(S)$ for the maximum exponent of the Weyl group $W(S)$, i.e., it is the maximum of all the exponents $e(S_i)$ associated with the connected components $S_i$ of $S$. We identify $S$ with the root system it represents.

\begin{dfn}\label{D:attm}
Let $\wt{G}^{(n)}$ be a cover of an almost simple $G$.
\begin{enumerate}
\item[(i)] For $G$ of classical type, a set $S \subseteq \Delta$ is called $\wt{G}^{(n)}$-autotomous if there exists $\beta \in \Delta -S$ and a connected component $S_j \subseteq S \sqcup \set{\beta}$ such that $\nn \gest e(S_j) + 1$.
\item[(ii)] For $G$ of exceptional type, a set $S \subseteq \Delta$ is called $\wt{G}^{(n)}$-autotomous if 
\begin{enumerate}
\item[$\bullet$] $S = H + A_k$ with $k \gest 0$, $\val{H} \gest k+1$;
\item[$\bullet$]  $H+ A_{k+1} \subseteq \Delta$ also;
\item[$\bullet$] $\nn \gest k+2$.
\end{enumerate}
\end{enumerate}
Here $A_k$ (in contrast with $\tilde{A}_k$) means it is associated with long roots in $\Delta$, and for $G$ of simply-laced type every root is considered as long. 
\end{dfn}
As concrete examples, the sets $\tilde{A}_1, \tilde{A}_2 \subset \Delta$ are both $\wt{F}_4^{(2)}$-autotomous for a primitive cover. We also remark that the condition in Definition \ref{D:attm}(i), if applied to $G$ of exceptional type, gives weaker condition than (ii).

\begin{thm} \label{T:main2}
Assume $p$ is sufficiently large. 
Let $S \subseteq \Phi(\nu)$. Assume $S$ is not $\wt{G}^{(n)}$-autotomous. If $G$ is of exceptional type, we further assume that $\nn \lest e(S)$. 
Then we have the equalities 
$$\WF^{\rm geo}(\pi_S)= \mca{O}_{\rm Spr}(j_{W_\nu^S}^W(\varepsilon_\nu^S))=  d_{BV, G}^{(n)}(\mca{O}_{\rm reg}^{\vee,S}) = d_{BV, G}^{(n)}(\mca{O}(\phi_{{\rm AZ}(\pi_S)})).$$
\end{thm}
The remaining of this section is devoted to a proof of this.

\subsection{A reduction}
We note that for every $S \subseteq \Phi(\nu)$, one has
$$\left( {\rm Ind}_{\wt{P}_S}^{\wt{G}} \Theta(\wt{M}_S) \right)^{\rm s.s.} = \bigoplus_{\substack{S' \subseteq \Phi(\nu) \\ S' \supseteq S}} \pi_{S'},$$
and by M\oe glin--Waldspurger \cite[(9)]{MW87}, we get
$$\WF^{\rm geo}({\rm Ind}_{\wt{P}_S}^{\wt{G}} \Theta(\wt{M}_S)) = {\rm Ind}_{\mbf{M}_S}^{\mbf{G}}(\WF^{\rm geo}(\Theta(\wt{M}_S))) = d_{BV, G}^{(n)}(\mca{O}_{\rm reg}^{\vee,S}).$$ 
This gives that for all $S' \supseteq S$, one has
$$\WF^{\rm geo}(\pi_{S'}) \lest d_{BV, G}^{(n)}(\mca{O}_{\rm reg}^{\vee,S'}).$$
Thus, to prove Theorem \ref{T:main2}, it suffices to show the following:

\begin{prop} \label{P:piSkey}
Let $S \subseteq \Phi(\nu) \subseteq \Delta$. Assume $S$ is not $\wt{G}^{(n)}$-autotomous. If $G$ is of exceptional type, we further assume that $\nn \lest e(S)$. Then for every $S' \supsetneq S$, one has
$$d_{BV, G}^{(n)}(\mca{O}_{\rm reg}^{\vee,S'}) < d_{BV, G}^{(n)}(\mca{O}_{\rm reg}^{\vee,S}).$$
\end{prop}

The condition on $\nn$ in Proposition \ref{P:piSkey} is closely related to the non-genericity of theta representations on covering Levi subgroups. Using \cite{KOW, GLT25}, we first explicate the condition  $\WF^{\rm geo}(\Theta(\wt{G}^{(n)})) < \mca{O}_{\rm reg}$, given in Table \ref{table 1} and Table \ref{table 2}.

\begin{table}[H] 
\caption{$\WF^{\rm geo}(\Theta(\wt{G}^{(n)})) < \mca{O}_{\rm reg}$ for classical type} \label{table 1}
\vskip 5pt
\renewcommand{\arraystretch}{1.4}
\begin{tabular}{|c|c|c|c|c|c|}
\hline
 & type $A_r$  &    type $B_r$  & type $C_r$ &  type $D_r$       \\
\hline
\hline
$\nn$ odd   & $\nn \lest r$   &  $\nn \lest 2r-1$   & $\nn \lest 2r-1$    &  $\nn \lest 2r-3$         \\
$\nn$ even & $\nn \lest r$    & $\nn \lest 2r$  & $\nn/2 \lest 2r-2$    & $\nn \lest 2r-4$      \\
\hline
\end{tabular}
\end{table}

\begin{table}[H] 
\caption{$\WF^{\rm geo}(\Theta(\wt{G}^{(n)})) < \mca{O}_{\rm reg}$ for exceptional type} \label{table 2}
\vskip 5pt
\renewcommand{\arraystretch}{1.4}
\begin{tabular}{|c|c|c|c|c|c|c|}
\hline
 type $G_2$  &    type $F_4$  & type $E_6$ &  type $E_7$ & $E_8$       \\
\hline
\hline
$\nn \lest 5$ &  $\nn \lest 11$   &  $\nn \lest 11$   & $\nn \lest 17$    &  $\nn \lest 29$         \\
or $\nn=6,9$ & or $\nn=12, 14, 16$    &  &     &       \\
\hline
\end{tabular}
\end{table}
We note that except the cases of $\nn$ even for type $C_r$, and $\nn=6, 9$ for type $G_2$, and $\nn=12, 14, 16$ for type $F_4$, all other cases are of the form $\nn \lest e(\Delta)$, where $e(\Delta)$ is the maximum exponent of the Weyl group of $G$. 

\subsection{Proof  for classical types} 
For every partition $\mfr{m}=(m_1, ..., m_s)$ with $s\gest 2$ and $m_i \in \N_{\gest 1}$, we write $\val{\mfr{m}}=\sum_{i=1}^s m_i$. We first state a basic but key property of the functions $\mfr{s}(-;z)$ and $d_{com}^{(z)}(-)$ given in \eqref{smn}.

\begin{lm} \label{L:key01}
Keep the notation as above. Then
$$d_{com}^{(z)}(\val{\mfr{m}}) = \mfr{s}(\val{\mfr{m}}; z) \lest \sum_{i=1}^s \mfr{s}(m_i;z) = d_{com}^{(z)}(\mfr{m}).$$
Moreover, if there is a non-empty subset $J\subseteq \set{1, 2, ..., s}$ such that $z < \sum_{i \in J} m_i$, then the above inequality is strict.
\end{lm}
\begin{proof}
The two equalities hold by the definition of $d_{com}^{(z)}$. The middle inequality follows from the order-reversing property of $d_{com}^{(z)}$, since $\val{\mfr{m}} \gest \mfr{m}$ as partitions. For the second assertion, the assumption implies that the leading numeral of the partition $d_{com}^{(z)}(\mfr{m})$ is bounded below by that of $\sum_{i\in J} \mfr{s}(m_i;z)$, and is $\gest z+1$. However, the leading numeral of $\mfr{s}(\val{\mfr{m}}; z)$ is equal to $z$. This gives the strict inequality.
\end{proof}

In this subsection, we consider group $\wt{G}$ of classical type $A_{r-1}, B_r, C_r$ or $D_r$. 
Let $S \subseteq \Delta$. Assume $S$ is not $\wt{G}$-autotomous. It implies the following property:
\begin{enumerate}
\item[(P)] for every $\beta \in \Delta - S$ and every connected component $S_j$ of $S \sqcup \set{\beta}$, one has $\nn \lest e(S_j)$.
\end{enumerate}
Since $d_{BV,G}^{(n)}$ is order-reversing, to show the inequality in Proposition \ref{P:piSkey}, it suffices to consider $S'$ of the form $S \sqcup \set{\beta}$. To verify the inequality for such $S'$, we will use property (P) above.
For a partition $\mfr{p}$, we write $L^\sharp(\mfr{p})$ for its leading numeral.

\subsubsection{Type $A_{r-1}$}
Let $\wt{G}$ be a cover of $G$ whose root system is of type $A_{r-1}$. Let 
$$\mfr{p}=(p_1, p_2, ..., p_i, p_{i+1}, ...,  p_s)$$ be an ordered partition of $r$ associated with the Levi subgroup $M_{S_\mfr{p}}$ of type 
$A_{p_1-1} \times ... \times A_{p_s-1}$. Let $\mfr{q}=(q_1, q_2, ..., q_{s-1})$ be another partition such that $S_\mfr{q} =  S_\mfr{p} \sqcup \set{\beta}$. Then there is $i$ such that $\mfr{q}=(p_1, ..., p_{i-1}, p_{i}+p_{i+1}, p_{i+2}, ..., p_s)$

Property (P) implies that $\nn < p_i + p_{i+1}$ and it follows from Lemma \ref{L:key01} that
$$d_{BV, G}^{(n)}(\mfr{q}) = \sum_{j=1}^{s-1} \mfr{s}(q_j; \nn) < \sum_{i=1}^s \mfr{s}(p_i; \nn) = d_{BV, G}^{(n)}(\mfr{p}),$$
Hence, Proposition \ref{P:piSkey} and thus Theorem \ref{T:main2} hold in this case.

\subsubsection{Type $B_r$}
We consider cover $\wt{G}$ of type $B_r$. Let $\mfr{p}=(p_1, p_2, ..., p_{s-1}, p_s)$ be an ordered partition of $r$ associated with the Levi subgroup $M_{S_\mfr{p}}$ of type 
$$A_{p_1-1} \times ... \times A_{p_{s-1}-1} \times B_{p_s}.$$
Let $\mfr{q}:=(q_1, q_2, ..., q_{s-2}, q_{s-1})$ be the partition associated with $S_\mfr{q}$ such that $S_\mfr{q} = S_\mfr{p} \sqcup \set{\beta}$. 
Then there are $q_{j_0}, p_{i_0}, p_{i_0+1}$ such that $q_{j_0} = p_{i_0} + p_{i_0+1}$, and the other numerals in $\mfr{q}$ and $\mfr{p}$ are identical. 
According to Table \ref{table 1}, there are two cases to consider, as follows.

First, if $\nn$ is odd, then by the definition in \eqref{E:dBV-B}, we have $d_{BV,G}^{(n)}(\mca{O}) = d_{com}^{(\nn)}(\mca{O})^+{}_B$. 
The regular orbit of $\wt{M}_{S_\mfr{p}}^\vee \subset \wt{G}^\vee$ is $\mca{O}_{\rm reg}^{\vee,S_\mfr{p}} = (p_1^2, ..., p_{s-1}^2, 2p_s)$. Similarly, $\mca{O}_{\rm reg}^{\vee,S_\mfr{q}} = (q_1^2, ..., q_{s-2}^2, 2q_{s-1})$.
 If $j_0\lest s-2$, then property (P) implies $\nn < p_{i_0} + p_{i_0+1}$. Lemma \ref{L:key01} implies that $d_{com}^{(\nn)}(\mca{O}_{\rm reg}^{\vee,S_\mfr{q}}) < d_{com}^{(\nn)}(\mca{O}_{\rm reg}^{\vee,S_\mfr{p}} )$. We have
$$L^\sharp(d_{com}^{(\nn)}(\mca{O}_{\rm reg}^{\vee,S_\mfr{q}})) + 2 \lest  L^\sharp( d_{com}^{(\nn)}(\mca{O}_{\rm reg}^{\vee,S_\mfr{p}} )),$$
and thus 
$d_{BV, G}^{(n)}(\mca{O}_{\rm reg}^{\vee,S_\mfr{q}}) < d_{BV, G}^{(n)}(\mca{O}_{\rm reg}^{\vee,S_\mfr{p}} )$.
If $j_0=s-1$, then we get $q_{s-1} = p_{s-1} + p_s$. Property (P) implies $\nn < 2p_s + p_{s-1} + p_{s-1}$. Lemma \ref{L:key01} gives that $d_{com}^{(\nn)}(\mca{O}_{\rm reg}^{\vee,S_\mfr{q}}) < d_{com}^{(\nn)}(\mca{O}_{\rm reg}^{\vee,S_\mfr{p}} )$, and in this case 
$L^\sharp(d_{com}^{(\nn)}(\mca{O}_{\rm reg}^{\vee,S_\mfr{q}}))$ is odd due to the fact $q_{s-1}>0$, and
$$L^\sharp(d_{com}^{(\nn)}(\mca{O}_{\rm reg}^{\vee,S_\mfr{q}})) + 1 \lest  L^\sharp( d_{com}^{(\nn)}(\mca{O}_{\rm reg}^{\vee,S_\mfr{p}} )).$$
Thus, $d_{BV, G}^{(n)}(\mca{O}_{\rm reg}^{\vee,S_\mfr{q}}) < d_{BV, G}^{(n)}(\mca{O}_{\rm reg}^{\vee,S_\mfr{p}} )$ also holds.

Second, if $\nn$ is even, then $\wt{G}^\vee$ is of type $B_r$. We have $d_{BV,G}^{(n)}(\mca{O}) = d_{com}^{(\nn)}(\mca{O})_B$. 
The regular orbit of $\wt{M}_{S_\mfr{p}}^\vee \subseteq \wt{G}^\vee$ is
$$\mca{O}_{\rm reg}^{\vee,S_\mfr{p}} = (p_1^2, ..., p_{s-1}^2, 2p_s +1).$$
Similarly, $\mca{O}_{\rm reg}^{\vee,S_\mfr{q}} = (q_1^2, ..., q_{s-2}^2, 2q_{s-1} +1)$.
 If $j_0\lest s-2$, then property (P) implies $\nn < p_{i_0} + p_{i_0+1}$. Argue exactly as the $\nn$ being odd case, we get  
$d_{BV, G}^{(n)}(\mca{O}_{\rm reg}^{\vee,S_\mfr{q}}) < d_{BV, G}^{(n)}(\mca{O}_{\rm reg}^{\vee,S_\mfr{p}} )$.
If $j_0=s-1$, then we get $q_{s-1} = p_{s-1} + p_s$. Property (P) implies $\nn < (2p_s +1) + p_{s-1} + p_{s-1}$. Lemma \ref{L:key01} gives that $d_{com}^{(\nn)}(\mca{O}_{\rm reg}^{\vee,S_\mfr{q}}) < d_{com}^{(\nn)}(\mca{O}_{\rm reg}^{\vee,S_\mfr{p}} )$, and in this case 
$L^\sharp(d_{com}^{(\nn)}(\mca{O}_{\rm reg}^{\vee,S_\mfr{q}}))$ is even, and
$$L^\sharp(d_{com}^{(\nn)}(\mca{O}_{\rm reg}^{\vee,S_\mfr{q}})) + 1 \lest  L^\sharp( d_{com}^{(\nn)}(\mca{O}_{\rm reg}^{\vee,S_\mfr{p}} )).$$
Thus, $d_{BV, G}^{(n)}(\mca{O}_{\rm reg}^{\vee,S_\mfr{q}}) < d_{BV, G}^{(n)}(\mca{O}_{\rm reg}^{\vee,S_\mfr{p}} )$ also holds. 

The above verifies Proposition \ref{P:piSkey} and thus Theorem \ref{T:main2} for type $B_r$.

\subsubsection{Type $C_r$}
Let $\wt{G}$ be a cover of $G$ of Cartan type $C_r$. Let $\mfr{p}=(p_1, p_2, ..., p_{s-1}, p_s)$ be a partition of $r$ associated with the Levi subgroup $M_{S_\mfr{p}}$ of type 
$$A_{p_1-1} \times ... \times A_{p_{s-1}-1} \times C_{p_s}.$$
Let $\mfr{q}:=(q_1, q_2, ..., q_{s-2}, q_{s-1})$ be the partition associated with $S_\mfr{q}$ such that $S_\mfr{q} = S_\mfr{p} \sqcup \set{\beta}$. 
Then there are $q_{j_0}, p_{i_0}, p_{i_0+1}$ such that $q_{j_0} = p_{i_0} + p_{i_0+1}$, and the other numerals in $\mfr{q}$ and $\mfr{p}$ are identical. 
There are also two cases to consider, according to the parity of $\nn$.

First, if $\nn$ is odd, then $\wt{G}^\vee$ is of type $B_r$ and we have $d_{BV,G}^{(n)}(\mca{O}) = d_{com}^{(\nn)}(\mca{O})^-{}_C$. 
The regular orbit of $\wt{M}_{S_\mfr{p}}^\vee \subset \wt{G}^\vee$ is $\mca{O}_{\rm reg}^{\vee,S_\mfr{p}} = (p_1^2, ..., p_{s-1}^2, 2p_s + 1)$. Similarly, $\mca{O}_{\rm reg}^{\vee,S_\mfr{q}} = (q_1^2, ..., q_{s-2}^2, 2q_{s-1}+1)$.
 If $j_0\lest s-2$, then property (P) implies $\nn < p_{i_0} + p_{i_0+1}$. Lemma \ref{L:key01} implies that $d_{com}^{(\nn)}(\mca{O}_{\rm reg}^{\vee,S_\mfr{q}}) < d_{com}^{(\nn)}(\mca{O}_{\rm reg}^{\vee,S_\mfr{p}} )$ with $L^\sharp(d_{com}^{(\nn)}(\mca{O}_{\rm reg}^{\vee,S_\mfr{q}})) + 2 \lest  L^\sharp( d_{com}^{(\nn)}(\mca{O}_{\rm reg}^{\vee,S_\mfr{p}} ))$. This clearly gives 
$d_{BV, G}^{(n)}(\mca{O}_{\rm reg}^{\vee,S_\mfr{q}}) < d_{BV, G}^{(n)}(\mca{O}_{\rm reg}^{\vee,S_\mfr{p}} )$.
If $j_0=s-1$, then we get $q_{s-1} = p_{s-1} + p_s$. Property (P) implies $\nn < 2p_s + p_{s-1} + p_{s-1}$. Lemma \ref{L:key01} gives that $d_{com}^{(\nn)}(\mca{O}_{\rm reg}^{\vee,S_\mfr{q}}) < d_{com}^{(\nn)}(\mca{O}_{\rm reg}^{\vee,S_\mfr{p}} )$, and in this case 
$L^\sharp(d_{com}^{(\nn)}(\mca{O}_{\rm reg}^{\vee,S_\mfr{q}}))$ is odd, and
$$L^\sharp(d_{com}^{(\nn)}(\mca{O}_{\rm reg}^{\vee,S_\mfr{q}})) + 1 \lest  L^\sharp( d_{com}^{(\nn)}(\mca{O}_{\rm reg}^{\vee,S_\mfr{p}} )).$$
Thus, it is easy to see $d_{BV, G}^{(n)}(\mca{O}_{\rm reg}^{\vee,S_\mfr{q}}) < d_{BV, G}^{(n)}(\mca{O}_{\rm reg}^{\vee,S_\mfr{p}} )$ holds as well.

Second, if $\nn$ is even, then $\wt{G}^\vee$ is of type $C_r$. The map $d_{BV,G}^{(n)}$ is given in  \eqref{E:dBV-C}.
The regular orbit of $\wt{M}_{S_\mfr{p}}^\vee \subseteq \wt{G}^\vee$ is $\mca{O}_{\rm reg}^{\vee,S_\mfr{p}} = (p_1^2, ..., p_{s-1}^2, 2p_s)$. Similarly, $\mca{O}_{\rm reg}^{\vee,S_\mfr{q}} = (q_1^2, ..., q_{s-2}^2, 2q_{s-1})$.
 If $j_0\lest s-2$, then property (P) implies $\nn < p_{i_0} + p_{i_0+1}$. Thus, $\nn/2 < p_{i_0} + p_{i_0+1}$. Similar as the $\nn$ odd case above,
 we get
 $$d_{com}^{(\nn/2)}(\mca{O}_{\rm reg}^{\vee,S_\mfr{q}}) < d_{com}^{(\nn/2)}(\mca{O}_{\rm reg}^{\vee,S_\mfr{p}} ) \text{ with } L^\sharp(d_{com}^{(\nn/2)}(\mca{O}_{\rm reg}^{\vee,S_\mfr{q}})) + 2 \lest  L^\sharp( d_{com}^{(\nn/2)}(\mca{O}_{\rm reg}^{\vee,S_\mfr{p}} )),$$
and easily deduce $d_{BV, G}^{(n)}(\mca{O}_{\rm reg}^{\vee,S_\mfr{q}}) < d_{BV, G}^{(n)}(\mca{O}_{\rm reg}^{\vee,S_\mfr{p}} )$.
If $j_0=s-1$, then we get $q_{s-1} = p_{s-1} + p_s$. Property (P) implies $\nn < 2p_s + p_{s-1} + p_{s-1}$ and thus $\nn/2 < 2p_s + p_{s-1} + p_{s-1}$. Lemma \ref{L:key01} gives that $d_{com}^{(\nn/2)}(\mca{O}_{\rm reg}^{\vee,S_\mfr{q}}) < d_{com}^{(\nn/2)}(\mca{O}_{\rm reg}^{\vee,S_\mfr{p}} )$. Note that if $2p_s\gest \nn/2$ or $p_{s-1} \gest \nn/2$, we get 
$$L^\sharp(d_{com}^{(\nn/2)}(\mca{O}_{\rm reg}^{\vee,S_\mfr{q}})) + 2 \lest  L^\sharp( d_{com}^{(\nn/2)}(\mca{O}_{\rm reg}^{\vee,S_\mfr{p}} )).$$
On the other hand, if $2p_s < \nn/2$ and $p_{s-1}< \nn/2$, then 
$$L^\sharp( d_{com}^{(\nn/2)}(\mca{O}_{\rm reg}^{\vee,S_\mfr{p}} )) - L^\sharp(d_{com}^{(\nn/2)}(\mca{O}_{\rm reg}^{\vee,S_\mfr{q}})) = 2p_{s} + 2p_{s-1} - \nn/2 \gest 2.$$
In view of these two inequalities and \eqref{E:dBV-C}, we get $d_{BV, G}^{(n)}(\mca{O}_{\rm reg}^{\vee,S_\mfr{q}}) < d_{BV, G}^{(n)}(\mca{O}_{\rm reg}^{\vee,S_\mfr{p}} )$ as well.

The above verifies Proposition \ref{P:piSkey} and Theorem \ref{T:main2} for type $C_r$.

\subsubsection{Type $D_r$}
For $\wt{G}$ of type $D_r$, at the partition level, the inequality in Proposition \ref{P:piSkey} follows from a parallel argument as in the $B_r, C_r$ case. We omit the details.

It suffices to remark on the case concerning very even orbits. However, suppose $\mfr{p}$ or $\mfr{q}$ (possibly both) is a very even partition associated with $S_\mfr{p} \subsetneq  S_\mfr{q}$ respectively. If they are very even, then we have two orbits $\mfr{p}^{\rm I}, \mfr{p}^{\rm II}$ and $\mfr{q}^{\rm I}, \mfr{q}^{\rm II}$. However, since as partitions, we already have $d_{BV, G}^{(n)}(\mfr{q}) < d_{BV, G}^{(n)}(\mfr{p})$, it follows that for any $\flat, \sharp \in \set{\emptyset, {\rm I}, {\rm II}}$ one has (see \cite[Theorem 6.2.5]{CM93})
$$d_{BV, G}^{(n)}(\mfr{q}^\flat) < d_{BV, G}^{(n)}(\mfr{p}^\sharp)$$
as desired. Here $\mfr{p}^\emptyset$ just means that $\mfr{p}$ is not very even, and it represents the unique orbit associated with $\mfr{p}$; similarly for $\mfr{q}$.

\subsection{Proof for exceptional types} 
For a cover $\wt{G}$ of exceptional groups and $S \subseteq \Delta$,  we have verified 
Proposition \ref{P:piSkey} directly by using the data of $d_{BV,G}^{(n)}$ given in \cite[Appendix]{GLLS}. This gives Theorem \ref{T:main2} for exceptional groups. We remark that in view of Table \ref{table 2}, the condition $\nn \lest e(S)$ implies that $\Theta(\wt{M}_S^{(n)})$ is not generic, and this non-genericity is equivalent to the inequality $d_{BV, G}^{(n)}(\mca{O}_{\rm reg}^{\vee,S}) < \mca{O}_{\rm reg}^G$.

\begin{eg}
First, consider $\wt{G}=\wt{F}_4^{(n)}$, and $S\subset \Delta$ corresponds to the Levi  $A_2 \subset F_4$.  Here $S$ is not autotomous, and Theorem \ref{T:main2} holds for $\pi_S$ when $n=1, 2$.
As another example, consider $G=E_8$ and $S$ such that $\mca{O}_{\rm reg}^S=D_4 + A_2$. Then Theorem  \ref{T:main2} holds for this $\pi_S$ when $1\lest n \lest 5$.
\end{eg}


\section{Leading coefficient $c_\Theta(\mca{O})_\psi$ for theta representation}
In this section, we consider a theta representation $\Theta:=\Theta(\wt{G}^{(n)}):=\Theta(\nu)$ of $\wt{G}^{(n)}$ and study the coefficient $c_\Theta(\mca{O})_\psi$   for $\mca{O} \in {\rm WF}(\Theta)$ that arises from the Harish-Chandra character expansion of $\Theta$. Here $\psi$ is a non-trivial character of $F$. The involved measures are naturally chosen as in \cite{MW87, Var14, Pate}, and thus $c_\Theta(\mca{O})_\psi$ is equal to the dimension of certain degenerate Whittaker models of $\Theta$. Henceforth, we utilize such an identification.
A speculative formula for $c_\Theta(\mca{O})_\psi$ was given in \cite{GaTs, GLT25} for certain ``persistent" cover $\wt{G}^{(n)}$. In this section, we verify the formula in some special cases. However, we also give examples illustrating that the original formula demands a modification in general.

\begin{dfn} \label{D:ds}
Let $\wt{G}^{(n)}$ be the cover of an almost simple and simply-connected $G$ associated with $Q$ such that $Q(\alpha^\vee) =1$ for any short simple coroot $\alpha^\vee$. 
We call $\wt{G}^{(n)}$ of $\diamondsuit$-type if it is one of the following: 
\begin{enumerate}
\item[$\bullet$] $\wt{\SL}_{r+1}^{(n)}$ with $\gcd(n, r+1)=1$;
\item[$\bullet$] $\wt{\Spin}_{2r+1}^{(n)}, \wt{\Sp}_{2r}^{(n)}, \wt{\Spin}_{2r}^{(n)}$ with $n$ odd;
\item[$\bullet$] $\wt{G}_2^{(n)}$ with any $n$;
\item[$\bullet$] $\wt{F}_4^{(n)}, \wt{E}_6^{(n)}, \wt{E}_7^{(n)}$ with $2, 3 \nmid n$;
\item[$\bullet$] $\wt{E}_8^{(n)}$ with $2, 3, 5 \nmid n$.
\end{enumerate}
Such a cover is necessarily a primitive cover. 
\end{dfn}
Recall that the geometric wavefront orbit of $\Theta(\wt{G}^{(n)})$
is equal to $\mca{O}_{\rm Spr}(j_{W_\nu}^W \varepsilon)$. For $\diamondsuit$-type covers of classical groups of type $X$, we have
$$\mca{O}_{\rm Spr}(j_{W_\nu}^W \varepsilon) = \mca{O}_X^{k, n},$$
where $k=2r+1, 2r, 2r$ for $G=\Spin_{2r+1}, \Sp_{2r}, \Spin_{2r}$ respectively, and $\mca{O}_X^{k, n}$ is the $X$-collapse of the partition $\mca{O}^{k, n}:=(n^a b)$ with $k=a n + b, 0\lest b < n$. See \cite{GLT25, KOW} for more details.

Recall $\msc{X}:=Y/Y_{Q,n}$ from \eqref{D:X}. Since we have fixed $Q$, we write $\msc{X}_n:=\msc{X}$ to highlight the dependence on $n$. Let 
$$\sigma_{\msc{X}_n}: W \to {\rm Perm}(\msc{X}_n)$$
be the permutation representation of $W$ induced from the action $w(x), w \in W, x \in Y$.

\begin{thm} \label{T:cO}
Let $\wt{G}^{(n)}$ be a covering group of $\diamondsuit$-type. Assume $p$ is sufficiently large, and that the geometric wavefront orbit $\mca{O}_{\rm Spr}( j_{W_\nu}^W \varepsilon_{W_\nu} )$ of $\Theta(\wt{G}^{(n)})$ is  the regular orbit of a Levi subgroup; we further assume that $\mca{O}_X^{k, n} = \mca{O}^{k, n}$ if $X \in \set{B, C}$ and that $\mca{O}_X^{2r, n}=(n^{2a+1},1)$ if $X=D$. Then one has
$$c_\Theta(\mca{O})_\psi = \angb{ j_{W_\nu}^W \varepsilon_{W_\nu} }{ \sigma_{\msc{X}_n} \otimes \varepsilon_W }_W$$
for every $F$-rational orbit $\mca{O}$ in ${\rm WF}(\Theta(\wt{G}^{(n)}))$. 
\end{thm} 

\subsection{Method of computation} 
We explain the main ideas and steps in the proof of Theorem \ref{T:cO}. 
Let $\mca{O}$ be in ${\rm WF}(\Theta(\wt{G}^{(n)}))$ as in Theorem \ref{T:cO}. Then $\mca{O}$ is a regular orbit of a Levi subgroup $M_\mca{O} \subseteq G$, which actually only depends on $\mca{O}_{\rm Spr}( j_{W_\nu}^W \varepsilon_{W_\nu})$, i.e., the geometric closure of $\mca{O}$.
By the result \cite[Theorem 1.5]{GGS21} of Gomez--Gourevitich--Sahi, we have
$$c_\Theta(\mca{O})_\psi = \dim \Wh_{\psi_\mca{O}}({\rm Jac}_{M_\mca{O}}(\Theta)),$$
i.e., it is equal to the dimension of $\psi_\mca{O}$-Whittaker space of the Jacquet module of $\Theta$ with respect to $M_\mca{O}$. Here $\psi_\mca{O}$ is a non-degenerate character of the unipotent subgroup of the Borel subgroup of $M_\mca{O}$, defined using $\psi$ and $\mca{O}$. By the periodicity of theta representation, we get ${\rm Jac}_{M_\mca{O}}(\Theta)$ is equal to a theta representation $\Theta(\wt{M}_\mca{O}^{(n)})$ of $\wt{M}_\mca{O}^{(n)}$. Now, by result of \cite{Ga6} and \cite[Proposition 9.2]{GGK1}, we have
\begin{equation} \label{E:dWh-T}
\dim \Wh_{\psi_\mca{O}}(\Theta(\wt{M}_\mca{O}^{(n)})) = \angb{ \varepsilon_{M_\mca{O}} }{ \sigma_{\msc{X}_n} }_{W(M_\mca{O})}.
\end{equation}
Thus, the key is to verify the equality
\begin{equation} \label{E:cOkey}
\angb{\varepsilon_W\otimes j_{W_\nu}^W \varepsilon_{W_\nu} }{ \sigma_{\msc{X}_n} }_W = \angb{ \varepsilon_{M_\mca{O}} }{ \sigma_{\msc{X}_n} }_{W(M_\mca{O})}
\end{equation}
for such $\Theta(\wt{G}^{(n)})$. We compute the left hand side of \eqref{E:cOkey} by using the result of Gyoja--Nishiyama--Shimura \cite{GNS99}, and for the right hand side we use the result of Sommers \cite{Som97}. 
For $\wt{G}_2^{(n)}$, one can verify \eqref{E:cOkey} directly. For the rest of this subsection, we assume $\wt{G}^{(n)}$ is of $\diamondsuit$-type except when $G=G_2, 2|n$ or $3|n$.
These conditions on $n$ for $\diamondsuit$-type covers are imposed so that we could apply \cite{GNS99, Som97} directly. In fact, we will obtain explicit formulas for the two sides of \eqref{E:cOkey}. We elaborate this below.

First, for $\wt{G}^{(n)}$ of $\diamondsuit$-type (except when $G=G_2, 2|n$ or $3|n$) and for every $\phi \in \Irr(W)$, one has
\begin{equation} \label{E:circ1}
\angb{\phi}{ \sigma_{\msc{X}_n} }_W = \frac{\chi_\phi(1)}{ \val{W} } \cdot \tau^*(\phi, n),
\end{equation}
where $\chi_\phi$ is the character of $\phi$ and $\tau^*(\phi, n)$ is investigated in \cite{GNS99}, especially regarding the question of whether $\tau^*(\phi, n)$ distinguishes Lusztig's families of $\phi \in \Irr(W)$. Also, for classical types, the number $\chi_\phi(1)=\dim \phi$ can be computed by using the hook formula.

On the other hand, the Levi subgroup $M_\mca{O}$ is associated with a subset $J \subseteq \Delta$, and thus we write $W_J:= W(M_\mca{O})$. For $w\in W$, denote
$$d(w):=\dim (Y\otimes \R)^w.$$
Let $m_1, ..., m_{\val{J}}$ be the exponents of the Weyl group $W_J$. It follows from \cite{Som97} that for all $\wt{G}^{(n)}$ of $\diamondsuit$-type (except when $G=G_2, 2|n$ or $3|n$) we have
\begin{equation} \label{E:circ2}
\begin{aligned}
\angb{ \varepsilon_{W_J} }{ \sigma_{\msc{X}_n} }_{W_J} & = \frac{1}{\val{W_J}} \sum_{w\in W_J} \varepsilon_{W_J}(w)\cdot n^{d(w)} \\
& =  \frac{(-1)^r}{ \val{W_J} } \sum_{w\in W_J} (-n)^{d(w)} \\
& = (-1)^r \cdot \chi(\hat{\mca{B}}_{-n}^J) \\
& = \frac{ n^{r-\val{J}} }{\val{W_J}} \cdot \prod_{i=1}^{\val{J}} (n- m_j)
\end{aligned}
\end{equation}
where $\chi(\hat{\mca{B}}_{t}^J)$ is the characteristic of the affine Springer fiber $\hat{\mca{B}}_{t}^J$ attached to $t$ and $J$ as in \cite{Som97}.

\subsection{Detailed proof of Theorem \ref{T:cO}}
We first look at the two extreme cases of Theorem \ref{T:cO}: for $\wt{G}^{(n)}$ of $\diamondsuit$-type, if $n=1$ or $n\in \N$ is such that ${\rm WF}^{\rm geo}(\Theta) = \mca{O}_{\rm reg}$, then the equality $c_\Theta(\mca{O})_\psi = \angb{ j_{W_\nu}^W \varepsilon_{W_\nu} }{ \sigma_{\msc{X}_n} \otimes \varepsilon_W }_W$ holds. Indeed,  the $n=1$ case is trivial and the second case follows from \eqref{E:dWh-T}.

We verify \eqref{E:cOkey} case by case. In view of the preceding observation, we only need to discuss the case when $\mca{O}_{\rm Spr}(j_{W_\nu}^W \varepsilon)$ is nonzero and is the regular orbit of a proper Levi subgroup of $G$.
%
\subsubsection{Type $A_r$} 
Consider $\wt{G}^{(n)} = \wt{\SL}_{r+1}^{(n)}$ with $\gcd(r+1, n)=1$. Let $\lambda$ be a partition of $k \in \N_{\gest 1}$ and let $\phi(\lambda) \in \Irr(S_k)$ be the associated representation. In particular, $\phi(k)=\mbm{1}$ and $\phi(1^k) = \varepsilon$.
The dimension of $\phi(\lambda)$ is given by the hook formula (see \cite[Theorem 3.10.2]{Sag01}):
\begin{equation} \label{E:hook}
\dim \phi(\lambda) = \frac{k!}{ \prod_{(i, j) \in \lambda} h_{i, j} },
\end{equation}
where we view $\lambda$ as a Young diagram and $h_{i, j}:=h(i, j)$ denotes the hook function of $(i, j)$.

We have
$$
j_{W_\nu}^W \varepsilon_{W_\nu} = \phi(n^a b),
$$
where $r+1 = an + b, 0\lest b <n$. We get
$$\varepsilon_W \otimes j_{W_\nu}^W \varepsilon_{W_\nu} = \phi((a+1)^b, a^{n-b}).$$
It follows  from \eqref{E:circ1} that (setting $\phi(\lambda):= \phi((a+1)^b, a^{n-b})$)
$$\begin{aligned}
\angb{ \phi(\lambda) }{ \sigma_{\msc{X}_n} }_W & = \frac{\tau^*(\phi(\lambda), n)}{ \prod_{(i, j)\in \lambda} h_{i, j} } \\
& = \frac{ n^{-1} \cdot \prod_{(i, j) \in \lambda} (n + c_{i, j}) }{ \left(\prod_{i=1}^b \prod_{j=1}^{a+1} (n-i + j) \right) \cdot \left( \prod_{i=1}^{n-b} \prod_{j=1}^a (n-b - i + j) \right) } \\
& = \frac{ (n-1) (n-2) ... (n-(b-1)) }{ b! },
\end{aligned}
$$
where the second equality above follows from \cite[Proposition 3.1]{GNS99} with $c_{i, j}:=j-i$ denoting the content of the $(i, j)$ entry of $\lambda$.

On the other hand, $M_\mca{O}$ is associated with the partition $(n^a, b)$, and thus is of type $$a A_{n-1} + A_{b-1}.$$
It then follows from \eqref{E:circ2} that
$$\angb{ \varepsilon_{W_J} }{ \sigma_{\msc{X}_n} }_{W_J}  =  \frac{n^{r-a(n-1)-(b-1)}}{ (n!)^a \cdot b! } \cdot \left( (n-1)!)^a  \right) \cdot \prod_{i=1}^{b-1} (n-i),$$
which is easily seen to be equal to $\angb{ \phi(\lambda) }{ \sigma_{\msc{X}_n} }_W$ above. This completes the proof of Theorem \ref{T:cO} for type $A_r$.

\subsubsection{Type $C_r$} We consider $\wt{G}^{(n)} = \wt{\Sp}_{2r}^{(n)}$ with $n=2m+1$ odd.
Since we need $\mca{O}_{\rm Spr}(j_{W_\nu}^W \varepsilon)$ to be the regular orbit of a Levi subgroup and also $\mca{O}_C^{2r, n} = \mca{O}^{2r, n}$, it follows from \cite{GLT25} that we must have
$$r=an + b \text{ with } a\gest 1, 0\lest b \lest \frac{n-1}{2},$$
and in this case
$$j_{W_\nu}^W \varepsilon_{W_\nu} = \phi((m^a, b); (m+1)^a), \quad \mca{O}_{\rm Spr}(j_{W_\nu}^W \varepsilon) = (n^{2a}, 2b) = \mca{O}^{2r, n}.$$
This gives 
$$\varepsilon_W \otimes j_{W_\nu}^W \varepsilon_{W_\nu}  = \phi(a^{m+1}; ((a+1)^b, a^{m-b}) )=:\phi(\lambda; \mu).$$
The hook formula for  $W(C_r)$ gives that
$$\dim \phi(\lambda; \mu) = \frac{r!}{\prod_{x'\in \lambda} h(x') \cdot \prod_{x'' \in \mu} h(x'')},$$
where $h(-)$ is the hook function mentioned above.

We get from \cite{GNS99} that
$$\begin{aligned}
\angb{ \phi(\lambda;\mu) }{ \sigma_{\msc{X}_n} }_W & = \frac{r!}{\val{W}} \cdot \frac{\prod_{x'\in \lambda}(n+ 2c(x') + 1) \cdot  \prod_{x''\in \mu}(n+ 2c(x'') + 1) }{\prod_{x'\in \lambda} h(x') \cdot \prod_{x'' \in \mu} h(x'')} \\
& = \frac{ m(m-1) (m-2) ... (m-(b-1)) }{ b! }.
\end{aligned}
$$
On the other hand, in this case $M_\mca{O}$ is of Cartan type
$$a\cdot A_{n-1} + C_b.$$
It follows from \cite{Som97} that
$$\angb{ \varepsilon_{W_J} }{ \sigma_{\msc{X}_n} }_{W_J}  =  \frac{n^{r-(an + b -a)}}{ (n!)^a \cdot 2^b \cdot b! } \cdot \left( (n-1)!\right)^a   \cdot \prod_{i=1}^{b} (n-2i+1),$$
which is equal to $\angb{ \phi(\lambda;\mu) }{ \sigma_{\msc{X}_n} }_W$, by noting that $n=2m+1$. Thus, Theorem \ref{T:cO} for type $C_r$ is proved.

\begin{rmk}
The technical condition that $\mca{O}_C^{2r, n} = \mca{O}^{2r, n}$ (besides the assumption that the former orbit is a regular orbit of a Levi subgroup) in Theorem \ref{T:cO} is indispensable. Indeed, if we consider $r=2m$ and $n=2m+1 = r+1$, then in this case
$$j_{W_\nu}^W \varepsilon = \phi(m;m) \text{ and } \mca{O}_{\rm Spr} (j_{W_\nu}^W\varepsilon)=(r,r).$$
We have $\varepsilon_W \otimes j_{W_\nu}^W \varepsilon = \phi(1^m; 1^m)$ and following the same method of computation, we get
$$\angb{ \phi(1^m;1^m) }{ \sigma_{\msc{X}_n} }_W = m+1.$$
On the other hand, 
$$c_\Theta(\mca{O})_\psi = \angb{ \varepsilon_{W_J} }{ \sigma_{\msc{X}_n} }_{W_J} = n = 2m+1.$$
This indicates a necessity of modifying the original speculation in \cite[Conjecture 4.1 (4.4)]{GLT25}.
\end{rmk}

\subsubsection{Type $B_r$}
We consider $\wt{\Spin}_{2r+1}^{(n)}$ with $n=2m+1$ odd. It follows from \cite{GLT25} that we only need to consider the case when 
$$r=na + b \text{ with } 0\lest b \lest m.$$
In this case, we get
$$j_{W_\nu}^W \varepsilon = \phi((m^a, b); (m+1)^a) \text{ and } \mca{O}_{\rm Spr} (j_{W_\nu}^W \varepsilon) = (n^{2a}, 2b+1).$$
We also have
$$\varepsilon_{W} \otimes j_{W_\nu}^W \varepsilon_{W_\nu} = \phi(a^{m+1}; ((a+1)^b, a^{m-b}) )=:\phi(\lambda; \mu).$$
The computation is then exactly the same as the type $C_r$ case, and we have
$$\angb{ \phi(\lambda;\mu) }{ \sigma_{\msc{X}_n} }_W= \angb{ \varepsilon_{W_J} }{ \sigma_{\msc{X}_n} }_{W_J} = \frac{ m(m-1) (m-2) ... (m-(b-1)) }{ b! }.$$
This verifies Theorem \ref{T:cO} for the type $B_r$ case.

\subsubsection{Type $D_r$}
We consider $\wt{\Spin}_{2r}^{(n)}$ with $n=2m+1$. If ${\rm WF}^{\rm geo}(\Theta)$ is the regular orbit of a Levi subgroup, we need
$$r-1=na + b, 0\lest b \lest m,$$
and in this case
$$j_{W_\nu}^W \varepsilon = \phi(((m+1)^a, b+1); m^a) \text{ and } \mca{O}_{\rm Spr} (j_{W_\nu}^W \varepsilon) = (n^{2a}, 2b+1, 1).$$
We get
$$\phi(\lambda; \mu):= \varepsilon_W\otimes j_{W_\nu}^W \varepsilon = \phi(a^m; ((a+1)^{b+1}, a^{m-b})),$$
and
$$\dim \phi(\lambda;\mu) = \frac{r!}{ (am)! \cdot (am + a + b +1)! } \cdot \frac{1}{ \prod_{x'\in \lambda} h(x') \cdot \prod_{x'' \in \mu} h(x'') }.$$
By using \eqref{E:circ1} and \cite[Propositition 3.6(1)]{GNS99}, a long but straightforward computation gives that
$$\angb{ \phi(\lambda;\mu) }{  \sigma_{\msc{X}_n} }_W = \frac{(m+1)m(m-1) ... (m+1-b)}{(b+1)!}.$$
On the other hand,
$M_\mca{O}$ is of Cartan type 
$$a\cdot A_{n-1} + D_{b+1}.$$
It then follows from \eqref{E:circ2} that
$$c_\Theta(\mca{O})_\psi = \angb{ \varepsilon_{W_J} }{ \sigma_{\msc{X}_n} }_{W_J} = \frac{ m(m-1) (m-2) ... (m-(b-1)) }{ (b+1)! }\cdot (n-b).$$
Thus, we see that if $b=m$, or equivalently $\mca{O}=(n^{2a+1},1)$, then $c_\Theta(\mca{O})_\psi = \angb{ \phi(\lambda;\mu) }{  \sigma_{\msc{X}_n} }_W$ as desired. This verifies Theorem \ref{T:cO} for the type $D_r$ case.

We remark also that if $b\ne m$, then the equality does not hold.

\subsubsection{Type $G_2$}
For $\wt{G}_2^{(n)}, n\in \N_{\gest 1}$, there are exactly two nontrivial cases to consider, enforced by the assumption in Theorem \ref{T:cO}: $n=2$ or 3.

In the first case, for $n=2$ we get $j_{W_\nu}^W \varepsilon = \phi_{2,2}$ and $\mca{O}_{\rm Spr} (j_{W_\nu}^W \varepsilon)=\tilde{A}_1$.
By using the character of $\sigma_{\msc{X}_n}$ computed in \cite{Ga6}, we get
$$\angb{\phi_{2,2}}{ \varepsilon\otimes \sigma_{\msc{X}_2} }=\angb{ \varepsilon_{W_J} }{ \sigma_{\msc{X}_n} }_{W_J}=1.$$
In the second case, if $n=3$, then 
$j_{W_\nu}^W \varepsilon = \phi_{1,3}'' \text{ and } \mca{O}_{\rm Spr} (j_{W_\nu}^W \varepsilon)=A_1$.
Again, using the character of $\sigma_{\msc{X}_n}$ given in \cite{Ga6}, we get
$$\angb{\phi_{1,3}'' }{ \varepsilon\otimes \sigma_{\msc{X}_2} }=\angb{ \varepsilon_{W_J} }{ \sigma_{\msc{X}_n} }_{W_J}=1.$$
This verifies Theorem \ref{T:cO} for $G_2$.

\subsubsection{Type $E_6$}
A cover $\wt{E}_6^{(n)}$ satisfies the assumption in Theorem \ref{T:cO}, in the non-trivial cases, only for $n=5$. In this case (i.e., for $\wt{E}_6^{(5)}$), one has 
$j_{W_\nu}^W \varepsilon = \phi_{60,5} \text{ and } \mca{O}_{\rm Spr} (j_{W_\nu}^W \varepsilon)=A_4 + A_1$. Again, using \eqref{E:circ1} and the result of \cite{GNS99}, we get
$$\angb{\phi_{60,5} }{ \varepsilon\otimes \sigma_{\msc{X}_5} }=  \angb{\phi_{60,11} }{ \sigma_{\msc{X}_5} } = 2.$$
On the other hand, we get from \eqref{E:circ2} that $\angb{ \varepsilon_{W_J} }{ \sigma_{\msc{X}_5} }_{W_J}=2$ as well. This verifies Theorem \ref{T:cO} for $E_6$.

\subsubsection{Type $E_7$}
For $\wt{E}_7^{(n)}$, there are exactly two non-trivial cases to consider: $n=5, 7$.

If $n=5$, then $j_{W_\nu}^W \varepsilon = \phi_{210,10} \text{ and } \mca{O}_{\rm Spr} (j_{W_\nu}^W \varepsilon)=A_4 + A_2$. Using \eqref{E:circ1} and the result of \cite{GNS99}, we get 
$$\angb{\phi_{210,10} }{ \varepsilon\otimes \sigma_{\msc{X}_5} }=  \angb{\phi_{210,13} }{ \sigma_{\msc{X}_5} } = 2.$$
On the other hand, we get from \eqref{E:circ2} that $\angb{ \varepsilon_{W_J} }{ \sigma_{\msc{X}_5} }_{W_J}=2$ as well.

If $n=7$, then $j_{W_\nu}^W \varepsilon = \phi_{105,6} \text{ and } \mca{O}_{\rm Spr} (j_{W_\nu}^W \varepsilon)=A_6$. Again, using \eqref{E:circ1} and the result of \cite{GNS99}, we get 
$$\angb{\phi_{105,6} }{ \varepsilon\otimes \sigma_{\msc{X}_7} }=  \angb{\phi_{105,21} }{ \sigma_{\msc{X}_5} } = 1.$$ 
On the other hand, we get from \eqref{E:circ2} that $\angb{ \varepsilon_{W_J} }{ \sigma_{\msc{X}_7} }_{W_J}=1$. This verifies Theorem \ref{T:cO} for $E_7$.

\subsubsection{Type $E_8$}
For $\wt{E}_8^{(n)}$ we only need to consider $n=7$, and in this case $j_{W_\nu}^W \varepsilon = \phi_{2835,14}$ and $\mca{O}_{\rm Spr} (j_{W_\nu}^W \varepsilon)=A_6 + A_1.$
It follows from \eqref{E:circ1} and the result of \cite{GNS99} that 
$$\angb{\phi_{2835,14} }{ \varepsilon\otimes \sigma_{\msc{X}_7} }=  \angb{\phi_{2835,22} }{ \sigma_{\msc{X}_7} } = 3.$$ 
On the other hand, we get from \eqref{E:circ2} that $\angb{ \varepsilon_{W_J} }{ \sigma_{\msc{X}_7} }_{W_J}=3$. This verifies Theorem \ref{T:cO} for $E_8$.

\begin{rmk}
There is no $\diamondsuit$-type cover $\wt{F}_4^{(n)}, n\gest 2$ such that $\WF^{\rm geo}(\Theta)$ is the regular orbit of a proper Levi subgroup of $F_4$, as required in  Theorem \ref{T:cO}.
\end{rmk}


\end{document}